\documentclass[12pt,a4paper]{amsart}
\usepackage{amsfonts}
\usepackage{amsthm}
\usepackage{amsmath}
\usepackage{amscd}
\usepackage[latin2]{inputenc}
\usepackage{t1enc}
\usepackage[mathscr]{eucal}
\usepackage{indentfirst}
\usepackage{graphicx}
\usepackage{graphics}
\usepackage{pict2e}
\usepackage{epic}
\numberwithin{equation}{section}
\usepackage[margin=2.9cm]{geometry}
\usepackage{epstopdf}

\usepackage{xcolor}
\usepackage{amssymb}
\usepackage{graphicx}
\usepackage{tikz}
\usetikzlibrary{matrix,arrows,decorations.pathmorphing}
\usepackage{tikz-cd}

\usepackage{pgfplots}
\pgfplotsset{compat=newest}

\usepackage{fourier}

\usepackage[shortlabels]{enumitem}

\usepackage{extarrows}

\theoremstyle{plain}
\newtheorem{Th}{Theorem}[section]
\newtheorem{Lemma}[Th]{Lemma}
\newtheorem{Cor}[Th]{Corollary}
\newtheorem{Prop}[Th]{Proposition}

 \theoremstyle{definition}
\newtheorem{Def}[Th]{Definition}

\newtheorem{Rem}[Th]{Remark}
\newtheorem{?}[Th]{Problem}

\begin{document}

\title{New moduli spaces of one-dimensional sheaves on $\mathbb{P}^3$}
\author{Dapeng Mu}

\begin{abstract}
In \cite{altavilla2019euler}, a one-dimensional family of ``Euler'' stability conditions on $\mathbb{P}^n$ are conjectured to converge to Gieseker stability for coherent sheaves. Here, we focus on ${\mathbb P}^3$, first identifying Euler stability conditions with double-tilt stability conditions, and then we consider moduli of one-dimensional sheaves, proving some asymptotic results, boundedness for walls, and then explicitly computing walls and wall-crossings for sheaves supported on rational curves of degrees $3$ and $4$.
\end{abstract}

\maketitle
\tableofcontents

\section{Introduction}

In \cite{altavilla2019euler}, a one-parameter family of Euler stability conditions on the derived category of coherent sheaves on ${\mathbb P}^n$ was introduced. The heart, denoted by $\mathscr{A}_t$, consists of complexes
$$\mathscr{A}_t:=\left\{ \left[ \mathscr{O}_{\mathbb{P}^n}^{a_{-n}}(-k-n)\to \cdots \to \mathscr{O}_{\mathbb{P}^n}^{a_{-1}}(-k-1)\to \mathscr{O}_{\mathbb{P}^n}^{a_{0}}(-k) \right] : a_{0},...,a_{-n} \in \mathbb{Z}_{\geq 0}, t\in \mathbb{R} \right\}$$
in which $k$ is the roundup of $t$ to the closest integer. With the observation that the Hilbert polynomial $P_t(\mathscr{O}_{\mathbb{P}^n})=\chi(\mathscr{O}_{\mathbb{P}^n}(t))$ of $\mathscr{O}_{\mathbb{P}^n}$ has simple roots $t = -1, -2, ..., -n$. We define the central charge using the Euler characteristic as follows. For any object $E\in \mathscr{A}_t$, let $\chi_t(E)$ be the Euler characteristic of the twisted object $E\otimes \mathscr{O}_{\mathbb{P}^n}(t)$, i.e. 
$$\chi_t(E):=\int_{\mathbb{P}^3} ch(E) \cdot ch(\mathscr{O}_{\mathbb{P}^3}(t))\cdot Td(\mathbb{P}^3)$$ It's the Hilbert polynomial of $E$, and denote $\chi'_t(E)$ as the derivative of $\chi_t(E)$ with respect to $t$. 
Define the central charge as $Z_t:=\chi'_t + i\cdot \chi_t$.
It was proved (in \cite{altavilla2019euler}, also in Section $3$) that the pair $\sigma_t=(\mathscr{A}_t, Z_t)$ is a stability condition on $\mathbb{P}^n$. We conjectured that the moduli space of Euler-stable complexes for large $t$ coincides with the moduli of Gieseker-stable sheaves. Here, we focus on the conjecture for objects in ${\mathcal D}^b({\mathbb P}^3)$ of class $v=(0,0,ch_2>0,ch_3)$, i.e. the class of a one-dimensional coherent sheaf.

On $\mathbb{P}^3$, there is a construction of Bridgeland stability conditions by the double-tilting approach (\cite{Bayer_2013},\cite{Macr__2014}). Denote this stability condition by $\sigma_{\alpha,\beta, s}=(\mathscr{A}^{\alpha, \beta}, Z_{\alpha, \beta, s})$. The heart $\mathscr{A}^{\alpha, \beta}$ is obtained from tilting $Coh(\mathbb{P}^3)$ twice, and the central charge is the following function with twisted Chern characters (Def \ref{twisted chern character}),

$$Z_{\alpha, \beta, s}:= -ch^{\beta}_3+(s+\frac{1}{6})\alpha^2H^2ch^{\beta}_1+i\cdot(Hch^{\beta}_2-\frac{\alpha^2}{2}H^3ch^{\beta}_0)$$
For $\alpha=\frac{1}{\sqrt{3}}$, $\beta=-t-2$, $s=\frac{1}{3}$, it's straightforward to check that the central charge becomes $Z^{\mathscr{B}_t}_t := Z_{\frac{1}{\sqrt{3}}, -t-2, \frac{1}{3}}=-\chi_t + i \cdot \chi'_t$. Denote the corresponding heart $\mathscr{A}^{\frac{1}{\sqrt{3}}, -t-2}$ by $\mathscr{B}_t$, and this one-dimensional stability condition by $\sigma^{\mathscr{B}_t}_t:=(\mathscr{B}_t, Z^{\mathscr{B}_t}_t)$. 
We will show that the two stability conditions $\sigma_t$ (Euler) and $\sigma^{\mathscr{B}_t}$ (double-tilt) are essentially the same in the sense that $\sigma_t$ is a tilt of $\sigma^{\mathscr{B}_t}_t$.

To study the asymptotic behavior of objects in $\mathscr{A}_t$, we make use of the finiteness of $\mathscr{A}_t$ and the better-behaved walls of $\mathscr{B}_t$. We extend the stability condition  $\sigma^{\mathscr{B}_t}_t$ to the "$(t,u)$" upper half plane by modifying the central charge to be $Z_{t,u}=-\chi_t+\frac{u^2}{2}\chi''_t+i\cdot \chi'_t$ ($\chi''_t$ denotes the second derivative of $\chi_t$, and the $"u"$ parameter is essentially the $"s"$ parameter in the double-tilt stability). The pair $\sigma_{t,u}=(\mathscr{B}_t, Z_{t,u})$ is a stability condition, analogous to the construction of tilt stability on surfaces. We start with studying the asymptotic behavior of a fixed object $E$ for $u>>0$.

\begin{Th}
For any fixed $t$, and any object $E \in \mathscr{B}_t$ wtih class $v(E)=(0,0,ch_2>0,ch_3)$, there exists a $u_E>0$ such that for all $u>u_E$, $E$ is $\sigma_{t,u}$-stable if and only if $E$ is a Gieseker stable sheaf. 
\end{Th}

Next, we consider the boundedness of the walls for any given class $v\in K_{num}(\mathbb{P}^3)$ of a one-dimensional sheaf. There are bounded and unbounded parts of potential walls. 
We find that the bounded walls satisfy $|t-2-ch_3/ch_2|\leq ch_2+2\sqrt{2ch_2}$. The unbounded walls remain somewhat of a mystery. We expect that unbounded potential walls are not actual walls.

To support our conjecture, we explicitly compute the wall-crossings for the class $v=(0,0,3,-5)$, the class of the structure sheaf of a twisted cubic. There was previous work on the class of ideal sheaves of space (rational and elliptic) curves in \cite{MR3803142} \cite{MR3597844} \cite{gallardo1609families} \cite{2015arXiv150904608S}, and it was shown that the last moduli space was the Hilbert scheme of curves. Our main result is the following:

\begin{Th}
For the class $v=(0,0,3,-5)\in K_{num}{(\mathbb{P}^3)}$, there are two walls in $\mathscr{A}_1$ ($t\in (0,1]$) defined by the short exact sequences:

$$W_1 \ (t=0.35): \quad 0\to \mathscr{O}_{\mathbb{P}^3}\to E \to Q[1]\to 0$$

$$W_2 \ (t=0.72): \quad 0\to \mathscr{O}_{\Lambda}\to E \to \mathscr{F}_1\to 0$$

In the first sequence, $Q$ is a coherent sheaf with two generic representatives: $\mathscr{I}_C$ and $\mathscr{F}$, where $\mathscr{F}$ is a sheaf containing torsion. 

In the second sequence, $\Lambda \subset {\mathbb P}^3$ is a plane and $\mathscr{F}_1$ is a complex fitting in the short exact sequence: $0 \to \mathscr{F}_1\to \mathbb{C}_P\to \mathscr{O}_{\Lambda}(-3)[2] \to 0$ in $\mathscr{A}_1$, where $\mathbb{C}_P$ is the skyscraper sheaf of a point in $\Lambda$.

The moduli spaces in the three corresponding chambers are as follows:

\begin{enumerate}
    \item $t\in (0,0.35)$. The moduli space is empty.
    
    \item $t\in (0.35,0.72)$. The moduli space consists of two components: 

    \begin{enumerate}
      \item 
      $K_{(2,3)}$: a smooth $12-$dimensional Kronecker moduli space. 
      \item 
      $M_{\mathscr{F}}$: a $\mathbb{P}^3$ bundle over a closed $5-$dimensional smooth flag variety $H\subset K_{(2,3)}$.
    \end{enumerate}
    \item $t\in (0.72,1]$. $M_{\mathscr{F}}$ disappears. $K_{(2,3)}$ is blown up along $H$, denoted by $\mathbf{B}:=Bl_H(K_{(2,3)})$. A new component $\mathbf{P}$ comes in, glued to $\mathbf{B}$ along the exceptional divisor of $\mathbf{B}$. $\mathbf{P}$ is the relative Simpson scheme over $\mathbb{P}^{3\vee}$, fibered by the scheme $\mathcal{M}^{3t+1}_{\mathbb{P}^2}$, the moduli space of Gieseker semistable sheaves with Hilbert polynomial $P_t=3t+1$ on $\mathbb{P}^2$. This is the Gieseker moduli space of class $v$ on $\mathbb{P}^3$.
\end{enumerate}
\end{Th}

In \cite{MR2683291}, Maican proved that the functor $F: \mathscr{F} \to \mathcal{E}xt^{n-1}(\mathscr{F}, \omega_{\mathbb{P}^n})$ preserves Gieseker stability for sheaves on $\mathbb{P}^n$. It was generalized to Bridgeland stable complexes on $\mathbb{P}^2$ in \cite{MR3615584}. We prove the analogous result on $\mathbb{P}^3$, and a similar duality result holds on $\mathbb{P}^n$ ($n\in \mathbb{Z}_{\geq 0}$) as well.
For a class $v\in K_{num}(\mathbb{P}^3)$ and an object $E\in \mathscr{A}_t$ with this class, define its dual as $E^D:=R\mathscr{H}om(E,\omega_{\mathbb{P}^3})[2]$.

\begin{Th}
$E\in \mathscr{A}_t$ is (semi)stable with phase $\phi\in (0,1)$ 
if and only if $E^D[1]$ is 
(semi)stable in $\mathscr{A}_{-t}$ 
with phase $1-\phi$ for all $t\in \mathbb{R}$.
\end{Th}

In the end, we show an example in which an actual wall is built up from two pieces of distinct numerical walls. Unlike surfaces, where walls are nested (\cite{MR3217637}), walls on a threefold may intersect (\cite{2015arXiv150904608S}, \cite{jardim2019walls}). Let $C \subset \mathbb{P}^3$ be a rational quartic curve contained in a quadric surface $Q \subset \mathbb{P}^3$. We propose that its actual wall in the $(t-u)-$plane is the outermost parts of the numerical walls defined by these two sequences  

$$0\to \mathscr{O}_{\mathbb{P}^3}\to \mathscr{O}_{C} \to \mathscr{I}_{C}[1]\to 0 \quad \text{and} \quad 0\to \mathscr{O}_{Q}\to \mathscr{O}_{C} \to \mathscr{I}_{C/Q}[1] \to 0$$

The paper is organized as follows. In section 2, we give a brief review of Bridgeland stability conditions. In section 3, we introduce the one-dimensional Euler stability condition and show its relation with the double-tilt stability on $\mathbb{P}^3$. Section 4 is evidence for our main conjecture that for any one-dimensional class $v\in K_{num}(\mathbb{P}^3)$, Euler stable complexes are Gieseker stable sheaves for all large $t$. We will show the walls for one-dimensional classes and prove some asymptotic results for sheaves and complexes. In section 5, we show the duality results for stable objects in $\mathscr{A}_t$ and $\mathscr{B}_t$. In sections 6 and 7, we will focus on the fixed class $v=(0,0,3,-5)$, finding the walls for the class and describing the moduli spaces in each chamber. Finally, in section 8, we show an example that an actual wall is built up from distinct numerical walls.

\begin{Rem}
There is related recent work on walls and the asymptotic stability for threefolds of Picard rank $1$ in \cite{jardim2019walls}, \cite{pretti2021zero}, \cite{jardim2021vertc}.
\end{Rem}

\textbf{Acknowledgement:} I would like to express my gratitude to my advisor Professor Dr. Aaron Bertram, for his support, guidance, and patience during my time as his student at the University of Utah. I would like to thank Professor Emanuele Macr{\`i} for his kind suggestions, and my friends Huachen Chen, Ziwen Zhu for helpful conversations. I am also grateful to the department of mathematics at the University of Utah for their help and support. The author is currently supported by FAPESP grant number 2020/03499-0 which is part of the FAPESP Thematic Project 2018/21391-1.

\section{Background on Bridgeland stability conditions}

\subsection{Stability conditions}
In this section, we recall some definitions of Bridgeland Stability Conditions and the constructions on a smooth threefold $X$ over $\mathbb{C}$. We refer to the following articles for more details. 
\cite{Happel_1996} for tilting theory,
\cite{Bridgeland_2007}, \cite{Bridgeland_2008}, \cite{Macr__2017}, \cite{MR3221296} for Bridgeland stability conditions,
\cite{Macr__2014}, \cite{Bayer_2013}, \cite{MR3573975} for stability on a threefold.

\begin{Def}
The heart of a bounded t-structure on $D^b(X)$ is a full additive subcategory $\mathscr{A}\subset D^b(X)$ satisfying the following conditions, 
\begin{enumerate}[(a)]
    \item For integers $i>j$ and $A,B \in \mathscr{A}$, we have $Hom(A[i], B[j])=0$.
    \item For all $E\in D^b(X)$, there exists integers $k_1>...>k_m$ and objects $E_i\in D^b(X)$, $A_i\in \mathscr{A}$ for $i=1,2,...,m$ and a diagram consisting of distinguished triangles:
    
    \begin{tikzcd}
    0=E_0 \arrow[r]& E_1\arrow[r]\arrow[d] & E_2\arrow[r]\arrow[d]&...\arrow[r]& E_{m-1}\arrow[r]\arrow[d] & E_m\arrow[d]=E\\
    {}&A_1[k_1]\arrow[ul,dotted]& A_2[k_2]\arrow[ul, dotted]&{}&A_{m-1}[k_{m-1}]\arrow[ul,dotted]&A_{m}[k_m]\arrow[ul,dotted]\\
    \end{tikzcd}
\end{enumerate}
\end{Def}

The heart of a bounded t-structure is indeed an abelian category. The proof can be found in \cite{MR751966} and \cite{Macr__2017}.
  
Let $\mathscr{A}$ be an abelian category. $K_0(\mathscr{A})$, $K_{num}(\mathscr{A})$ denote the K-group and numerical K-group of $\mathscr{A}$ respectively.

\begin{Def}
 A linear function $Z: K_0(\mathscr{A})\to \mathbb{C}$ is called a central charge (or a stability function) if for every $E\in \mathscr{A}$, $\mathscr{I}m(Z(E))\geq 0$, and if $\mathscr{I}m(Z(E))= 0$ then $\mathscr{R}e(Z(E))<0$.
\end{Def}

With the heart $\mathscr{A}$ and the central charge $Z$, we can define the (semi-)stability of an object in $\mathscr{A}$ as follows:

\begin{Def}
Let $Z: K_0(\mathscr{A})\to \mathbb{C}$ be a central charge. Define the slope as $\mu:=-{\mathscr{R}e}/{\mathscr{I}m}$. A non-zero object $E\in \mathscr{A}$ is called (semi-)stable if for all non-trivial subobject $F\subset E \in \mathscr{A}$, $\mu(F)(\leq) < \mu(E)$ holds.
\end{Def}

Fix a finite rank lattice $\Lambda$ with a fixed norm $||\cdot||$ on it. Let $v$ be a surjective linear map, $ v: K_0(\mathscr{A})\twoheadrightarrow \Lambda$. For a Bridgeland stability condition, the central charge $Z$ is required to factor through a finite rank lattice, i.e. $Z: K_0(\mathscr{A})\overset{v}{\to}\Lambda \to \mathbb{C}$. The lattice is usually chosen as the numerical classes, i.e. $\Lambda=(H^nch_0, H^{n-1}ch_1, ... , H^{n-l}ch_l )\subset K_{num}(X)$. 

Next, we define the Harder-Narasimhan property and the support property. A pair $\sigma =(\mathscr{A}, Z)$ satisfying the Harder-Narasimhan property is sometimes called a pre-stability. It is a stability condition if it also satisfies the support property.

\begin{Def}
A Bridgeland Stability Condition on $X$ is a pair $\sigma =(\mathscr{A}, Z)$ consisting of a heart of a bounded t-structure $\mathscr{A}\subset D^b(X)$, a central charge $Z: \mathscr{A}\to \Lambda \to  \mathbb{C}$, and two more properties as follows, 

\begin{enumerate}[(a)]
    
    \item Every object $E\in \mathscr{A}$ satisfies the Harder-Harasimhan property:

    $E$ has a finite filtration $0=E_0\subset E_1,\subset ..., \subset E_n=E$, in $\mathscr{A}$ such that the quotients $F_i:=E_i/E_{i-1}$ are semi-stable and their slopes are strictly decreasing, i.e. $\mu(v(F_1))>\mu(v(F_2))>...>\mu(v(F_n))$.
    
    \item $\sigma$ satisfies the support property: 
    
    For a fixed norm $||\cdot ||$ on $\Lambda$, there exists a $C>0$ such that for all semi-stable object $E\in \mathscr{A}$, $||v(E)||\leq C|Z(E)|$ holds.

\end{enumerate}
\end{Def}

The support property can also be defined by a bilinear form $Q$ on the lattice (\cite{MR3573975}, \cite{MR3905133}):

\begin{Def}
A pre-stability condition $\sigma=(\mathscr{A},Z)$ satisfies the support propery if there is a bilinear form $Q$ on $\Lambda\otimes \mathbb{R}$ such that 
\begin{enumerate}[(a)]
    \item All $\sigma$- semistable objects $E\in \mathscr{A}$ satisfy the inequality $Q(v(E),v(E))\geq 0$.

    \item All non-zero vectors $v\in \Lambda\otimes \mathbb{R}$ with $Z(v)=0$ satisfy $Q(v,v)<0$.
\end{enumerate}
\end{Def}

The Bridgeland stability condition can be equivalently  defined on a triangulated category as well (\cite{Bridgeland_2007}, \cite{Macr__2017}).

\begin{Def}
A slicing $\mathscr{P}$ of $D^b(X)$ is a collection of subcategories $\mathscr{P}(\phi)\subset D^b(X)$ for every $\phi\in\mathbb{R} $, such that

\begin{enumerate}[(a)]
    \item $\mathscr{P}(\phi)[1]=\mathscr{P}(\phi+1)$,
    
    \item For $A\in \mathscr{P}(\phi_1)$, and $B\in \mathscr{P}(\phi_2)$, if $\phi_1>\phi_2$ then $Hom(A,B)=0$.
    
    \item For any $E\in D^b(X)$, there are real numbers $\phi_1>...>\phi_m$ and distinguished triangles in $D^b(X)$
    
    \begin{tikzcd}
    0=E_0 \arrow[r]& E_1\arrow[r]\arrow[d] & E_2\arrow[r]\arrow[d]&...\arrow[r]& E_{m-1}\arrow[r]\arrow[d] & E_m\arrow[d]=E\\
    {}&A_1\arrow[ul,dotted]& A_2\arrow[ul, dotted]&{}&A_{m-1}\arrow[ul,dotted]&A_{m}\arrow[ul,dotted]\\
    \end{tikzcd}
    
    such that $A_i \in \mathscr{P}(\phi_i)$.
    
\end{enumerate}
\end{Def}

The last property is the Harder-Narasimhan filtration of $E$ in $D^b(X)$.

\begin{Def}
A Bridgeland Stability condition on $D^b(X)$ is a pair $\sigma=(\mathscr{P},Z)$ where $\mathscr{P}$ is a slicing, $Z: \Lambda \to \mathbb{C}$ is a linear map (the central charge), satisfying the following two properties:

\begin{enumerate}
    \item For any non-zero $E\in \mathscr{P}(\phi)$, we have 
    
    $$Z(v(E))\in \mathbb{R}_{>0}\cdot e^{\sqrt{-1}\pi\phi}$$
    \item (support property) There exists a constant $C>0$, such that $||v(E)||\leq C|Z(E)|$ for any $0\neq E \in \mathscr{P}(\phi), \phi\in \mathbb{R}$.
\end{enumerate}
\end{Def}

Let $\mathscr{A}:=\mathscr{P}((0,1])$ be the extension closure of slices $\left\{\mathscr{P}(\phi):\phi\in (0,1]\right\}$. The category $\mathscr{A}$ turns out to be the heart of a bounded t-structure. Moreover, the two stability conditions, $\sigma_1=(\mathscr{P}, Z)$ and $\sigma_2=(\mathscr{A},Z)$ are equivalent (\cite{Bridgeland_2007} Prop $5.3$).

\subsection{Construction of stability conditions on a smooth threefold.}

In this subsection, we recall a construction of stability conditions, the tilting approach, and some known results on a smooth variety up to dimension $3$.

We start by defining the twisted Chern character.

\begin{Def}{\label{twisted chern character}}
For any $B\in NS_{\mathbb{R}}(\mathbb{P}^3)$, define the twisted Chern characters as $ch^B:=ch\cdot e^{-B}$.
\end{Def}

By expanding $ch\cdot e^{-B}$, we have the first few twisted Chern characters as:

\begin{align*}
ch^B_0=&ch_0\\
ch^B_1=&ch_1-B\cdot ch_0\\
ch^B_2=&ch_2-B\cdot ch_1+\frac{B^2}{2}ch_0\\
ch^B_3=&ch_3-B\cdot ch_2+\frac{B^2}{2}ch_1-\frac{B^3}{6}ch_0\\
\cdots
\end{align*}

Stability conditions on a smooth curves $C$ is essentially the pair $\sigma=(Coh(C), Z=-deg+i\cdot rk)$ (\cite{Bridgeland_2007} \cite{Macr__2007}), which looks like the Mumford stability for vector bundles on curves. 

For a smooth surface $X$, $Coh(X)$ can no longer be used as the heart to define stability conditions. In fact, for dim$X\geq 2$, the pair $(Coh(X),Z)$ is not a stability condition for any central charge $Z$. A proof can be found in \cite{MR2541209} (Lemma $2.7$). To obtain the correct heart, we tilt the category $Coh(X)$.

For the rest of this subsection, we will focus on the construction on $\mathbb{P}^3$. The similar approach works for some other smooth varieties as well. (\cite{Bridgeland_2008}, \cite{MR2998828}, \cite{MR2852118} for surfaces, \cite{MR3262194} for quadric threefolds, \cite{MR3370123}, \cite{MR3454685}, \cite{MR3573975} for abelian threefolds, \cite{MR3908763},\cite{MR3743105} for Calabi-Yau threefolds, \cite{Bayer_2013} for a conjectural construction on all smooth varieties.)

We start with the first tilt of $Coh(\mathbb{P}^3)$. 

Let $H$ be the ample class $\mathscr{O}_{\mathbb{P}^3}(1)$ on $\mathbb{P}^3$. $B\in NS_{\mathbb{R}}(\mathbb{P}^3)$ is a multiple of $H$, $B=\beta H$. For simplicity, we write $ch^{\beta}$ instead of $ch^{\beta H}$.

Define the slope function on $Coh(\mathbb{P}^3)$ as $$\mu_{\beta}:=\frac{ch^{\beta}_1}{ch_0}=\frac{ch_1}{ch_0}-\beta$$

The slope function defines the following torsion pair:

$$\mathscr{T}_{\beta}:=\{E\in Coh(\mathbb{P}^3): \forall E\twoheadrightarrow Q\neq 0, \mu_{\beta}(Q)>0 \}$$

$$\mathscr{F}_{\beta}:=\{E\in Coh(\mathbb{P}^3): \forall 0\neq F\hookrightarrow E,  \mu_{\beta}(F)\leq 0 \}$$

A new heart of a bounded t-structure is defined as the extension closure  $Coh^{\beta}(\mathbb{P}^3):=\left< \mathscr{F}_{\beta}[1], \mathscr{T}_{\beta} \right>$. 
Let $\Gamma$ be the lattice $\Gamma:=(H^3ch_0,H^2ch_1,Hch_2)$. Define a linear function $Z_{\alpha,\beta}: \Gamma \to \mathbb{C}$ as $\displaystyle Z_{\alpha, \beta}:=-Hch^{\beta}_2+\frac{\alpha^2}{2}H^3ch_0+i\cdot H^2ch^{\beta}_1$, in which $\alpha \in \mathbb{R}_{\geq 0}$. The pair $(Coh^{\beta}(\mathbb{P}^3), Z_{\alpha, \beta})$ is called a tilt stability on $\mathbb{P}^3$ (\cite{Bayer_2013}). It is a weak stability condition since $Z_{\alpha, \beta}$ maps some non zero objects (for instance, skyscraper sheaves) to $0$. 

Let $\displaystyle \nu_{\alpha, \beta}:=\frac{Hch^{\beta}_2-\frac{\alpha^2}{2}H^3ch_0}{H^2ch^{\beta}_1}$ be the slope function defined by $Z_{\alpha, \beta}$. We have the following Bogomolov inequality for tilt stable objects:

\begin{Th}(\cite{Bayer_2013}, Cor $7.3.2$). For any $\nu_{\alpha, \beta}$ semistable objects $E\in Coh^{\beta}(X)$, the following inequality holds
$$\overline{\Delta}(E):=(H^2ch^{\beta}_1)^2-2H^3ch^{\beta}_0\cdot Hch^{\beta}_2(E)\geq 0$$
\end{Th}

For smooth surfaces, a single tilt would be sufficient to define stability conditions ( \cite{Bridgeland_2008}, \cite{MR2998828}, \cite{MR2852118}). For a threefold, we need an additional tilt.

Define a torsion pair on $Coh^{\beta}(\mathbb{P}^3)$ as

$$\mathscr{T}_{\alpha,\beta}:=\{E\in Coh^{\beta}(\mathbb{P}^3): \forall E\twoheadrightarrow Q\neq 0, \nu_{\alpha, \beta}(Q)>0 \}$$

$$\mathscr{F}_{\alpha,\beta}:=\{E\in Coh^{\beta}(\mathbb{P}^3): \forall 0\neq F\hookrightarrow E,  \nu_{\alpha, \beta}(F)\leq 0 \}$$

Similarly, we obtain a heart of a bounded t-structure as $\mathscr{A}^{\alpha, \beta}(\mathbb{P}^3):=\left< \mathscr{F}_{\alpha,\beta}[1], \mathscr{T}_{\alpha,\beta} \right>$. The central charge and the bilinear form $Q$ defining the support property are given as follows: (\cite{Bayer_2013},\cite{Macr__2014}, \cite{MR3597844}) 
\[
\begin{array}{rl}
Z_{\alpha, \beta, s}:= & -ch^{\beta}_3+(s+\frac{1}{6})\alpha^2H^2ch^{\beta}_1+i\cdot(Hch^{\beta}_2-\frac{\alpha^2}{2}H^3ch^{\beta}_0)\\
Q_{\alpha, \beta, K}(E)= & ((ch_1(E))^2-2ch_0(E)ch_2(E))(K\alpha^2+\beta^2)+\\
&(6ch_0(E)ch_3(E)-2ch_1(E)ch_2(E))\beta-6ch_1(E)ch_3(E)+4(ch_2(E))^2\\
\end{array}
\]

\begin{Th}
(\cite{Bayer_2013},\cite{Macr__2014}).
$(\mathscr{A}^{\alpha,\beta}((\mathbb{P}^3)), Z_{\alpha, \beta, K})$ is a Bridgeland Stability Condition on $\mathbb{P}^3$ for all $\beta\in \mathbb{R}$, and $\alpha,s>0$. The support property is satisfied with respect to $Q_{\alpha, \beta, K}$. 
\end{Th}

Lastly, for a smooth threefold $X$ and a given class $v\in K_{num}(X)$, there is the following continuity result.

\begin{Th} (\cite{MR3573975} Prop $8.10$)
The function $\displaystyle \mathbb{R}_{>0}\times \mathbb{R}\times \mathbb{R}_{>0} \to Stab(X,v)$ defined by 

$\displaystyle (\alpha, \beta, s) \mapsto (\mathscr{A}^{\alpha, \beta}(X), Z_{\alpha, \beta, s})$ is continuous.
\end{Th}

\section{The Euler stability}

In this section, we start with defining the Euler stability on $\mathbb{P}^n$ based on \cite{altavilla2019euler}. See \cite{altavilla2019euler} for further details and applications on Euler stability. And then in section $3.2$, we focus on the three-dimensional case and show that the Euler stability on $\mathbb{P}^3$ can also be defined by the tilting approach.

\subsection{The Euler stability condition on $\mathbb{P}^n$}

We start with defining the category.

\begin{Def}
Define $\mathscr{A}_m:=\left< \mathscr{O}_{\mathbb{P}^n}(-m-n)[n], \mathscr{O}_{\mathbb{P}^n}(-m-n+1)[n-1], ... , \mathscr{O}_{\mathbb{P}^n}(-m) \right>$ ($m\in \mathbb{Z}$) as the extension closure of the exceptional collection $\{ \mathscr{O}_{\mathbb{P}^n}(-m-n)[n], \mathscr{O}_{\mathbb{P}^n}(-m-n+1)[n-1], ...,, \mathscr{O}_{\mathbb{P}^n}(-m) \}$ on $\mathbb{P}^n$.
\end{Def}

A general element $E \in \mathscr{A}_m$ is a complex: 

$$E=\left[ \mathscr{O}^{a_{-n}}_{\mathbb{P}^n}(-m-n)\to \mathscr{O}^{a_{-n+1}}_{\mathbb{P}^n}(-m-n+1) \to ... \to \mathscr{O}^{a_0}_{\mathbb{P}^n}(-m)  \right]$$ 

where $a_i\in \mathbb{Z}_{\geq 0}$ for $i=0, -1, ..., -n$. The (n$+$1)-tuple $[a_{-n}, ..., a_0]$ is called the dimension vector of $E$ in $\mathscr{A}_m$, i.e. $\underline{dim}(E):=[a_{-n}, ..., a_0]$ .

Next, we define the central charge as $Z_t:=\chi'_t+i\cdot \chi_t$, where $\chi_t$ is the twisted Euler characteristic 

$$\chi_t(E)=\int_{\mathbb{P}^n}ch(E)\cdot ch(\mathscr{O}_{\mathbb{P}^n}(t))\cdot Td(\mathbb{P}^n)$$
and $\chi'_t(E)$ is the derivative of $\chi_t(E)$ with respect to $t$.

\begin{Prop}
The pair $\sigma_t:=(\mathscr{A}_m, Z_t=\chi'_t+i\cdot \chi_t)$ defines a pre-stability condition for all $m\in \mathbb{Z}$, $t\in (m-1,m]$.
\end{Prop}

\begin{proof}

Firstly, the pair $(\mathscr{A}_m, Z_t)$ satisfies the Harder-Narasimhan property because $\mathscr{A}_m$ is Artinian, meaning that every descending chain in $\mathscr{A}_m$ terminates. We prove next that $Z_t$ maps every object in $\mathscr{A}_m$ to the upper half-plane. It follows from the observation that the polynomial $\chi_t(\mathscr{O}_{\mathbb{P}^n})$ has simple roots $t= -1,...,-n$. We show the details below.

It's sufficient to check that $Z_t$ maps the generators $\left\{ \mathscr{O}_{\mathbb{P}^n}(-m-n)[n], ..., \mathscr{O}_{\mathbb{P}^n}(-m)\right\}$ of $\mathscr{A}_m$ to the upper plane. More precisely, we will show that $\chi_t(\mathscr{O}_{\mathbb{P}^n}(-m-i)[i]) \geq 0$, and if $\chi_t(\mathscr{O}_{\mathbb{P}^n}(-m-i)[i])=0$, then $\chi'_t(\mathscr{O}_{\mathbb{P}^n}(-m-i)[i])<0$.

The polynomial $\chi_t(\mathscr{O}_{\mathbb{P}^n}(-m))=(t+n-m)(t+n-m-1)\cdots(t+1-m)/n!$ has simple roots $t=m-1, m-2, ..., m-n$. The sign of $\chi_t$ is given as follows,

\begin{enumerate}
    \item $\chi_t(\mathscr{O}_{\mathbb{P}^n}(-m))>0$ for

$
\begin{array}{lc}
t\in (m-n,m-n+1) ,...,(m-3,m-2), (m-1,\infty)& \text{if}\quad n=  \text{even}  \\
t\in (m-n+1,m-n+2) ,...,(m-3,m-2), (m-1,\infty) & \text{if}\quad n=  \text{odd}
\end{array}
$

\item $\chi_t(\mathscr{O}_{\mathbb{P}^n}(-m))<0$ for

$
\begin{array}{lc}
t\in (m-n+1,m-n+2) ,...,(m-4,m-3), (m-2,m-1)& \text{if}\quad n=  \text{even}  \\
t\in (m-n,m-n+1) ,...,(m-4,m-3), (m-2,m-1) & \text{if}\quad n=  \text{odd}
\end{array}
$
\end{enumerate}

\begin{enumerate}[$\bullet$]
    \item For $\mathscr{O}_{\mathbb{P}^n}(-m)$, we have $\chi_t(\mathscr{O}_{\mathbb{P}^n}(-m))>0$ when $t\in (m-1,m]$.

    \item For $\mathscr{O}_{\mathbb{P}^n}(-m-1)[1]$,

$\chi_t(\mathscr{O}_{\mathbb{P}^n}(-m-1)[1]=-\chi_t(\mathscr{O}_{\mathbb{P}^n}(-m-1))=-\chi_{t-1}(\mathscr{O}_{\mathbb{P}^n}(-m))>0$ when $t\in (m-1,m)$. $\chi_t(\mathscr{O}_{\mathbb{P}^n}(-m-1)[1]=0$ when $t=m$, and in this case we have $\chi'_m(\mathscr{O}_{\mathbb{P}^n}(-m-1)[1])=-\chi'_m(\mathscr{O}_{\mathbb{P}^n}(-m-1))<0$.

\item For $\mathscr{O}_{\mathbb{P}^n}(-m-i)[i]$, $i=2,3,...,n$, using the sign of $\chi_t$ above it's straightforward to check that

\[
\begin{array}{ll}
\chi_t(\mathscr{O}_{\mathbb{P}^n}(-m-i)[i])>0 & when t\in (m-1,m)\\
\chi_t(\mathscr{O}_{\mathbb{P}^n}(-m-i)[i])=0 & when t=m\\
\end{array}
\]

The sign of $\chi'_m(\mathscr{O}_{\mathbb{P}^n}(-m-i))$ is alternating for $i=1,2,...,n$. So the sign of $\chi'_m(\mathscr{O}_{\mathbb{P}^n}(-m-i))$ and $\chi'_m(\mathscr{O}_{\mathbb{P}^n}(-m))$ are different by $(-1)^i$. On the other hand, the shift by $[i]$ will change the sign of $\chi'_t(\mathscr{O}_{\mathbb{P}^n}(-m-i))$ by $(-1)^i$. So the sign of $\chi'_m(\mathscr{O}_{\mathbb{P}^n}(-m-i)[i])$ and $\chi'_m(\mathscr{O}_{\mathbb{P}^n}(-m))$ are different by $(-1)^i \cdot (-1)^i=1$. This imples that the signs of $\chi'_m(\mathscr{O}_{\mathbb{P}^n}(-m-i))[i]$ are the same for all $i=0,1,...,n$, and they are all negative.  

\end{enumerate}
\end{proof}

The proof above indicates that the stability conditions can be extended to the following continuous family.

\begin{Def}
Define $\mathscr{A}_t:=\mathscr{A}_{\lceil t\rceil}$ ($t\in \mathbb{R}$), where "$\lceil t \rceil$" is the roundup of $t$ to the closest integer.
\end{Def}

\begin{Cor}
The pair $\sigma_t=(\mathscr{A}_t, Z_t=\chi'_t+i\cdot\chi_t)$ defines a family of pre-stability conditions on $\mathbb{P}^n$ for $t\in \mathbb{R}$.
\end{Cor}

Next, we show that $\sigma_t=(\mathscr{A}_t, Z_t=\chi'_t+i\cdot\chi_t)$ satisfies the support property, and then $\sigma_t$ will be a stability condition. We will show in the next proposition the case of $\mathbb{P}^3$ when $t\in (0,1]$. The proof for all $\mathbb{P}^n$ ($n\geq 4$) and $t\in \mathbb{R}$ is analogous.

\begin{Prop}
The pre-stability condition $\sigma_t=(\mathscr{A}_t, Z_t=\chi'_t+i\cdot \chi_t)$ on $\mathbb{P}^3$ satisfies the support property for $t\in (0,1]$.
\end{Prop}

\begin{proof}
By definition, we need to find numbers $C_t>0$, such that $\displaystyle \frac{|Z_t(E)|}{||v(E)||}>C_t$ for any $t\in (0,1]$. The category is $\mathscr{A}_1$ which is generated by these objects:

$$u_1:=\displaystyle \mathscr{O}_{\mathbb{P}^3}(-4)[3], u_2:=\mathscr{O}_{\mathbb{P}^3}(-3)[2], u_3:=\mathscr{O}_{\mathbb{P}^3}(-2)[1], u_4:=\mathscr{O}_{\mathbb{P}^3}(-1)$$

Observe that for any $t\in (0,1]$, the linear span $U:=\{u=a_1u_1+a_2u_2+a_3u_3+a_4u_4| a_i\geq 0\}$ is not the entire upper half plane. For instance, Figure \ref{support property} shows the case when $t=0.9$. So we can find a line $\Gamma_t$ (the dotted line in the figure) such that all the objects in $U$ have a non zero projection to $\Gamma_t$. For any object $\displaystyle u=\sum_{i=1}^4 a_iu_i$ ($a_i\geq 0$), denote its projection to $\Gamma_t$ by $p_t(u)$, and let $a_t:=min \{ |Z_t(p_t(u_i))| :i=1,2,3,4\}$, $b_t:=max \{ ||(u_i)|| :i=1,2,3,4\}$. 

$$ \frac{|Z_t(u)|}{||v(u)||}=\frac{|Z_t(\sum_{i=1}^4 a_iu_i)|}{||\sum_{i=1}^4 a_iu_i||}>\frac{|\sum_{i=1}^4a_ip_t(u_i)|}{\sum_{i=1}^4a_i||u_i||}\geq \frac{(\sum_{i=1}^4a_1)a_t}{(\sum_{i=1}^4a_1)b_t}=\frac{a_t}{b_t}=:C_t>0$$

\begin{figure}[ht]
\begin{tikzpicture}[scale=1.2]
\begin{axis}[ axis lines=middle
,xmin=-4,xmax=4,ymin=-4,ymax=4
]
\addplot [domain=0:3, thick]{0.5*x};
\addplot [domain=-3:0, thick]{-0.096*x};
\addplot[black, dashed] {-5*x};
\filldraw [black] (1.619,0.81) circle (2pt);
\node[color=black] at (1.619,1.2) {$u_1$};
\filldraw [black] (-0.229,0.031) circle (2pt);
\filldraw [black] (-0.161,0.018) circle (2pt);
\node[color=black] at (-0.229,0.5) {$u_2$};
\node[color=black] at (-0.461,-0.4) {$u_3$};
\filldraw[black] (-0.449,0.043);
\node[color=black] at (-0.749,0.4) {$u_4$};
\node[color=black] at (-1,3) {$\Gamma_t$};
\filldraw [black] (1,2) circle (2pt);
\node[color=black] at (1,2.5) {$u$};
\filldraw [black] (-0.3,1.5) circle (2pt);
\node[color=black] at (-1,1.6) {$p_t(u)$};
\end{axis}
\end{tikzpicture}
\caption{}
    \label{support property}
\end{figure}

\end{proof}

\subsection{The Euler Stability Condition on $\mathbb{P}^3$ from tiltings}

In this subsection, we will show that the Euler stability is indeed related to the stability $\sigma_{\alpha,\beta,s}=(\mathscr{A}^{\alpha, \beta}, Z_{\alpha,\beta,s})$ in the way that $\mathscr{A}_t$ is an additional tilt of $\mathscr{A}^{\alpha, \beta}$. We start with reviewing the three tilts from $Coh(\mathbb{P}^3)$, where all central charges are defined by derivatives of $\chi_t$. And then we show that the heart $\mathscr{A}_t$ in Euler stability coincides with a one-dimensional slide of the heart by a triple tilts.

\begin{enumerate}
    \item The first tilt.

The first tilt is made with respect to the first Todd class of $\mathbb{P}^3$ ($td_1(\mathbb{P}^3)=2H$). Define the central charge on $Coh(\mathbb{P}^3)$ as $Z_{1,t}=-\chi''_t + i \cdot \chi'''_t$, and the slope function is

$$\mu_t:=\frac{\chi''_t}{\chi'''_t}=\frac{ch^{-t-2}_1}{ch_0}=\frac{ch_1+(t+2)ch_0}{ch_0}$$

This defines the following torsion pair of $Coh(\mathbb{P}^3)$

$\mathscr{T}:=\left\{E\in Coh(\mathbb{P}^3): \forall E\twoheadrightarrow Q, \mu_t(Q)>0 \right\}$

$\mathscr{F}:=\left\{E\in Coh(\mathbb{P}^3): \forall F\hookrightarrow E, \mu_t(F)\leq 0 \right\}$

and we obtain the tilted heart as $Coh^{-t-2}(\mathbb{P}^3):=\left< \mathscr{F}[1], \mathscr{T} \right>$.

\item The second tilt.

We define the weak stability function $Z_{2,t}:=-\chi'_t+i\cdot\chi''_t$ on $Coh^{-t-2}(\mathbb{P}^3)$ and its slope is given by: 

$$\nu_t:=\frac{\chi'_t}{\chi''_t}=\frac{ch^{-t-2}_2-\frac{1}{6}ch_0}{ch^{-t-2}_1}$$

We have the torsion pair of the category $Coh^{-t-2}(\mathbb{P}^3)$ and the new heart of a bounded t-structure as follows

$\mathscr{T}_{\beta}:=\left\{E\in Coh^{-t-2}(\mathbb{P}^3): \forall E\twoheadrightarrow Q, \nu_t(Q)>0 \right\}$

$\mathscr{F}_{\beta}:=\left\{E\in Coh^{-t-2}(\mathbb{P}^3): \forall F\hookrightarrow E, \nu_t(F)\leq 0 \right\}$

$\mathscr{B}_t:=\left< \mathscr{F}_{\beta}[1], \mathscr{T}_{\beta} \right>$. 

The central charge for $\mathscr{B}_t$ is defined as 

$$Z_{3,t}=-\chi_t+i\cdot \chi'_t=-ch^{-t-2}_3+\frac{1}{6}ch^{-t-2}_1+i\cdot (ch^{-t-2}_2-\frac{1}{6}ch^{-t-2}_0)$$

Comparing $Z_{3,t}$ to $Z_{\alpha, \beta, s}$ in the previous section, we have $\alpha=\frac{1}{\sqrt{3}}$, $\beta=-t-2$, and $s=\frac{1}{3}$. So the pair $(\mathscr{B}_t, Z_t)$ is a one-dimensional stability condition and the support property is given by the following quadratic form: 

\[
\begin{array}{rl}
Q_t(E)= & ((ch_1(E))^2-2ch_0(E)ch_12(E))(\frac{1}{3}+(t+2)^2)+\\
& (6ch_0(E)ch_3(E)-2ch_1(E)ch_2(E))(-t-2)\\
& -6ch_1(E)ch_3(E)+4(ch_2(E))^2
\end{array}
\]

\item {The third tilt.}

We now make a tilt of the category $\mathscr{B}_t$. The slope on $\mathscr{B}_t$ is defined as $\lambda_t=\frac{\chi_t}{\chi'_t}$. Similarly, we have the following torsion pair: 

$$\mathscr{T}':=\left\{E\in \mathscr{B}_t: \forall E\twoheadrightarrow Q, \lambda_t(Q)>0 \right\}$$  

$$\mathscr{F}':=\left\{E\in \mathscr{B}_t: \forall F\hookrightarrow E, \lambda_t(F)\leq 0 \right\}$$

Define $\mathscr{A}'_t:=\left<\mathscr{F}'[1],\mathscr{T}'\right>$ to be the new heart of bounded t-structure, and $Z_t := \chi'_t+i\cdot \chi_t$. 
\end{enumerate}

The next proposition shows that $\mathscr{A}'_t$ coincides with $\mathscr{A}_t$, so the Euler stability $(\mathscr{A}_t, Z_t)=(\mathscr{A}'_t, Z_t)$ is a stability condition by tilting $Coh(\mathbb{P}^3)$ three times.

\begin{Prop}{\label{Euler heart from tilts}}
The category $\mathscr{A}'_t$ is the extension closure of the following objects: 

$$\left\{ \mathscr{O}_{\mathbb{P}^3}(-n-3)[3], \mathscr{O}_{\mathbb{P}^3}(-n-2)[2], \mathscr{O}_{\mathbb{P}^3}(-n-1)[1], \mathscr{O}_{\mathbb{P}^3}(-n) \right\}$$ 

where $n:=\lceil t \rceil \in \mathbb{Z}$.
\end{Prop}
\begin{proof}
$\mathscr{A}'$ is a heart of a bounded t-structure since it is obtained from tilting $Coh(\mathbb{P}^3)$ (\cite{Happel_1996}). The category $\mathscr{A}_t$ generated by the objects

$$\left\{ \mathscr{O}_{\mathbb{P}^3}(-n-3)[3], \mathscr{O}_{\mathbb{P}^3}(-n-2)[2], \mathscr{O}_{\mathbb{P}^3}(-n-1)[1], \mathscr{O}_{\mathbb{P}^3}(-n)\right\}$$

is also the heart of a bounded t-structure (\cite{MR2335991} Lemma $3.14$). By \cite{Macr__2017} Prop $5.6$, two hearts must coincide if one is contained in another. So it's sufficient to prove that the objects $\left\{ \mathscr{O}_{\mathbb{P}^3}(-n-3)[3], \mathscr{O}_{\mathbb{P}^3}(-n-2)[2], \mathscr{O}_{\mathbb{P}^3}(-n-1)[1], \mathscr{O}_{\mathbb{P}^3}(-n)\right\}$ generating $\mathscr{A}_t$ are all contained in $\mathscr{A}'_t$.

We will prove that the objects $\left\{ \mathscr{O}_{\mathbb{P}^3}(-3)[3], \mathscr{O}_{\mathbb{P}^3}(-2)[2], \mathscr{O}_{\mathbb{P}^3}(-1)[1], \mathscr{O}_{\mathbb{P}^3}\right\}$ are in the category $\mathscr{A}'_t$ where $t\in (-1,0]$, and the proof for a general $t\in \mathbb{R}$ is analogous. We will keep track of the line bundles $\mathscr{O}_{\mathbb{P}^3}, \mathscr{O}_{\mathbb{P}^3}(-1), \mathscr{O}_{\mathbb{P}^3}(-2),\mathscr{O}_{\mathbb{P}^3}(-3)$ in $Coh(\mathbb{P}^{3})$ while we make the three time tiltings.

The line bundles $\mathscr{O}_{\mathbb{P}^3}(m)$ are Mumford stable for any $m\in \mathbb{Z}$ , so they are also twisted Mumford stable.
The fact that $\mu_t(\mathscr{O}_{\mathbb{P}^3})>0$
$\mu_t(\mathscr{O}_{\mathbb{P}^3}(-1))>0$
$\mu_t(\mathscr{O}_{\mathbb{P}^3}(-2))<0$
$\mu_t(\mathscr{O}_{\mathbb{P}^3}(-3))<0$ for $t\in (-1,0]$ implies
that the category $Coh^{-t-2}(\mathbb{P}^3)$ contains these objects 

$$\mathscr{O}_{\mathbb{P}^3}, \mathscr{O}_{\mathbb{P}^3}(-1), \mathscr{O}_{\mathbb{P}^3}(-2)[1], \mathscr{O}_{\mathbb{P}^3}(-3)[1]$$
when $t\in (-1,0]$.

Then we use the slope function $\nu_t=\frac{ch^{-t-2}_2-1/6ch_0}{ch^{-t-2}_1}$ for the second tilt. It was proved in \cite{Macr__2014} and \cite{MR3597844} that line bundles and their shift $\mathscr{O}_{\mathbb{P}^3}(m)$, 
$\mathscr{O}_{\mathbb{P}^3}(m)[1]$ are tilt stable. A straightforward computation shows that $\mathscr{B}_t$ (a tilt of $Coh^{-t-2}(\mathbb{P}^3)$), when $t\in (-1,0]$, contains the following objects:

\begin{equation*}
\mathscr{O}_{\mathbb{P}^3}, \enskip
\left\{ 
\begin{array}{ll}
    \mathscr{O}_{\mathbb{P}^3}(-1) &  t\in (-1+1/\sqrt{3},0] \\
    \mathscr{O}_{\mathbb{P}^3}(-1)[1] &  t\in (-1,-1+1/\sqrt{3}]
\end{array}
\right., \enskip
\left\{ 
\begin{array}{ll}
    \mathscr{O}_{\mathbb{P}^3}(-2)[1] &  t\in (-1/\sqrt{3},0] \\
    \mathscr{O}_{\mathbb{P}^3}(-2)[2] &  t\in (-1,-1/\sqrt{3}]
\end{array}
\right., \enskip
\mathscr{O}_{\mathbb{P}^3}(-3)[2]
\end{equation*}

For the last tilt, we have that line bundles and their shifts are stable (\cite{Macr__2014}), i.e. $\mathscr{O}_{\mathbb{P}^3}(m)$, $\mathscr{O}_{\mathbb{P}^3}(m)[1]$, $\mathscr{O}_{\mathbb{P}^3}(m)[2]$ are stable in the double tilt $\mathscr{B}_t$.

Using the slope function $\displaystyle \lambda_t=\frac{\chi_t}{\chi'_t}=\frac{ch^{-t-2}_3-1/6ch^{-t-2}_1}{ch^{-t-2}_2-1/6ch_0}$, the claim follows from a direct computation of those $\lambda_t(\mathscr{O}_{\mathbb{P}^3}(-i)[j])$'s in the last step. So as expected, the objects
$\mathscr{O}_{\mathbb{P}^3}, \mathscr{O}_{\mathbb{P}^3}(-1)[1], \mathscr{O}_{\mathbb{P}^3}(-2)[2], \mathscr{O}_{\mathbb{P}^3}(-3)[3]$ are in the category $\mathscr{A}'_t$. 
\end{proof}

\section{The Gieseker chamber for Euler stabilty}

In this section, we work with a fixed one-dimensional class $v=(0,0,m=ch_2>0,ch_3)\in K_{num}{\mathbb{P}^3}$, and show part of the result that there exists a Gieseker chamber for the Euler stability condition $\sigma_t=(\mathscr{A}_t, Z_t=\chi'_t+i\cdot\chi_t)$ on $\mathbb{P}^3$. We expect that the Gieseker chamber shows up for $t>>0$, and for $t<<0$ stable objects are shifted Gieseker stable sheaves $F[1]$ by duality results (in section 5). We modify the stability condition on $\mathscr{B}_t$ as $\sigma_{t,u}=(\mathscr{B}_t, Z_{t,u}=-\chi_t+\frac{u^2}{2}\chi''_t+i\cdot\chi'_t)$ and work mostly in the $(t,u)$ plane. $\sigma_{t,u}$ is indeed a stability condition from \cite{Bayer_2013},\cite{Macr__2014}. We start with describing the numerical walls, and then prove that the Gieseker chamber shows up for $u>>0$. Finally, we work on the global walls for the class $v$. We also use the Euler characteristic $\chi=2ch_2+ch_3$ (from Riemann-Roch) of the class $v$ if it's more convenient.

\subsection{Descriptions of numerical walls.}

Let $E\in \mathscr{B}_t$ be an object whose class is $v$, and its Hilbert polynomial is $P_E(t) = mt + \chi$. Suppose a potential wall is defined by the short exact sequence $0\to A\to E \to B \to 0$ in $\mathscr{B}_t$, then a direct computation shows that its numerical wall falls into the following three possibilities in the $(t,u)$ plane:

\begin{enumerate}
\item[Type 1] Shown as the red walls in Figure \ref{Numerical walls type 1,2}. There is a bounded elliptic part of the wall and a "vertical" part whose asymptote is defined by $t=-\frac{ch_1(A)}{ch_0(A)}-2$. The green point in Figure \ref{Numerical walls type 1,2} has coordinates $C:=(-\frac{\chi}{m},0)$. It is the center of the "elliptic part", which means that all the "elliptic parts" of the walls of Type 1 form a nested family with the center $C$. 
\item[Type 2] Shown as the blue wall in Figure \ref{Numerical walls type 1,2}. This is when $ch_0(A)=0$ and $ch_1(A)\neq 0$. (An object $A\in \mathscr{B}_t$ with $ch_0(A)=ch_1(A)=0$ doesn't define a wall.) The wall is a semi-circle with center $C$ (same $C$ as in Type 1). All the semi-circles form a nested family with the same center $C$.
\item[Type 3] The mirror image of Type 1, as shown in Figure \ref{Numerical walls type 3}. The asymptote of the vertical part is also defined by $t=-\frac{ch_1(A)}{ch_0(A)}-2$ and the center of the elliptic part is $C$ as well.
\end{enumerate}

\begin{figure}[ht]
     \centering
\begin{tabular}{lr}
\begin{minipage}{.4\textwidth}
\begin{tikzpicture}[scale=0.85]
\begin{axis}[axis lines=middle ,xmin=-3,xmax=1,ymin=-1.5,ymax=2]
\draw[line width=1pt, color=red] (-2.131,2) to[out=260,in=90] (-2.533,0) to[out=270,in=100] (-2.131,-2);
\draw[line width=1pt, color=red] (-1.324,0) to[out=90,in=180] (-0.551,0.795) to[out=0,in=90] (0.349,0);
\draw[line width=1pt, color=red] (-1.324,0) to[out=270,in=180] (-0.551,-0.795) to[out=0,in=270] (0.349,0);
\draw[line width=1pt, color=blue] (-1/3,0)circle(0.8);
\draw[line width=1pt, dashed] (-2.05,2)--(-2.05,-2);
\filldraw[green] (-0.334,0) circle (2pt);
\end{axis}
\end{tikzpicture}
\caption{Type 1 and 2}
\label{Numerical walls type 1,2}

\end{minipage}
\begin{minipage}{.4\textwidth}
\begin{tikzpicture}[scale=0.8]
\begin{axis}[axis lines=middle ,xmin=-2,xmax=2,ymin=-1.5,ymax=2]
\draw[line width=1pt, color=red] (-1.1,0) to[out=90,in=180] (-0.048,0.91) to[out=0,in=90] (0.85,0);
\draw[line width=1pt, color=red] (-1.1,0) to[out=270,in=180] (-0.048,-0.91) to[out=0,in=270] (0.85,0);
\draw[line width=1pt, color=red] (1.5,2) to[out=280,in=90] (1.85,0) to[out=270,in=260] (1.5,-2);
\draw[line width=1pt, dashed] (1.4,2)--(1.4,-2);
\filldraw[green] (-0.334,0) circle (2pt);
\end{axis}
\end{tikzpicture}
\caption{Type 3}
\label{Numerical walls type 3}
\end{minipage}
\end{tabular}
     \caption{Numerical Walls}
     \label{numerical walls}
 \end{figure}

For Type 1 and 3, the "elliptic part" might not show up, but the vertical part always exists.

\subsection{Asymptotic results for sheaves and complexes.}

In this subsection, we show that for a fixed one-dimensional class $v\in K_{num}(\mathbb{P}^3)$, the Gieseker chamber appears in the $(t-u)-$plane when $u>>0$. We start with the asymptotic behavior of sheaves for $u>>0$, and prove that a sheaf $E$ with class $v$ is Gieseker stable if and only if it is $\sigma_{t,u}$ stable when $u>>0$.

\begin{Lemma}\label{Boundedness for sheaves}
For a fixed $\displaystyle t\in \mathbb{Q}$ and a Gieseker stable sheaf $E\in \mathscr{B}_t$ with class $v$, there exist $B_1\in \mathbb{R}$ and $B_2\in \mathbb{R}^+$ such that for all $A\hookrightarrow E\in \mathscr{B}_t$, we have  $\displaystyle \frac{\chi_t(A)}{\chi'_t(A)}\leq B_1$, and $\displaystyle \frac{\chi''_t(A)}{\chi'_t(A)}\geq B_2>0$ or $\chi''_t(A)=0$. 
\end{Lemma}

\begin{proof}
Assume $t=a/b$ where $a\in \mathbb{Z}$, and $b\in \mathbb{Z}^+$.

Let $B\in \mathscr{B}_t$ be the quotient of $A\hookrightarrow E$, i.e. $0\to A\to E\to B\to 0$. The corresponding long exact sequence in $Coh^{-t-2}(\mathbb{P}^3)$  is 
$$0\to \mathscr{H}^{-1}_{\beta}(B) \to A \to E \to \mathscr{H}^{0}_{\beta}(B)\to 0$$
where $\mathscr{H}^{-1}_{\beta}$ and $\mathscr{H}^{0}_{\beta}$ denote the cohomologies in $Coh^{-t-2}(\mathbb{P}^3)$.

This implies that $A\in Coh^{-t-2}\subset \mathscr{B}_t$. So we have $\chi''_t(A)\geq0$, and $\chi'_t(A)\geq0$. 
Moreover, we have $\chi'_t(B)\geq 0$ and $\chi'_t(A)+\chi'_t(B)=\chi'_t(E)=m$. So $\chi'_t(A)\in [0,m]$, and if $\chi'_t(A)=0$ then $\chi''_t(A)=0$ otherwise the tilt slope $\displaystyle \nu(A)=\frac{\chi'_t(A)}{\chi''_t(A)}=0$ will make $A$ shifted in $\mathscr{B}_t$ (i.e. $A[1]\in\mathscr{B}_t$). In this case, $A$ is a sheaf supported on points and this violates the fact that $E$ is Gieseker stable. So $\chi'_t(A)>0$. 

\begin{enumerate}
\item Lower bound for $\displaystyle \frac{\chi''_t(A)}{\chi'_t(A)}$.

$\displaystyle \frac{\chi''_t(A)}{\chi'_t(A)}=\frac{ch_1+(t+2)ch_0}{\chi'_t(A)}\geq \frac{ch_1+(t+2)ch_0}{m}\geq \frac{1/b}{m}=\frac{1}{bm}$ or $\displaystyle \frac{\chi''_t(A)}{\chi'_t(A)}=0$. We choose $\displaystyle B_2=\frac{1}{bm}$. The last inequality holds because $ch_0,ch_1\in \mathbb{Z}$, and $ch_1+(t+2)ch_0>0$.

\item Upper bound for  $\displaystyle \lambda_t(A)= \frac{\chi_t(A)}{\chi'_t(A)}$.

From the fact that $ch_0,ch_1\in \mathbb{Z}$, $ch_2\in \frac{1}{2}\mathbb{Z}$, and $\chi'_t(A)>0$, we have $\displaystyle \chi'_t(A)=ch_2(A)+(t+2)ch_1(A)+\frac{(t+2)^2}{2}ch_0(A)-\frac{1}{6}ch_0(A)\geq\frac{1}{6b}>0$.

So we just need an upper bound for $\chi_t(A)$. Consider the Harder-Narasimhan filtration of $A\in \mathscr{B}_t$ with respect to the central charge $Z_t=-\chi_t+i\cdot\chi'_t$: 

$$0=A_0\subset A_1\subset ...\subset A_i\subset A_{i+1}\subset ... \subset A$$

Let $i\in \mathbb{Z}$ be the number such that the semistable factor $\displaystyle \frac{A_i}{A_{i-1}}$ is the last one whose slope is positive, i.e.  $\displaystyle \lambda_t(\frac{A_i}{A_{i-1}})>0$, and $\displaystyle\lambda_t( \frac{A_{i+1}}{A_{i}})\leq 0$. If there is no such $i$, then $\chi'_t(E)\leq 0$ and we define $B_1$ to be $0$.

Similarly for the Harder-Narasimhan filtration for $E\in \mathscr{B}_t$: 

$$0=E_0\subset E_1\subset ...\subset E_j\subset E_{j+1}\subset ... \subset E$$

let $j$ be the index such that the Harder-Narasimhan factor is the last one with a positive slope, i.e. 
$\displaystyle\lambda_{t}(\frac{E_j}{E_{j-1}})>0$ and $\displaystyle\lambda_{t}(\frac{E_{j+1}}{E_{j}})\leq0$.

In the short exact sequence $0\to A_i \to A \to \frac{A}{A_i}\to 0$ we have $\chi_t(A_i)>0$ and $\displaystyle \chi_t(\frac{A}{A_i})\leq0$ from how we choose $i$. This implies $\chi_t(A)\leq\chi_t(A_i)$, and it is sufficient to find an upper bound for $\chi_t(A_i)$.

Consider $0\to A_i\xrightarrow{f} E$ and the diagram
\begin{center}
\begin{tikzcd}
{}&{}&E_j\arrow[hook]{d}\\
0\arrow{r}&A_i \arrow{dr}{\phi}\arrow{ur}{f_1}\arrow{r}{f}&E\arrow[two heads]{d}\\
{}&{}&E/E_j\\
\end{tikzcd}
\end{center}

$\phi$ is a zero map because $A_i$ is extended by semistable objects with $\lambda_t>0$ and $E/E_j$ is extended by objects with $\lambda_t\leq 0$. So the morphism $f$ lifts to a morphism $f_1$ from $A_i$ to $E_j$ where they are both extended by semistable objects in $\mathscr{B}_t$ with positive slope $\lambda_t$. 

Now we make another tilt from $\mathscr{B}_t$ to $\mathscr{A}_t$ and the morphism $0\to A_i \to E_j$ stays the same in $\mathscr{A}_t$ because they are both generated by objects with positive slopes $\lambda_t$. For a fixed $E$ and $t$, the subobject $E_j$ and its dimension vector are fixed in $\mathscr{A}_t$. There are only finitely many choices of subobjects of $E_j$ in $\mathscr{A}_t$, and this  implies that $\chi_t(A_i)$ is bounded from above for all $A\hookrightarrow E$ in $\mathscr{B}_t$. So there exists $B_1\in \mathbb{R}$ such that $\displaystyle \frac{\chi_t(A)}{\chi'_t(A)}\geq B_1$.

\end{enumerate}

\end{proof}

\begin{Rem}
The above proof (Lemma \ref{Boundedness for sheaves}) works for all sheaves $\mathscr{F}\in Coh(\mathbb{P}^3)$ as well. Just replace $\chi'_t(E)=m$ by $\chi'_t(\mathscr{F})$ which is a fixed number.
\end{Rem}

\begin{Prop}\label{stable sheaf}
If $E\in \mathscr{B}_t$ is a sheaf with class $v$, then $E$ is Gieseker (semi)stable if and only if $E$ is $\sigma_{t,u}-$ (semi)stable for all sufficiently large $u$ (The bound of $u$ can be chosen as $\frac{2}{B_2}(B_1-t-\frac{\chi}{m})$).
\end{Prop}

\begin{proof}

Let $\displaystyle \lambda_{t,u}:=\frac{\chi_t-\frac{u^2}{2}\chi''_t}{\chi'_t}$ be the Bridgeland slope. The "if" part follows from the fact that $\lambda_{t,u}(E)=\frac{mt+\chi}{m}=t+{\chi}/{m}$ if $E$ has class $v$, and the Bridgeland slope is equal to the twisted Mumford slope.

For the "only if" part, suppose the claim is not true, then for any fixed $u_0>>0$ we can always find an $A_{u_0}$, such that $\lambda_{t,u_0}(A_{u_0})\geq(>)\lambda_{t,u_0}(E)$. More explicitly, we have

$$ \lambda_{t,u_0}(A_{u_0})=\frac{\chi_t(A_{u_0})-\frac{u_0^2}{2}\chi''_t(A_{u_0})}{\chi'_t(A_{u_0})}\geq (>)\frac{\chi_t(E)}{\chi'_t(E)}=\frac{mt+\chi}{m}=t+\frac{\chi}{m}$$

If $\chi''_t(A_{u_0})\neq 0$, then $\chi''_t(A_{u_0})> 0$ because $A_{u_0}\in Coh^{-t-2}(\mathbb{P}^3)$. The expression of $\lambda_{t,u}$ shows that $\displaystyle\lim_{u\to \infty}\lambda_{t,u}(A)=-\infty$ (for $\chi''_t(A)\neq 0$). Using the boundedness results from Lemma \ref{Boundedness for sheaves}, we actually have a universal bound and the following claim: 

There exists a $u'\in \mathbb{R}$ such that for all $u>u'$ and $A\hookrightarrow E \in \mathscr{B}_t$, we have $\lambda_{t,u}(A)<\lambda_{t,u}(E)$. This violates the assumption that $A$ destabilize $E$. So it implies that if $E$ is not Euler stable for all large $u$ ($u>u'$), then we must have $ch^{-t-2}_1(A)=\chi''_t(A)=0$ for any destabilizing object $A\hookrightarrow E \in \mathscr{B}_t$.

We then go back to the short exact sequence $0\to A\to E\to B\to 0$ in $\mathscr{B}_t$ with its long exact sequence in $Coh^{-t-2}$: 
\begin{center}
\begin{tikzcd}
    &&&Q \arrow{d} &&\\
    0 \arrow{r} &\mathscr{H}^{-1}_{\beta}(B) \arrow{r} & A \arrow{r}\arrow{ur} &E \arrow{r} & \mathscr{H}^{0}_{\beta}(B) \arrow{r} & 0 \\
\end{tikzcd}
\end{center}

where the morphism $A\to E$ factors through $Q \in Coh^{-t-2}$. 
From the fact that $\chi''_t(A)=ch^{-t-2}_1(A)=\chi''_t(E)=ch^{-t-2}_1(E)=0$, 
and all the objects in the diagram has $ch^{-t-2}_1=\chi''_t\geq 0$ (because they are in $Coh^{-t-2}(\mathbb{P}^3)$), 
we have $\chi''_t(\mathscr{H}^{-1}_{\beta}(B))=\chi''_t(Q)=\chi''_t(\mathscr{H}^{0}_{\beta}(B))=0$.

This implies that $\displaystyle \nu(\mathscr{H}^{-1}_{\beta}(B))=\frac{\chi'_t(\mathscr{H}^{-1}_{\beta}(B))}{\chi''_t(\mathscr{H}^{-1}_{\beta}(B))=0}=\infty$ which contradicts with the fact that it is shifted from $Coh^{-t-2}(\mathbb{P}^3)$ to $\mathscr{B}_t$. So  $\mathscr{H}^{-1}_{\beta}(B)=0$, and the short exact sequence $0\to A\to E\to B\to 0$ is indeed in $Coh^{-t-2}(\mathbb{P}^3)\subset \mathscr{B}_t$.

Similarly, consider its long exact sequence of cohomologies in $Coh(\mathbb{P}^3)$ as: 

$$0\to \mathscr{H}^{-1}(B)\to A \to E \to \mathscr{H}^{0}(B)\to 0$$

From the fact that $ch_0(A)=ch_0(E)=0$ and $ch^{-t-2}_1(A)=ch^{-t-2}_1(E)=0$,
we have 

$ch_0(\mathscr{H}^{-1}(B))=ch_0(\mathscr{H}^{0}(B))$, 
and $ch^{-t-2}_1(\mathscr{H}^{-1}(B))=ch^{-t-2}_1(\mathscr{H}^{0}(B))$. 
In particular, 

$\displaystyle \frac{ch^{-t-2}_1(\mathscr{H}^{-1}(B))}{ch_0(\mathscr{H}^{-1}(B))}=\frac{ch^{-t-2}_1(\mathscr{H}^{0}(B))}{ch_0(\mathscr{H}^{0}(B))}$, 
which is a contradiction unless one of $\mathscr{H}^{-1}(B)$ and $\mathscr{H}^{0}(B)$ is zero and the nonzero object has its twisted Mumford slope $\infty$. 
So  $\mathscr{H}^{-1}(B)=0$ and $\mathscr{H}^{0}(B)\neq 0$ 
with $\displaystyle \frac{ch^{-t-2}_1(\mathscr{H}^{0}(B))}{ch_0(\mathscr{H}^{0}(B))}=\infty$.

This shows that $B$ is actually a sheaf, and the short exact sequence $0 \to A\to E\to B\to 0$ is indeed in $Coh(\mathbb{P}^3)$. Now $A$ is a subsheaf of $E$ and they are both one-dimensional. We have that the slope $\lambda_{t,u}$ for the class $v$ coincides with the Mumford slope, so the assumption $\lambda_{t,u}(A)\geq (>)\lambda_{t,u}(E)$ is equivalent to $E$ being Mumford (Gieseker) unstable which is a contradiction.
\end{proof}

Next, we show that if an object $E\in \mathscr{B}_t$ with class $v$ is $\sigma_{t,u}-$ stable for all $u>>0$, then $E$ must be a sheaf.

\begin{Lemma}{\label{complex is sheaf}}
If $E\in Coh^{\beta}\subset \mathscr{B}_t$ with class $v$, then $E$ is a sheaf.
\end{Lemma}

\begin{proof}
In $Coh^{\beta}$ we have the short exact sequence:

$$0\to \mathscr{H}^{-1}(E)[1]\to E \to \mathscr{H}^{0}(E) \to 0$$

Denote $(\mathscr{T}_t,\mathscr{F}_t)$ to be the torsion pair on $Coh(\mathbb{P}^3)$ defined by $\mu_t$.

Since $ch_0(E)=ch_1(E)$, we have that $ch_0(\mathscr{H}^{-1}(E))=ch_0(\mathscr{H}^{0}(E))$ and $ch_1(\mathscr{H}^{-1}(E))=ch_0(\mathscr{H}^{0}(E))$. This implies that $\mu_t(\mathscr{H}^{-1}(E))=\mu_t(\mathscr{H}^{-1}(E))$, but it contradicts with $\mathscr{H}^{-1}(E)\in \mathscr{F}$ and $\mathscr{H}^{0}(E)\in \mathscr{T}$. So one of $\mathscr{H}^{-1}(E)$ and $\mathscr{H}^{0}(E)$ is zero, and then the twisted Mumford slope $\mu_t$ of the non zero object will be infinity (because it's $\frac{0}{0}$). This shows that $\mathscr{H}^{-1}(E)=0$, and $E\cong \mathscr{H}^{0}(E)$ which is a sheaf.  
\end{proof}

\begin{Prop}\label{stable complex is sheaf}
For an object $E\in \mathscr{B}_t$ with class $v$, if $E$ is $\sigma_{t,u}$ stable for all $u>>0$, then $E$ must be a sheaf. 
\end{Prop}

\begin{proof}
$E$ fits into the short exact sequence in $\mathscr{B}_t$:

$$0\to \mathscr{H}^{-1}_{\beta}(E)[1] \to E \to \mathscr{H}^{0}_{\beta}(E) \to 0$$

It is sufficient to show that $\mathscr{H}^{-1}_{\beta}(E)=0$, and then the claim would follow from Lemma \ref{complex is sheaf}. 

Suppose $\mathscr{H}^{-1}_{\beta}(E)\neq 0$, then $\displaystyle \nu(\mathscr{H}^{-1}_{\beta}(E))=\frac{\chi'_t(\mathscr{H}^{-1}_{\beta}(E))}{\chi''_t(\mathscr{H}^{-1}_{\beta}(E))}\leq 0$ and  $\chi''_t(\mathscr{H}^{-1}_{\beta}(E))>0$.

Consider the $\sigma_{t,u}$ slope of $\mathscr{H}^{-1}_{\beta}(B)[1]$:  

$$\lambda_{t,u}(\mathscr{H}^{-1}_{\beta}(E)[1])=\frac{\chi_t(\mathscr{H}^{-1}_{\beta}(E)[1])-\frac{u^2}{2}\chi''_t(\mathscr{H}^{-1}_{\beta}(E)[1])}{\chi'_t(\mathscr{H}^{-1}_{\beta}(E)[1])}$$

$\displaystyle \lim_{u\to \infty}\lambda_{t,u}(\mathscr{H}^{-1}_{\beta}(E)[1])=\infty$ so $E$ is not stable for $u>>0$ 
unless $\chi''_t(\mathscr{H}^{-1}_{\beta}(E)[1])=0$. But this implies $\displaystyle\nu_t(\mathscr{H}^{-1}_{\beta}(E))=\frac{\chi'_t(\mathscr{H}^{-1}_{\beta}(E))}{\chi''_t(\mathscr{H}^{-1}_{\beta}(E))=0}=\infty$ and $\mathscr{H}^{-1}_{\beta}(E)\in \mathscr{B}_t$ which is a contradiction. So $\mathscr{H}^{-1}_{\beta}(E)=0$ and this proves the claim.
\end{proof}

\begin{Lemma}{\label{complex s}}
For a complex $E\in \mathscr{B}_t$ whose class is $v$, there exists an $u_E$ such that for all $u>u_E$, $E$ is not $\sigma_{t,u}$ stable.
\end{Lemma}

\begin{proof}
Consider the short exact sequence in $\mathscr{B}_t$ as 
$$0\to \mathscr{H}_{\beta}^{-1}(E)[1] \to E \to \mathscr{H}_{\beta}^{0}(E)\to 0$$

we must have $\mathscr{H}_{\beta}^{-1}(E)\neq 0$ and $\mathscr{H}_{\beta}^{-1}(E)\neq 0$ otherwise Lemma \ref{complex is sheaf} will tell us $E$ is a sheaf.

$\mathscr{H}_{\beta}^{-1}(E)\in Coh^{-t-2}(\mathbb{P}^3)$ implies $\chi''_t(\mathscr{H}_{\beta}^{-1}(E))>0$.
The claim follows from the observation that

$$\lambda_{t,u}(\mathscr{H}_{\beta}^{-1}(E)[1])=\frac{\chi_t(\mathscr{H}_{\beta}^{-1}(E)[1])-\frac{u^2}{2}\chi''_t(\mathscr{H}_{\beta}^{-1}(E)[1])}{\chi'_t(\mathscr{H}_{\beta}^{-1}(E)[1])}>\lambda_{t,u}(E)=\frac{mt+\chi}{m}$$

for $u>>0$ and we can even find the critical point $u_E$ when we have the equal sign,  $\displaystyle u_E:=\frac{(t+\frac{\chi}{m})\chi'_t(\mathscr{H}_{\beta}^{-1}(E)[1])-\chi_t(\mathscr{H}_{\beta}^{-1}(E)[1])}{-\chi''_t(\mathscr{H}_{\beta}^{-1}(E)[1])}$
\end{proof}

We are now ready to state the main theorem in this subsection.

\begin{Th}{\label{Euler equals Gieseker}}
For any object $E\in \mathscr{B}_t$ whose class is $v$, there exists $u_E>0$ such that for all $u>u_E$, $E$ is $\sigma_{t,u}-$ stable if and only if $E$ is a Gieseker stable sheaf.
\end{Th}

\begin{proof}
It follows from Proposition \ref{stable sheaf}, \ref{stable complex is sheaf}, and Lemma \ref{complex s}.
\end{proof}

Finally, we show a bound of $t$ for a complex $E\in \mathscr{B}_t$ to exist in the category. 

\begin{Lemma}{\label{bound for complexes}}
If $E\in \mathscr{B}_{t_0}$ is a complex, then $E$ can only exist in $\mathscr{B}_t$ for $t\in [a,a+m]$ where $a\in \mathbb{Z}$, and $t_0\in [a,a+m]$.
\end{Lemma}

\begin{proof}
$E$ fits into the short exact sequence in $\mathscr{B}_t$ as

$$0\to \mathscr{H}^{-1}_{\beta}(E)[1] \to E \to \mathscr{H}^{0}_{\beta}(E) \to 0$$
where $\mathscr{H}^{-1}_{\beta}(E)$ and $\mathscr{H}^{0}_{\beta}(E)$ are its cohomologies in $Coh^{-t-2}(\mathbb{P}^3)$.

Assume $ch_0(\mathscr{H}^{0}_{\beta}(E))=R$, $ch_1(\mathscr{H}^{0}_{\beta}(E))=C$ and $ch_2(\mathscr{H}^{0}_{\beta}(E))=D$ then it follows that   $ch_0(\mathscr{H}^{-1}_{\beta}(E))=R$, $ch_1(\mathscr{H}^{-1}_{\beta}(E))=C$ and $ch_2(\mathscr{H}^{-1}_{\beta}(E))=D-m$. 

From the definition of $Coh^{-t-2}(\mathbb{P}^3)$, we have the numerical results that:
\begin{equation*}
\left\{    
\begin{array}{ll}
\chi'_t(\mathscr{H}^{0}_{\beta}(E))\geq 0     & \chi''_t(\mathscr{H}^{0}_{\beta}(E))\geq 0 \\
\chi'_t(\mathscr{H}^{-1}_{\beta}(E))\leq 0     & 
\chi''_t(\mathscr{H}^{-1}_{\beta}(E))\geq 0\\
\end{array}
\right.
\end{equation*}

A direct computation shows that
\begin{equation*}
\left\{
\begin{array}{l}
\chi''_t(\mathscr{H}^{-1}_{\beta}(E))=\chi''_t(\mathscr{H}^{0}_{\beta}(E))=C+(t+2)R \\ 
\chi'_t(\mathscr{H}^{0}_{\beta}(E))=\frac{(t+2)^2}{2}R+(t+2)C+D-\frac{1}{6}R\\
\chi'_t(\mathscr{H}^{-1}_{\beta}(E))=\frac{(t+2)^2}{2}R+(t+2)C+D-\frac{1}{6}R-m\\
\end{array}
\right.
\end{equation*}
and the $\chi''_t$ is the axis of symmetry of those parabolas defined by $\chi'_t$ if $R\neq 0$.

The claim will follow from the following results.
\begin{enumerate}
    \item If $R>0$, then $f_1(t):=\chi'_t(\mathscr{H}^{-1}_{\beta}(E))$, and $f_2(t):=\chi'_t(\mathscr{H}^{}_{\beta}(E))$ are two parabolas and $f_2(t)$ is a shift upwards by $m$ from $f_1(t)$ as shown in Figure \ref{parabolas}.

The inequality $\chi''_t(\mathscr{H}^{-1}_{\beta}(E))=\chi''_t(\mathscr{H}^{0}_{\beta}(E))=C+(t+2)R+0$ corresponds to the region to the right of the dotted line. So the region for $t$ satisfying the above inequalities is the interval $[A,B]$.  A direct computation shows that $\displaystyle|AB|<\sqrt{\frac{2m}{R}}\leq \sqrt{2m}$ 

\item If $R<0$, then the proof is similar to case $(1)$.

\item If $R=0$, then $C>0$. We have in this case

\begin{equation*}
\left\{
\begin{array}{l}
\chi''_t(\mathscr{H}^{-1}_{\beta}(E))=\chi''_t(\mathscr{H}^{0}_{\beta}(E))=C \\ 
\chi'_t(\mathscr{H}^{0}_{\beta}(E))=(t+2)C+D\\
\chi'_t(\mathscr{H}^{-1}_{\beta}(E))=(t+2)C+D-m\\
\end{array}
\right.
\end{equation*}

The picture is shown as Figure \ref{lines}. A simple computation shows $|AB|=\frac{m}{C}\leq m$.

Finally, when $m\geq 3$, $\sqrt{2m}<m$. So for all the possible cases, we have $|AB|\leq m$, and the claim follows.
\end{enumerate}

\begin{figure}[ht]
     \centering
\begin{tabular}{lr}

\begin{minipage}{.5\textwidth}

\begin{tikzpicture}[scale=1]
\begin{axis}[ axis lines=middle
,xmin=-4,xmax=4,ymin=-4,ymax=4
]
\addplot [domain=-3:3]{x^2-x};
\addplot [domain=-3:3]{x^2-x-2};
\draw [dotted,thick]{(1/2,-4)--(1/2,4)};
\node [black] at (1,1/2) {A};
\node [black] at (2,1/2) {B};
\node [black] at (1,-1) {m};
\node [black] at (2,3) {$f_2$};
\node [black] at (3,2) {$f_1$};
\node [black] at (3.8,1/2) {t};
\node [black] at (1/4,3.8) {s};
\end{axis}
\end{tikzpicture}
\caption{}
    \label{parabolas}

\end{minipage}

\begin{minipage}{.5\textwidth}

\begin{tikzpicture}[scale=1]
\begin{axis}[ axis lines=middle
,xmin=-4,xmax=4,ymin=-4,ymax=4
]
\addplot [domain=-3:3]{x-1};
\addplot [domain=-3:3]{x-2};
\draw [dotted,thick]{(1/2,-4)--(1/2,4)};
\node [black] at (1,1/2) {A};
\node [black] at (2,1/2) {B};
\node [black] at (0.8,-0.8) {m};
\node [black] at (2,2) {$f_2$};
\node [black] at (3,1/2) {$f_1$};
\node [black] at (3.8,1/2) {t};
\node [black] at (1/4,3.8) {s};
\end{axis}
    \end{tikzpicture}
\caption{}
\label{lines}
\end{minipage}

\end{tabular}
\end{figure}

\end{proof}

\subsection{Boundedness of actual walls.}

We have shown that for any one-dimensional class $v\in K_{num}(\mathbb{P}^3)$, $s>>0$ and any $t\in \mathbb{R}$ correspond to the Gieseker chamber in the $(t-u)-$plane for $\mathscr{B}_t$. For a complex $E\in \mathscr{B}_t$ ($E$ is not a sheaf), it's unstable (resp. doesn't exist) for $u>>0$ (resp. $t>>0$ or $t<<0$). So if actual walls were bounded in both $s$ and $t$, then the ourtermost chamber ($s>>0,|t|>>0$) would be the Gieseker chamber. We show some partial results and expectations below.

Firstly we expect the following proposition to be true:

\begin{Prop}{\label{boundedness of walls}}
For any class $v=(0,0,m=ch_2>0,ch_3)\in K_{num}(\mathbb{P}^3)$, the actual walls are all from the bounded parts of the numerical walls (Type 1$\sim$3).  
\end{Prop}

If the proposition were true, then all actual walls would be bounded. This is because the bounded parts of the same type don't intersect. So the outermost wall would at worse consist of three pieces of types 1,2 and 3 each. 
Using the fact $0<\chi'_t(A)<ch_2$, and the same trick in Prop \ref{bound for complexes}, we have that the actual wall can at most cover a region $\mathcal{R}$ whose range in $t$ is $ch_2+2\sqrt{2ch_2}$. On the other hand, the center $C=\frac{\chi}{m}$ is fixed for a fixed $v$. So we have that any wall satisfies $|t-\frac{\chi}{m}|\leq ch_2+2\sqrt{2ch_2}$, and $t>\frac{\chi}{m}+ ch_2+2\sqrt{2ch_2}$ gives the Gieseker chamber for Euler stability.  

Theorem \ref{Euler equals Gieseker} implies that any vertical numerical wall can't be an actual wall for $u>>0$ because there are no wall-crossings for $u>>0$. We expect the following claim which will imply Prop \ref{boundedness of walls} with the help of the Bogomolov inequality. 

\begin{Prop}
For $E\in \mathscr{B}_t$ whose class is $v$, any vertical wall defined by $0\to A\to E \to B\to 0$ can't be an actual wall at $u=0$
\end{Prop}

\section{The duality results}

In this section, we show some duality properties of objects $E\in D^b(Coh(\mathbb{P}^3))$ in both $\mathscr{A}_t$ and $\mathscr{B}_t$. This section is motivated by work in \cite{MR2683291} for Gieseker stable sheaves and the results in \cite{MR3615584} for Bridgeland stable complexes in  $D^b(\mathbb{P}^2)$.  

\begin{Def}
Define the derived dual of $E\in D^b(Coh(\mathbb{P}^3))$ to be $E^D:=R\mathscr{H}om(E,\omega_{\mathbb{P}^3})[2]$.
\end{Def}

\begin{Def}
For a dimension vector $\underline{dim}=[a,b,c,d]$ ($a,b,c,d\in \mathbb{Z}_{\geq 0}$), its opposite vector is defined as $\underline{dim}^{op}:=[d,c,b,a]$.
\end{Def}

\begin{Prop}{\label{duality in At}}
For the stability condition $\sigma_t=(\mathscr{A}_t, Z_t=\chi'_t+i \cdot \chi_t)$, assume $t\notin \mathbb{Z}$. An object $E\in \mathscr{A}_t$ has its dimension vector $\underline{dim}(E)$ if and only if $E^D[1]\in \mathscr{A}_{-t}$ has its dimension vector $\underline{dim}(E^D[1])=\underline{dim}^{op}(E)$. Moreover, $E\in \mathscr{A}_t$ is $\sigma_t$- (semi)stable if and only if $E^D[1]\in \mathscr{A}_{-t}$ is $\sigma_{-t}$- (semi)stable.
\end{Prop}

\begin{proof}
Let $n:=\lceil t\rceil$. For $E\in \mathscr{A}_t$, it is quasi-isomorphic to a complex of vector bundles as: 

$$E\overset{qiso}{\cong}\left[\mathscr{O}^{a_3}_{\mathbb{P}^3}(-n-3)\to \mathscr{O}^{a_2}_{\mathbb{P}^3}(-n-2)\to \mathscr{O}^{a_1}_{\mathbb{P}^3}(-n-1)\to \mathscr{O}^{a_0}_{\mathbb{P}^3}(-n)\right]$$ where $a_i\in \mathbb{Z}_{\geq 0}$. A direct computation from the definition shows that  

$$E^D[1]\overset{qiso}{\cong}\left[\mathscr{O}^{a_0}_{\mathbb{P}^3}(n-4)\to \mathscr{O}^{a_1}_{\mathbb{P}^3}(n-3)\to \mathscr{O}^{a_2}_{\mathbb{P}^3}(n-2)\to \mathscr{O}^{a_3}_{\mathbb{P}^3}(n-1)\right]$$
which implies that $E^D[1]\in \mathscr{A}_{-t}$ with its dimension vector $\underline{dim}(E^D[1])=[a_0,a_1,a_2,a_3]=[a_3,a_2,a_1,a_0]^{op}$.

For the stability, a direct computation shows that $$\chi_t(E)=\chi_{-t}(E^D[1]), \quad \chi'_t(E)=-\chi'_{-t}(E^D[1])$$ where the derivative in $\chi'_{-t}(E^D[1])$ is with respect to the parameter "$-t$".

A short exact sequence $0\to A \to E \to B \to 0$ in $\mathscr{A}_t$ is essentially a sequence of complexes of vector spaces. Taking the dual of it, we have a dual sequence in $\mathscr{A}_{-t}$ as $0\to B^D[1] \to E^D[1] \to A^D[1] \to 0$. Let $\lambda_t:=-\frac{\chi'_t}{\chi_t}$ be the slope function defined by the central charge. Then the numerical results above imply that for any object $A\in \mathscr{A}_t$, $\lambda_t(A)=-\lambda_{-t}(A^D[1])$. 

Now we have the following equivalence: $A\hookrightarrow E$ in $\mathscr{A}_t$ with $\lambda_t(A)<(\leq)\lambda_t(E)$ is equivalent to $E^D[1]\twoheadrightarrow A^D[1]$ in $\mathscr{A}_{-t}$ with $\lambda_{-t}(E^D[1])<(\leq)\lambda_{-t}(A^D[1])$. So this implies that $E\in \mathscr{A}_t$ is (semi)stable if and only if $E^D[1]\in \mathscr{A}_{-t}$ is (semi)stable, and this proves the claim. 
\end{proof}

\begin{Rem}{\label{heart for t not integer}}
For $t\in \mathbb{R}\backslash\mathbb{Z}$, it can be easily checked that the generators $\mathscr{O}_{\mathbb{P}^3}(-\lceil t\rceil-i)[i]$ ($i=0,1,2,3$) are all mapped to the strict upper half plane by the central charge $Z_t:=\chi'_t+i\cdot \chi_t$, i.e. $\chi_t(\mathscr{O}_{\mathbb{P}^3}(-\lceil t\rceil-i)[i])>0$. In particular, there is no stable object $E\in \mathscr{A}_t$ with phase $1$, equivalently, $\mathscr{A}_t\subset \mathscr{P}_t(0,1)$. So under the duality, $(\mathscr{A}_t, Z_t)$ is sent to $(\mathscr{A}_{-t}, Z_{-t})$ as a stability condition with $\mathscr{A}_{-t}\subset \mathscr{P}_{-t}(0,1)$. Indeed, prop \ref{duality in At} also works for $t\in \mathbb{Z}$, but we need to modify the heart a bit. We will show it in Remark \ref{duality in At when t integer}. 
\end{Rem}

\begin{Rem}
The duality result in Prop \ref{duality in At} and a similar proof work for all $\mathbb{P}^n$ ($n\in \mathbb{Z}_{\geq 0}$).
\end{Rem}

\begin{Cor}{\label{slices for duality}}
For any $t\in \mathbb{R}\backslash \mathbb{Z}$, let $\mathscr{P}_t(\phi)$ ($\phi\in \mathbb{R}$) be the slicing such that the heart $\mathscr{A}_t$ is given by $\mathscr{P}_t(0,1]$. We have $E\in \mathscr{P}_t(\phi)$ if and only if $E^D[1]\in \mathscr{P}_{-t}(1-\phi)$ for $\mathscr{A}_{-t}$.
\end{Cor}

\begin{proof}
It follows from the numerical fact in Prop \ref{duality in At} that for $E\in \mathscr{A}_t$, if $Z_t(E)=Re+Im$, then $Z_{-t}(E^D[1])=-Re+Im$. So the phase $\phi$ ($\phi\in (0,1)$ from Remark \ref{heart for t not integer}) changes from $\phi$ to $1-\phi$ in the corresponding hearts. $\phi$ can be extended to all the real numbers by shifting $\mathscr{A}_t$ and $\mathscr{A}_{-t}$.
\end{proof}

\begin{Rem}{\label{duality in At when t integer}}
Proposition \ref{duality in At} and Corollary \ref{slices for duality} work for $t\in \mathbb{Z}$ as well. If $t\in \mathbb{Z}$, then the stability condition $(\mathscr{A}_t, Z_t)$ is supposed to be sent to $(\mathscr{A}_{1-t}, Z_{-t})$ by the duality. The pair $(\mathscr{A}_{1-t}, Z_{-t})$ is not a stability condition because the stable (simple) objects $\mathscr{O}_{\mathbb{P}^3}(-t-i)[i] \in \mathscr{A}_t$ ($i=1,2,3$) all have phase $1$, and duality will send them to phase $0$ in $\mathscr{A}_{-t}$ which is not in the heart.

This can be fixed by slightly tilting the upper half plane. The heart $\mathscr{A}_t$ ($t\in \mathbb{Z}$) is indeed a strict subset of $\mathscr{P}(0,1]$, and more precisely, $\mathscr{A}_t=\mathscr{P}(\phi_1,1]$ where $\phi_1=\tan^{-1}(\frac{6}{11})$. This is because $Z_t$ sends the generators $\mathscr{O}_{\mathbb{P}^3}(-t)$ to $(\frac{11}{6},1)$ and  $\mathscr{O}_{\mathbb{P}^3}(-t-i)[i]$ ($i=1,2,3$) to $(-\frac{1}{3},0)$ or $(-\frac{1}{6},0)$. So we just modify the heart to be $\mathscr{P}(\phi,\phi+1]$ as shown in Figure \ref{tilt the heart} ($\phi := \frac{1}{2}\phi_1$, and take $t=0$ as an example). The new heart under duality is $\mathscr{P}_{-t}(-\phi, 1-\phi]$ and this fixes the issue. 
\end{Rem}

\begin{figure}[ht]    

\begin{tikzpicture}[scale=1.2]
\begin{axis}[ axis lines=middle
,xmin=-4,xmax=4,ymin=-4,ymax=4
]
\addplot [domain=0:3, thick]{6/11*x};
\addplot [domain=-3:0, thick]{0};
\addplot[black, dashed] {6/22*x};
\filldraw [black] (11/6,1) circle (2pt);
\node[color=black] at (11/6,1.5) {$\mathscr{O}_{\mathbb{P}^3}$};
\filldraw [black] (-1/3,0) circle (2pt);
\filldraw [black] (-1/6,0) circle (2pt);
\node[color=black] at (-1/3,-0.5) {$\mathscr{O}_{\mathbb{P}^3}(-1)[1]$};
\node[color=black] at (-1/3,-1) {$\mathscr{O}_{\mathbb{P}^3}(-2)[2]$};
\node[color=black] at (-1/3,-1.5) {$\mathscr{O}_{\mathbb{P}^3}(-3)[3]$};
\node[color=black] at (3.5,0.5) {$\phi$};
\node[color=black] at (-3.5,-1.5) {$\phi+1$};
\end{axis}
\end{tikzpicture}
\caption{}
    \label{tilt the heart}
\end{figure}

From Corollary \ref{slices for duality}, we have a duality result for the category $\mathscr{B}_t$.

\begin{Cor}{\label{duality in Bt}}
For the stability condition $\sigma^{\mathscr{B}_t}_t=(\mathscr{B}_t, Z_t=-\chi_t+i\cdot \chi'_t)$, an object $E\subset \mathscr{P}_{\mathscr{B}_t}(0,1)\subset  \mathscr{B}_t$ is (semi)stable if and only if $E^D\subset \mathscr{P}_{\mathscr{B}_{-t}}(0,1) \subset \mathscr{B}_{-t} $ is (semi)stable.
\end{Cor}

\begin{proof}
Let $\mathscr{P}(0,1]=\mathscr{A}_t$ in terms of slicing of $\mathscr{A}_t$, then $\mathscr{B}_t=\mathscr{P}(-\frac{1}{2},\frac{1}{2}]$ by definition. The claim then follows from Proposition \ref{duality in At} and Remark \ref{slices for duality}. 
\end{proof}

Finally, we show the duality result for the modified stabilty condition $\sigma_{t,u}=(\mathscr{B}_t, Z_{t,u}=-\chi_t+\frac{u^2}{2}\chi''_t+i\cdot \chi'_t)$ on $\mathscr{B}_t$. Before stating the corollary, we compute the quiver region with respect to the stability condtion $\sigma_{t,u}$. The quiver region for $\mathbb{P}^2$ was introduced in \cite{Arcara_2013} section $7$. The situation for $\mathbb{P}^3$ is analogous. For the stability condition $\sigma_{t,u}=(\mathscr{B}_t, Z_{t,u})$ on $\mathscr{B}_t$, we make a tilt with respect to the slope 

$$\lambda_{t,u}:=-\frac{\mathscr{R}e(Z_{t,u})}{\mathscr{I}m(Z_{t,u})}=\frac{\chi_t-\frac{u^2}{2}\chi''_t}{\chi'_t}$$ 
and we have the following stability condition $\sigma'_{t,u} = (\mathscr{A}_{t,u}, Z'_{t,u}=\chi'_t+ i \cdot(\chi_t-\frac{u^2}{2}\chi''_t))$

The quiver region, detoted by $\mathscr{Q}$, is a region $\mathscr{Q}\subset (t-u)-$plane containing all points $(t,u)$ such that the following exceptional collection is contained in $\mathscr{A}_{t,u}$.

$$\mathscr{O}_{\mathbb{P}^3}(-3-n)[3], \mathscr{O}_{\mathbb{P}^3}(-2-n)[2], \mathscr{O}_{\mathbb{P}^3}(-1-n)[1], \mathscr{O}_{\mathbb{P}^3}(-n),  \quad (n:= \lceil t \rceil)$$

Following the proof of Prop \ref{Euler heart from tilts}, it's straightforward to see the quiver region for $(-1,0]\times \mathbb{R}_{\geq 0}$ in the $(t-u)$-plane is the region below these two hyperbolas $(t+2)^2-2u^2-1=0$ and $(t-1)^2-2u^2-1=0$ (including the $t-$axis), defined by $\lambda_{t,u}(\mathscr{O}_{\mathbb{P}^3})=0$ and $\lambda_{t,u}(\mathscr{O}_{\mathbb{P}^3}(-3))=0$. Since the quiver regions are periodic (the period is $t=1$), the entire requiver region is as shown in Figure \ref{quiver region}.

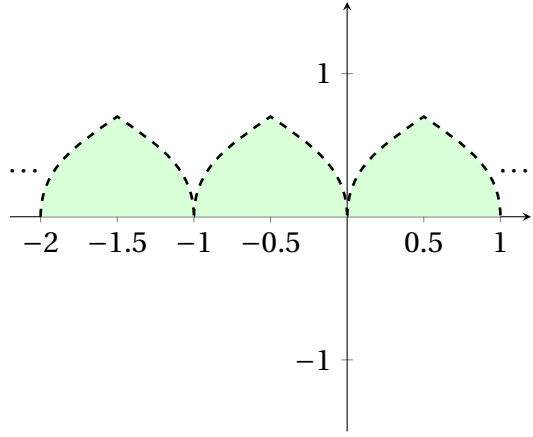
\begin{figure}[ht]
\begin{tikzpicture}
\begin{axis}[axis lines=middle ,xmin=-2.2,xmax=1.2,ymin=-1.5,ymax=1.5]
\draw[line width=1pt,fill=green!15,dashed] (0,0) to[out=90,in=215] (0.5,0.7) to[out=325,in=90] (1,0);
\draw[line width=1pt,fill=green!15,dashed] (-1,0) to[out=90,in=215] (-0.5,0.7) to[out=325,in=90] (0,0);
\draw[line width=1pt,fill=green!15,dashed] (-2,0) to[out=90,in=215] (-1.5,0.7) to[out=325,in=90] (-1,0);
\node[color=black] at (-2.1,0.3) {$\cdots$};
\node[color=black] at (1.1,0.3) {$\cdots$};
\end{axis}
\end{tikzpicture}
\caption{The quiver region}
\label{quiver region}
\end{figure}

\begin{Cor}{\label{duality in Bt with s}}
For the stability condition $\sigma_{t,u}=(\mathscr{B}_t, Z_{t,u}=-\chi_t+\frac{u^2}{2}\chi''_t+i\cdot \chi'_t)$ such that $(t,u)\in \mathscr{Q}$, $E\in \mathscr{P}_{\mathscr{B}_t}(0,1)\subset  \mathscr{B}_t$ is $\sigma_{t,u}-$(semi)stable if and only if $E^D\in \mathscr{P}_{\mathscr{B}_{-t}}(0,1) \subset \mathscr{B}_{-t} $ is $\sigma_{-t,u}-$(semi)stable.
\end{Cor}

\begin{proof}
The reason that $E^D\in \mathscr{B}_{-t} $ is the same with Cor \ref{duality in Bt} or Remark \ref{slices for duality}. For stability, observe that $\chi_t(E)=\chi_t(E^D)$, $\chi''_t(E)=\chi''_t(E^D)$, and $\chi'_t(E)=-\chi'_t(E^D)$. Let $\lambda_{t,u}$ be the slope $\displaystyle \lambda_{t,u}:=\frac{\chi_t-\frac{u^2}{2}\chi''_t}{\chi'_t}$, then $\lambda_{t,u}(E)=-\lambda_{t,u}(E^D)$. It's obvious that $(t,u)\in \mathscr{Q}$ if and only if $(-t,u) \in \mathscr{Q}$. So the claim follows in the same way as the proof in Prop \ref{duality in At}.
\end{proof}

\section{Walls for the class $3t\pm1$}

In this section, fix $v=(0,0,3,-5)$ and its dual class $v^{\vee}:=(0,0,3,-4)$. Their Hilbert polynomials are $P_v(t)=3t+1$ and $P_{v^{\vee}}(t)=3t+2$ respectively. We give two potential walls for $v$ in the $(t,u)$ plane in section $6.1$, and then in section $6.2$, we prove that they are actual walls at $u=0$. Finally, in section $6.3$, we use the duality results to study walls for the dual class $v^{\vee}$.

\subsection{Potential Walls in the "$(t,u)$" plane.}
    
Let $E\in \mathscr{B}_t$ be a sheaf whose class is $v$. In \cite{Freiermuth_2004}, the Gieseker moduli space $\mathscr{M}^{3t+1}_{\mathbb{P}^3}$ consists of two components.  The general sheaves from those component are $\mathscr{O}_{C}$ and $L_{C_E}$, where $C$ is the twisted cubic and $L_{C_E}$ is a degree $1$ line bundle on the plane cubic curve $C_E$. 

For the stability condition $\sigma_t=(\mathscr{B}_t, Z_{t,u})$, there are two walls for those Gieseker stable sheaves which are defined by sheaves $\mathscr{O}_{\mathbb{P}^3}$ and $\mathscr{O}_{\Lambda}$ as follows ($\Lambda\subset\mathbb{P}^3$ is a plane)

$$W_1: \quad  0\to \mathscr{O}_{\mathbb{P}^3} \to \mathscr{O}_C \to Q[1] \to 0 $$

$$W_2: \quad  0\to \mathscr{O}_{\Lambda} \to L_{C_E} \to \mathscr{F}_1\to 0 $$

In the first sequence, the quotient object $Q[1]$ has two possibilities. It can be either the shifted ideal sheaf of a twisted cubic curve $\mathscr{I}_C[1]$ or a shifted sheaf $\mathscr{F}[1]$ where $\mathscr{F}$ contains torsion. The sheaf $\mathscr{F}$ fits into the short exact sequence $0\to \mathscr{O}_{\Lambda}(-3)\to \mathscr{F}\to \mathscr{I}_P(-1)\to 0 $ ($P\subset \Lambda$ is a point in the plane). (\cite{Freiermuth_2004}, \cite{MR3803142}) 

The second sequence is indeed $0\to \mathscr{O}_{\Lambda}(-3) \to \mathscr{O}_{\Lambda} \to L_{C_E} \to \mathbb{C}_P\to 0$, where $P$ is a point on ${C_E}$. The first object  $\mathscr{O}_{\Lambda}(-3)$ need to be shifted to the front because of the definition of $\mathscr{B}_t$ and the position of the wall. The object $\mathscr{F}_1$ is then a complex extended by $\mathbb{C}_P$ and $\mathscr{O}_{\Lambda}(-3)[1]$. So $\mathscr{F}_1$ fits into a short exact sequence: $0\to \mathscr{O}_{\Lambda}(-3)[1] \to \mathscr{F}_1 \to \mathbb{C}_P \to 0$ in $\mathscr{B}_t$. The numerical walls are shown in Figure \ref{walls for 3t+1} (the bounded parts). 

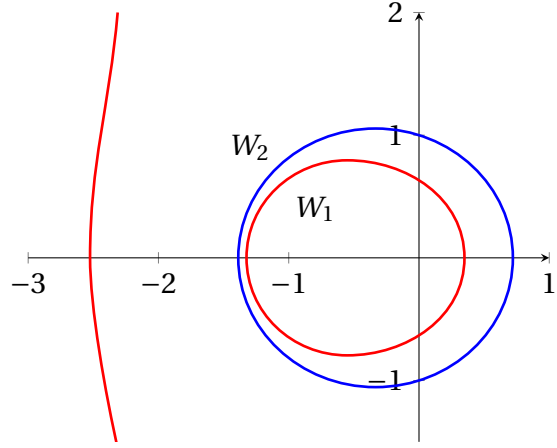
\begin{figure}[ht]
\begin{tikzpicture}
\begin{axis}[axis lines=middle ,xmin=-3,xmax=1,ymin=-1.5,ymax=2]
\draw[line width=1pt, color=red] (-1.324,0) to[out=90,in=180] (-0.551,0.795) to[out=0,in=90] (0.349,0);
\draw[line width=1pt, color=red] (-1.324,0) to[out=270,in=180] (-0.551,-0.795) to[out=0,in=270] (0.349,0);
\draw[line width=1pt, color=red] (-2.313,2) to[out=265,in=90] (-2.525,0) to[out=270,in=275] (-2.131,-2);
\draw[line width=1pt, color=blue] (-1/3,0)circle(1.054);
\node[color=black] at (-0.8,0.4) {$W_1$};
\node[color=black] at (-1.3,0.9) {$W_2$};
\end{axis}
\end{tikzpicture}
\caption{Walls for the class $v=(0,0,3,-5)$}
\label{walls for 3t+1}
\end{figure}

\subsection{Actual walls in $\mathscr{A}_1$}

We will show that those potential walls are actual walls when $u=0$. Figure \ref{walls for 3t+1} shows that the right endpoints of those walls both land in $t\in (0,1]$ which corresponds to the category $\mathscr{A}_1$. The sheaves $\mathscr{O}_{C}$ and $L_{C_E}$ are both in $\mathscr{A}_1$ since they are $1-$regular (Prop $1.8.8$ \cite{MR2095471}). We will prove that the objects $\mathscr{O}_{\mathbb{P}^3}$, $\mathscr{O}_{\Lambda}$, $\mathscr{F}_1$, $Q[1]$ which define $W_1$ and $W_2$ are stable in $\mathscr{A}_1$, and $W_1$ and $W_2$ are the only two actual walls in $\mathscr{A}_1$.

\subsubsection{Stability of $\mathscr{O}_{\mathbb{P}^3}$ and  $\mathscr{O}_{\Lambda}$}

In the category $\mathscr{A}_1$, $\mathscr{O}_{\mathbb{P}^3}$ has its presentation as the Koszul complex:

$$\left[ \mathscr{O}_{\mathbb{P}^3}(-4)\to \mathscr{O}^4_{\mathbb{P}^3}(-3)\to \mathscr{O}^6_{\mathbb{P}^3}(-2)\to \mathscr{O}^4_{\mathbb{P}^3}(-1)\right]\to \mathscr{O}_{\mathbb{P}^3}$$
and its dimension vector is $[1464]$. The object $\mathscr{O}_{\Lambda}$ is a quotient of  $\mathscr{O}_{\mathbb{P}^3}$, and the dimension vector is $[1463]$ in $\mathscr{A}_1$. Its presentation is obtained by taking off one of the $\mathscr{O}_{\mathbb{P}^3}(-1)$ from the presentation of $\mathscr{O}_{\mathbb{P}^3}$ together with all the morphisms mapping to it.

$$\left[ \mathscr{O}_{\mathbb{P}^3}(-4)\to \mathscr{O}^4_{\mathbb{P}^3}(-3)\to \mathscr{O}^6_{\mathbb{P}^3}(-2)\to \mathscr{O}^3_{\mathbb{P}^3}(-1)\right]\to \mathscr{O}_{\Lambda}$$

The stability of $\mathscr{O}_{\mathbb{P}^3}$ (it was also proved in \cite{Macr__2014}) and $\mathscr{O}_{\Lambda}$ in $\mathscr{A}_1$ for all $t\in (0,1]$ follows from checking all their subcomplexes.

\subsubsection{Potential walls}

We prove that $W_1$ and $W_2$ are the only possible walls in $\mathscr{A}_1$ for the class $v=(0,0,3,-5)$. Assume that $E\in \mathscr{A}_1$ with class $v$, and an actual wall for $E$ in $\mathscr{A}_1$ is defined by the short exact sequence $0\to A\to E\to B\to 0$. It's straightforward to check that the dimension vector of $E$ in $\mathscr{
A}_1$ is $[1694]$. More precisely,

$$\left[ \mathscr{O}_{\mathbb{P}^3}(-4)\to \mathscr{O}^6_{\mathbb{P}^3}(-3)\to \mathscr{O}^9_{\mathbb{P}^3}(-2)\to \mathscr{O}^4_{\mathbb{P}^3}(-1)\right]\overset{qiso}{\cong} E$$

$E$ contains a subcomplex as 
$$\left[ 0\to \mathscr{O}^6_{\mathbb{P}^3}(-3)\to \mathscr{O}^9_{\mathbb{P}^3}(-2)\to \mathscr{O}^4_{\mathbb{P}^3}(-1)\right]$$
and the corresponding quotient is 
$$E \twoheadrightarrow \left[ \mathscr{O}_{\mathbb{P}^3}(-4)\to 0\to 0\to 0\right]=\mathscr{O}_{\mathbb{P}^3}(-4)[3]$$

Serre Duality shows that $Hom(E,\mathscr{O}_{\mathbb{P}^3}(-4)[3])\cong Hom(\mathscr{O}_{\mathbb{P}^3},E)^{\vee}$. So there is always a non-zero morphism $\mathscr{O}_{\mathbb{P}^3}\to E$. Moreover, either $A$ or $B$ has dimension vector $[1,a_{-2},a_{-1},a_{0}]$ ($a_0$,$a_{-1}$,$a_{-2} \geq 0$). So without loss of generality, assume $A$ has dimension vector $[1,a_{-2},a_{-1},a_{0}]$. 

Apply Serre duality to $Hom(E,\mathscr{O}_{\mathbb{P}^3}(-4)[3]) \to Hom(A,\mathscr{O}_{\mathbb{P}^3}(-4)[3])$. We have $Hom(\mathscr{O}_{\mathbb{P}^3}, E)^{\vee} \to Hom(\mathscr{O}_{\mathbb{P}^3}, A)^{\vee}$, which gives $Hom(\mathscr{O}_{\mathbb{P}^3}, A)\to Hom(\mathscr{O}_{\mathbb{P}^3}, E)$. This implies that $A$ contains a quotient complex of $\mathscr{O}_{\mathbb{P}^3}$. 

Table \ref{quotient complexes of O_P3} contains all the quotients of $\mathscr{O}_{\mathbb{P}^3}$ in $\mathscr{A}_1$ and their stability. If $A$ is one of the quotient complexes in the table, then a direct computation of their walls show that $\mathscr{O}_{\mathbb{P}^3}$ and $\mathscr{O}_{\Lambda}$ are the only options for $A$ to define an actual wall in $\mathscr{A}_1$.

\begin{table}[ht]
    \centering
    \begin{tabular}{|l|l|}
\hline
quotients of $\mathscr{O}_{\mathbb{P}^3}$& {Region where it }\\
in dimension vectors& is stable \\
\hline
$[1464]=\mathscr{O}_{\mathbb{P}^3}$ & any $t\in (0,1]$\\
\hline
$[1463]=\mathscr{O}_{\Lambda}$ & any $t\in (0,1]$ \\
\hline
$[1462]=$(complex) & stable for $t\in (0,1/2)$ \\
\hline
$[1461]=$(complex) & stable for $t\in (0,0.541))$ \\
\hline
$[1460]=$(complex) & stable for $t\neq1$ \\
\hline
$[1452]=\mathscr{O}_{l}$ ($l$ is a line in $\mathbb{P}^3$) & any $t\in (0,1]$\\
\hline    
$[1451]=$(complex) & stable for $t\in(0,0.528)$ \\
\hline
$[1450]=$(complex) & stable for $t\neq 1$ \\
\hline
$[1441]=$(complex) & stable for $t\in(0,0.586)$ \\
\hline
$[1440]=$(complex) & stable for $t\neq 1$ \\
\hline
$[1431]=$(complex) & stable for $t\in(0,0.423)$ \\
\hline
$[1430]=$(complex) & stable for $t=\neq 1$ \\
\hline
$[1331]=\mathbb{C}_P$ & any $t\in(0,1]$ \\
\hline
$[1330]=$(complex) & stable for $t\neq 1$ \\
\hline
$[1320]=$(complex) & stable for $t\neq1$ \\
\hline
$[1310]=$(complex) & stable for $t\neq 1$ \\
\hline
$[1300]=$(complex) & stable for $t\neq 1$ \\
\hline
$[1210]= \mathscr{O}_{l}(-2)[1]=$(complex) & stable for $t\neq 1$ \\
\hline
$[1200]=$(complex) & stable for $t\neq 1$ \\
\hline
$[1100]=$(complex) & stable for $t\neq 1$ \\
\hline
\end{tabular}
   \caption{Quotient complexes of $\mathscr{O}_{\mathbb{P}^3}$ in $\mathscr{A}_1$.}
    \label{quotient complexes of O_P3}
\end{table}

Next, we show that if $A$ strictly contains one of the quotient complexes $T$ from Table \ref{quotient complexes of O_P3}, i.e. $T\hookrightarrow A$, then $A$ or $B$ won't be semistable at the wall. This implies that an actual wall for the class $v=(0,0,3,-,5)$ can only be defined by $\mathscr{O}_{\mathbb{P}^3}$ or $\mathscr{O}_{\Lambda}$.

Table \ref{possible subcomplex of E} consists of all possible dimension vectors of $A$ that satisfy the following constraints.

\begin{enumerate}
    \item $A$ has the desired dimension vector.

    $T \hookrightarrow A \hookrightarrow E$ and $dim(T)<dim(A)<dim(E)$ in the lexicographical order in $\mathscr{A}_1$.

    \item The wall falls in $\mathscr{A}_1$.

    $\lambda^{Euler}_t :=-\frac{\chi'_t}{\chi_t}$ denotes the slope of the Euler stability. We require that $t_A\in (0,1]$, where $t_A$ is a solution of $\lambda^{Euler}_t(A)=\lambda^{Euler}_t(E)$.
    
    \item $A$ is not destabilized by $T$ at the wall $t_A$.

    $\lambda^{Euler}_{t_A}(T)\leq \lambda^{Euler}_{t_A}(A)$.

\end{enumerate}

\begin{table}[ht]
    \centering
    \begin{tabular}{|l|l|}
\hline
$T$ and $dim(T)$& $dim(A)$ s.t. $A$ satisfies\\
&  constraints (1)$\sim$(3) above\\ 
\hline
$[1464]$, $T=\mathscr{O}_{\mathbb{P}^3}$&$[1564]$\\
\hline
&$[1664]$\\
\hline
&$[1674]$\\
\hline
&$[1684]$\\
\hline
$[1463]$, $T=\mathscr{O}_{\Lambda}$ &$[1563]$\\
\hline
\end{tabular}
\caption{All possibilities of $A$}
\label{possible subcomplex of E}
\end{table}

We show that all objects $A$ from Table \ref{possible subcomplex of E} are unstable at its wall. So none of them can define an actual wall.

\begin{enumerate}
    \item $dim(A)=[1564]$ and $\mathscr{O}_{\mathbb{P}^3} \hookrightarrow A$

    In this case, we have $$0\to \mathscr{O}_{\mathbb{P}^3} \to A \to \mathscr{O}_{\mathbb{P}^3}(-3)[2]\to 0$$

    $Ext^1(\mathscr{O}_{\mathbb{P}^3}(-3)[2],\mathscr{O}_{\mathbb{P}^3})=0$ implies that $A=\mathscr{O}_{\mathbb{P}^3}(-3)[2]\bigoplus\mathscr{O}_{\mathbb{P}^3}$ which is unstable at the its wall.

    \item $dim(A)=[1574]$ and $\mathscr{O}_{\mathbb{P}^3} \hookrightarrow A$

    In this case, we have $$0\to \mathscr{O}_{\mathbb{P}^3} \to A \to F\to 0$$
    where $F\in \mathscr{A}_1$ has dimension vector $[0110]$.
\begin{enumerate}[label=(\alph*)]

    \item If $F\overset{qiso}{\cong}[0\to \mathscr{O}_{\mathbb{P}^3}(-3) \overset{0}{\to} \mathscr{O}_{\mathbb{P}^3}(-2) \to 0]$, then consider the morphism 

    $$A \twoheadrightarrow \mathscr{O}_{\mathbb{P}^3}(-3)[2]\bigoplus \mathscr{O}_{\mathbb{P}^3}(-2)[1] \twoheadrightarrow \mathscr{O}_{\mathbb{P}^3}(-2)[1]$$

    Let $K$ be the kernel of the composition $A\twoheadrightarrow \mathscr{O}_{\mathbb{P}^3}(-2)[1]$, then $K$ fits the following short exact sequence. 

    $$0\to \mathscr{O}_{\mathbb{P}^3}\to K \to \mathscr{O}_{\mathbb{P}^3}(-3)[2] \to 0$$
    $Ext^1(\mathscr{O}_{\mathbb{P}^3}(-3)[2], \mathscr{O}_{\mathbb{P}^3}) = 0$ implies that $K=\mathscr{O}_{\mathbb{P}^3}\bigoplus \mathscr{O}_{\mathbb{P}^3}(-3)[2]$. So $A$ fits into the short exact sequence

    $$0\to \mathscr{O}_{\mathbb{P}^3}\bigoplus \mathscr{O}_{\mathbb{P}^3}(-3)[2] \to A \to \mathscr{O}_{\mathbb{P}^3}(-2)[1] \to 0$$

    A direct computation shows that $\mathscr{O}_{\mathbb{P}^3}(-3)[2]$ destabilizes $A$ at the wall defined by $A$.

    \item If $F\overset{qiso}{\cong}[0\to \mathscr{O}_{\mathbb{P}^3}(-3) \overset{\neq 0}{\to} \mathscr{O}_{\mathbb{P}^3}(-2) \to 0]$, then $F\overset{qiso}{\cong}\mathscr{O}_{\Lambda}(-2)[1]$, where $\Lambda\subset \mathbb{P}^3$ is a plane. 

    $Ext^1(\mathscr{O}_{\Lambda}(-2)[1], \mathscr{O}_{\mathbb{P}^3})=0$ implies that $A=\mathscr{O}_{\Lambda}(-2)[1]\bigoplus \mathscr{O}_{\mathbb{P}^3}$ which is unstable at the wall defined by $A$. 

\end{enumerate}

    \item $dim(A)=[1664]$ and $\mathscr{O}_{\mathbb{P}^3} \hookrightarrow A$

    In this case, we have 
    $$0\to \mathscr{O}_{\mathbb{P}^3} \to A \to \mathscr{O}_{\mathbb{P}^3}^2(-3)[2]\to 0$$
    $Ext^1(\mathscr{O}_{\mathbb{P}^3}^2(-3)[2], \mathscr{O}_{\mathbb{P}^3})=0$ implies $A=\mathscr{O}_{\mathbb{P}^3}^2(-3)[2]\bigoplus \mathscr{O}_{\mathbb{P}^3}$ which is unstable at its wall.

    \item $dim(A)=[1674]$ and $\mathscr{O}_{\mathbb{P}^3} \hookrightarrow A$

    $A$ fits the short exact sequence $$0\to \mathscr{O}_{\mathbb{P}^3} \to A \to F \to 0$$where $dim(F)=[0210]$.

\begin{enumerate}[label=(\alph*)]

    \item If $F$ has a subcomplex $\mathscr{O}_{\mathbb{P}^3}(-3)[2]$ (dimension is $[0100]$), then let $Q$ be the quotient $0\to \mathscr{O}_{\mathbb{P}^3}(-3)[2] \to F \to Q \to 0$ (dim($Q$)=$[0110]$). 

    Let $K$ be the kernel of the composition $A\twoheadrightarrow F \twoheadrightarrow Q$. Similarly, we have $0\to \mathscr{O}_{\mathbb{P}^3}\to K \to \mathscr{O}_{\mathbb{P}^3}(-3)[2] \to 0$, and $K=\mathscr{O}_{\mathbb{P}^3}\bigoplus \mathscr{O}_{\mathbb{P}^3}(-3)[2]$. A direct computations shows that $\mathscr{O}_{\mathbb{P}^3}(-3)[2]$ destabilized $A$ at its wall.

    \item If $\mathscr{O}_{\mathbb{P}^3}(-3)[2]$ is not a subcomplex of $F$, then $F$ is a complex satisfying the short exact sequence 
    $$0\to \mathscr{O}_{\mathbb{P}^3}(-3)[2] \to F \to \mathscr{O}_{L}(-1)[1] \to 0$$where $L\subset \mathbb{P}^3$ is a line. $Ext^1(F, \mathscr{O}_{\mathbb{P}^3})=0$ implies that $A=F\oplus \mathscr{O}_{\mathbb{P}^3}$ which is unstable at the wall.

\end{enumerate}

    \item $dim(A)=[1684]$ and $\mathscr{O}_{\mathbb{P}^3} \hookrightarrow A$

    $A$ fits into the short exact sequence $0\to \mathscr{O}_{\mathbb{P}^3} \to A \to F \to 0 $, where $dim(F)=[0220]$. 

    \begin{enumerate}[label=(\alph*)]

    \item If $F$ contains a subcomplex $F_1$ whose dimension vector is $[0210]$, i.e. $0\to F_1 \to F \to \mathscr{O}_{\mathbb{P}^3}(-2)[1] \to 0$.

    Let $K$ be the kernel of the composition $A \twoheadrightarrow F \twoheadrightarrow \to \mathscr{O}_{\mathbb{P}^3}(-2)[1]$, then $K$ fits into the sequence $0\to \mathscr{O}_{\mathbb{P}^3} \to K \to F_1 \to 0$. $K=F_1\bigoplus \mathscr{O}_{\mathbb{P}^3}$ since the extension class vanishes, and $A$ is destabilized by $F_1$ at the wall.

    \item If $F$ doesn't contain any subcomplexes whose dimension is $[0210]$. Equivalently, $F \cong [0\to \mathscr{O}_{\mathbb{P}^3}^2(-3)\overset{\phi}{\to} \mathscr{O}_{\mathbb{P}^3}^2(-2)\to 0]$ where $\phi$ maps to both copies of $\mathscr{O}_{\mathbb{P}^3}(-2)$. In this situation, $F$ fits into the sequence 
    $$0\to F_1[2] \to  F \to T[1] \to 0$$where $F_1, T \in Coh(\mathbb{P}^3)$ and $T$ is a torsion sheaf. We have $Ext^1(F, \mathscr{O}_{\mathbb{P}^3})=0$ which implies  $A=F\bigoplus \mathscr{O}_{\mathbb{P}^3}$. $A$ is unstable at its wall. 

\end{enumerate}

    \item $dim(A)=[1563]$ and $\mathscr{O}_{\Lambda} \hookrightarrow A$

    In this case, we have $0\to \mathscr{O}_{\mathbb{P}^3} \to A \to \mathscr{O}_{\mathbb{P}^3}(-3)[2] \to 0$. $Ext^1(\mathscr{O}_{\mathbb{P}^3}(-3)[2], \mathscr{O}_{\mathbb{P}^3})=0$ implies $A=\mathscr{O}_{\mathbb{P}^3}(-3)[2]\bigoplus \mathscr{O}_{\mathbb{P}^3}$, and $A$ is unstable at its wall.

\end{enumerate}

\subsubsection{Stability of $Q[1]$ and a description of its quiver moduli.}

$Q[1]\in \mathscr{A}_1$ is the quotient complex in the short exact sequence:

$$W_1: 0\to \mathscr{O}_{\mathbb{P}^3} \to \mathscr{O}_C \to Q[1] \to 0 $$

Its dimension vector is $[0,2,3,0]$ in $\mathscr{A}_1$, so $Q$ is presented by the complex 
$$[0\to \mathscr{O}^2_{\mathbb{P}^3}(-3)\overset{M}{\rightarrow} \mathscr{O}^3_{\mathbb{P}^3}(-2) \to 0]$$ 
where $M\in Hom(\mathbb{C}^2,\mathbb{C}^3)\otimes \mathbb{C}[x_0,...,x_3]_1$. A direct computation shows that  $\lambda_t(\mathscr{O}^3_{\mathbb{P}^3}(-2))> \lambda_t(\mathscr{O}^3_{\mathbb{P}^3}(-3))$ for all $t\in (0,1)$, in which $\lambda_t=-\frac{\chi'_t}{\chi_t}$ denotes the slope function.

So we can equivalently use King's notation of stability of quiver representations (\cite{King_1994}). Let $\theta:=(-3,2)$, and for any subcomplex $F\hookrightarrow Q[1]$ in $\mathscr{A}_1$ whose dimension vector is $\underline{dim}(F)=[0,a,b,0]$, define the pairing $\theta([0,a,b,0]):=(-3)a+(2)b$. $Q[1]$ is stable if $\theta(F)>0$ for any subcomplex $F\hookrightarrow Q[1]$. The moduli space $K^{[2,3]}_{\theta}$ (for simplicity, we denote it by $K_{\theta}$) is a GIT quotient of the representation of the generalized Kronecker quiver $K$ with dimension vector $[2,3]$ and stability condition $\theta$. It's smooth of dimension $12$ from \cite{King_1994} and \cite{MR3803142}.

\[
\begin{tikzcd}
K:\arrow[r,draw=none]&
    \bullet
    \arrow[r, draw=none, "\raisebox{+1.5ex}{\vdots}" description]
    \arrow[r, bend left,        "\alpha_1"]
    \arrow[r, bend right, swap, "\alpha_4"]
    &
    \bullet\\
    \end{tikzcd}
\]

Next, we show the stratification of $K_{\theta}$. Let $x_0,...,x_3$ be the coordinates of $\mathbb{P}^3$, and  $C^{\bullet}:=[0\to \mathscr{O}^2_{\mathbb{P}^3}(-3)\overset{M}{\rightarrow} \mathscr{O}^3_{\mathbb{P}^3}(-2) \to 0]$ be a stable complex. Up to a base change, there are nine possibilities of $M$ that make $C^{\bullet}$ stable:

\begin{center}
    (1) $M=\begin{pmatrix}x_0,x_1,x_2\\x_1,x_2,x_3\end{pmatrix}$
    (2) $M=\begin{pmatrix}x_1,x_0,0\\x_0,x_2,x_3\end{pmatrix}$
    (3) $M=\begin{pmatrix}x_3,x_0,0\\0,x_2,x_1\end{pmatrix}$

    (4) $M=\begin{pmatrix}x_2,x_1,0\\0,x_1,x_0\end{pmatrix}$
    (5) $M=\begin{pmatrix}x_1,x_0,0\\x_3,x_2,x_0\end{pmatrix}$
    (6) $M=\begin{pmatrix}x_3,x_0,0\\0,x_1,x_0\end{pmatrix}$

    (7) $M=\begin{pmatrix}x_1,x_0,0\\x_2,x_1,x_0\end{pmatrix}$
    (8) $M=\begin{pmatrix}x_1,x_0,0\\0,x_1,x_0\end{pmatrix}$
    (9) $M=\begin{pmatrix}x_1,0,x_2\\0,x_1,x_3\end{pmatrix}$
\end{center}
    
Among those matrices, (1) $\sim$ (8) correspond to ideal sheaves of space curves in $\text{Hilb}^{3t+1}_{\mathbb{P}^3}$  (These were also shown in \cite{phdthesis}). The matrix in (9) defines a shifted sheaf $\mathscr{F}[1]$, in which $\mathscr{F}$ contains torsion. $\mathscr{F}$ fits into the short exact sequence: $0\to \mathscr{O}_{\Lambda}(-3)\to \mathscr{F} \to \mathscr{I}_P(-1)\to 0$, where $\Lambda$ is the plane defined by $x_1=0$, and $P$ is the point on $\Lambda$ defined by $x_1=x_2=x_3=0$. Indeed, the set of all matrices of type (9) is the flag variety $Flag:=\{0\subset \mathbb{C} \subset \mathbb{C}^3 \subset \mathbb{C}^4\}$ which is smooth of dimension $5$.

We make the conclusion that there are two strata on $K_{\theta}$: A smooth closed subvariety of dimension $5$ parameterizing sheaves $\mathscr{F}$, and its complement in $K_{\theta}$ parameterizing the space curves with Hilbert polynomial $3t+1$. Correspondingly, there are two general representatives for $Q$ in $\mathscr{A}_1$ such that $Q[1]$ is stable, the ideal sheaf of a space curve $\mathscr{I}_C$ or the sheaf $\mathscr{F}$.

\subsubsection{Stability of the complex  $\mathscr{F}_1$.}

We start by defining the complex $\mathscr{F}_1$ and then show that it is the only stable complex in $\mathscr{A}_1$ with dimension vector $[0231]$. 

\textbf{Define the complex $\mathscr{F}_1$ in $\mathscr{A}_1$.}

We showed that the presentation of the sheaf $\mathscr{F}$ is 

$$\mathscr{O}^2_{\mathbb{P}^3}(-3)\overset{M}{\rightarrow} \mathscr{O}^3_{\mathbb{P}^3}(-2)\cong \mathscr{F}$$
where $M=\begin{pmatrix}x_1,0,x_2\\0,x_1,x_3\end{pmatrix}$. In fact, the complex $\mathscr{O}^2_{\mathbb{P}^3}(-3)\to \mathscr{O}^3_{\mathbb{P}^3}(-2)$ can be extended to the following complex 
$$\mathscr{O}^2_{\mathbb{P}^3}(-3)\overset{M}{\to}\mathscr{O}^3_{\mathbb{P}^3}(-2)\overset{N}{\to}\mathscr{O}_{\mathbb{P}^3}(-1)$$

where $M=\begin{pmatrix}x_1,0,x_2\\0,x_1,x_3\end{pmatrix}$, and $N=\begin{pmatrix}x_2\\x_3\\-x_1\end{pmatrix}$. Define $\mathscr{F}_1$ to be this new complex 

$$\mathscr{F}_1:=\left[\mathscr{O}^2_{\mathbb{P}^3}(-3)\overset{M}{\to}\mathscr{O}^3_{\mathbb{P}^3}(-2)\overset{N}{\to}\mathscr{O}_{\mathbb{P}^3}(-1)\right]$$

Compare $\mathscr{F}_1$ to the presentation (Koszul complex) of the skyscraper sheaf $\mathbb{C}_P$ in $\mathscr{A}_1$: (Without loss of generality, assume $P$ is defined by $x_1=x_2=x_3=0$ in $\mathbb{P}^3$) 

        $$ \mathscr{O}_{\mathbb{P}^3}(-4) \xlongrightarrow{\begin{pmatrix}x_1,-x_2,x_3\end{pmatrix}} \mathscr{O}^3_{\mathbb{P}^3}(-3)\xlongrightarrow{\begin{pmatrix}x_2,0,-x_3\\x_1,-x_3,0\\0,-x_2,x_1\end{pmatrix}} \mathscr{O}^3_{\mathbb{P}^3}(-2)\xlongrightarrow{\begin{pmatrix}x_3\\x_1\\x_2\end{pmatrix}} \mathscr{O}_{\mathbb{P}^3}(-1) \to \mathbb{C}_P$$

        We see that $\mathscr{F}_1$ is indeed a subcomplex of $\mathbb{C}_P$. The quotient complex is $[\mathscr{O}_{\mathbb{P}^3}(-4)\overset{x_1}{\to} \mathscr{O}_{\mathbb{P}^3}(-3)\to 0\to 0] \cong \mathscr{O}_{\Lambda}(-3)[2]$. So we have the short exact sequence  $0\to \mathscr{F}_1 \to \mathbb{C}_P \to \mathscr{O}_{\Lambda}(-3)[2] \to 0$ in $\mathscr{A}_1$, and $\mathscr{F}_1$ is a complex whose cohomologies are:
        $\mathscr{H}^{-1}(\mathscr{F}_1)=\mathscr{O}_{\Lambda}(-3)$, $\mathscr{H}^{0}(\mathscr{F}_1)=\mathbb{C}_P$ and $\mathscr{H}^{i}(\mathscr{F}_1)=0$ for $i\neq -1, 0$.\\

\textbf{Prove that $\mathscr{F}_1$ is the only stable object with the dimension vector $[0231]$.}

    Assume that $G$ is a complex whose dimension vector is $[0231]$ in $\mathscr{A}_1$, and $C^{\bullet}\hookrightarrow G$ is a subcomplex  whose dimension vector is $[0,c,b,a]$ ($c=0,1,2$, $b=0,1,2,3$, $a=0,1$). The following result is from a direct computation:

    \begin{enumerate}
            \item If $a=0$, then $\lambda_t(C^{\bullet})>\lambda_t(G)$ for any $b=0,1,2,3$, $c=0,1,2$ and $t>0.1716$.

            \item If $a=1$, then $\lambda_t(C^{\bullet})<\lambda_t(G)$ for any $b=0,1,2,3$, $c=0,1,2$ and $t>0.1716$.
    \end{enumerate}

    This result shows that for any $t\in(0.1716,1]$ in $\mathscr{A}_1$, $G$ is stable if and only if it doesn't have any subcomplex $C^{\bullet}$ whose dimension vector is $[0,c,b,0]$. 

    More explicitly, if $[0\to \mathscr{O}^2_{\mathbb{P}^3}(-3) \overset{M}{\to} \mathscr{O}^3_{\mathbb{P}^3}(-2) \overset{N}{\to} \mathscr{O}_{\mathbb{P}^3}(-1)]$ is the presentation of $G$ in $\mathscr{A}_1$, where $N=(f_1,f_2,f_3)$ consists of linear functions $f_i$ ($i=1,2,3$), then $f_1,f_2,f_3$ must be linearly independent. 


    Next, we show that the matrix $M$ has to be in the form: $\begin{pmatrix}f_3,0,-f_1\\0,f_3,-f_2\end{pmatrix}$, and this will prove the claim.


    There are two $\mathscr{O}_{\mathbb{P}^3}(-3)$'s mapping to $\mathscr{O}^3_{\mathbb{P}^3}(-2)$, and the morphisms are row vectors in $M$. Assume the first row of $M$ is $(\phi_1,\phi_2,\phi_3)$ ($\phi_i$'s are linear functions), and we have the following diagram:
\[
    \begin{tikzcd}
    {}&\mathscr{O}_{\mathbb{P}^3}(-2)\arrow{dr}{f_1}&{}\\
    \mathscr{O}_{\mathbb{P}^3}(-3) \arrow{ur}{\phi_1}\arrow{r}{\phi_2}\arrow{dr}{\phi_3}&\mathscr{O}_{\mathbb{P}^3}(-2)\arrow{r}{f_2}&\mathscr{O}_{\mathbb{P}^3}(-1)\\
    {}&\mathscr{O}_{\mathbb{P}^3}(-2)\arrow{ur}{f_3}&{}
    \end{tikzcd}
\]

    Firstly we show that $\phi_i$'s are linearly dependent.
    Let $\langle f_1,f_2\rangle$ be the sub vector space of $\mathbb{C}[x_0,x_1,x_2,x_3]_1$ spanned by $f_1, f_2$. Since $G$ is a complex, we must have $\phi_1f_1+\phi_2f_2+\phi_3f_3=0$. Consider the equation mod $\langle f_1,f_2\rangle$, 
    we have $\bar{\phi_3}\bar{f_3}=0$ in the quotient space $\mathbb{C}[x_0,x_1,x_2,x_3]_1/\langle f_1,f_2\rangle$. 
    $f_i$'s are linearly independent, so $\bar{f_3}\neq 0$. This implies $\bar{\phi_3}=0$ and $\phi_3=k_1f_1+k_2f_2$ for some $k_1,k_2\in \mathbb{C}$. 
    The equation now becomes $\phi_1f_1+\phi_2f_2+(k_1f_1+k_2f_2)f_3=0$ which simplifies to $(\phi_1+k_1f_3)f_1+(\phi_2+k_2f_3)f_2=0$. Then we have $f_2|(\phi_1+k_1f_3)f_1$ and $f_1|(\phi_2+k_2f_3)f_2$. 
    $f_i$'s are linearly independent, so we have $f_2|\phi_1+k_1f_3$ and $f_1|\phi_2+k_2f_3$. 
    $\phi_i$'s and $f_i$'s are all linear functions, so there is some $k\in \mathbb{C}$ such that $\phi_2+k_2f_3=kf_1$ and $\phi_1+k_1f_3=-kf_2$. 
    Now $\phi_1=-k_1f_3-kf_2$, $\phi_2=-k_2f_3+kf_1$, $\phi_3=k_1f_1+k_2f_2$, and they satisfy $k_2\phi_1-k_1\phi_2+k\phi_3=0$.
    So $\phi_i$'s are linearly dependent.

    Therefore up to a base change, we may assume $\phi_3=0$, $\phi_1=f_2$ and $\phi_2=-f_1$. The presentation of $G$ becomes:
\[
        \begin{tikzcd}
        \mathscr{O}_{\mathbb{P}^3}(-3)\arrow{r}{f_2}\arrow{dr}[swap]{\phi_1}&\mathscr{O}_{\mathbb{P}^3}(-2)\arrow{dr}{f_1}&{}\\
        \mathscr{O}_{\mathbb{P}^3}(-3) \arrow{ur}[swap]{-f_1}\arrow{r}{\phi_2}\arrow{dr}{\phi_3}&\mathscr{O}_{\mathbb{P}^3}(-2)\arrow{r}{f_2}&\mathscr{O}_{\mathbb{P}^3}(-1)\\
        {}&\mathscr{O}_{\mathbb{P}^3}(-2)\arrow{ur}{f_3}&{}
        \end{tikzcd}
\]

         $\phi_3\neq 0$ in the diagram.
        Otherwise, $(\phi_1,\phi_2)=c(f_2,-f_1)$ for some $c\in \mathbb{C}$. 
        The map from the second  $\mathscr{O}_{\mathbb{P}^3}(-3)$ to $\mathscr{O}^3_{\mathbb{P}^3}(-2)$ will be $0$ by a base change.
        This implies $[ 0\to \mathscr{O}_{\mathbb{P}^3}(-3)\to 0 \to 0]$ is a subcomplex of $G$, making $G$ unstable. 
        
        If none of those $\phi_i$'s is zero, then use $\phi_2$ and $\phi_3$ to eliminate $\phi_1$ by a base change. The presentation becomes:
\[
        \begin{tikzcd}
        \mathscr{O}_{\mathbb{P}^3}(-3)\arrow{r}{f_2}\arrow{dr}{-f_1}&\mathscr{O}_{\mathbb{P}^3}(-2)\arrow{dr}{f_1}&{}\\
        \mathscr{O}_{\mathbb{P}^3}(-3)\arrow{r}{-f_3}\arrow{dr}{f_2}&\mathscr{O}_{\mathbb{P}^3}(-2)\arrow{r}{f_2}&\mathscr{O}_{\mathbb{P}^3}(-1)\\
        {}&\mathscr{O}_{\mathbb{P}^3}(-2)\arrow{ur}{f_3}&{}
        \end{tikzcd}
\]

This diagram is exactly the presentation of $\mathscr{F}_1$, and we prove the claim.

\subsection{Walls for the dual class.} Lastly, in this section, we show the walls for the dual class of $v=(0,0,3,-5)$. By definition (section 5), the dual class $v^{\vee}$ is $(0,0,3,-4)$, and its Hilbert polynomial is $P(t) = 3t+2$.

We recall that the two general Gieseker stable sheaves with Hilbert polynomial $P(t)=3t+2$ are: (1) $E=\mathscr{O}_{C}(P)$, where $C$ is a space cubic curve and $P$ is a point on $C$. (2) $E=L_{C_E}$, which is a degree $2$ line bundle on a plane cubic curve $C_E$. 
Their walls in $\mathscr{A}_t$ are given by the short exact sequences $W_1'$ and $W_2'$ below. They are in fact defined by the derived dual of $W_1$ and $W_2$ for the class $3t+1$.

$$W_1': \quad 0\to \left[\mathscr{O}^3_{\mathbb{P}^3}(-1)\to \mathscr{O}^2_{\mathbb{P}^3} \right] \to E \to \mathscr{O}_{\mathbb{P}^3}(-3)[2] \to 0$$
where the complex $\left[\mathscr{O}^3_{\mathbb{P}^3}(-1)\to \mathscr{O}^2_{\mathbb{P}^3} \right]$ is the subobject of $E$.

$$W_2': \quad  0\to \mathscr{I}_{P/\Lambda}(1)\to E \to \mathscr{O}_{\Lambda}(-2)[1]\to 0$$
where $P$ is a point in the plane $\Lambda$.

The duality results show that $W_1'$ and $W_2'$ are the mirror images of $W_1$ and $W_2$ in the $(t,u)$ plane. We have shown that $W_1$ and $W_2$ are actual walls at their right endpoints, so $W_1'$ and $W_2'$ must be actual walls at their left endpoints (in $\mathscr{A}_{-1}$).

Suppose Proposition \ref{boundedness of walls} were proved true, then $W_1$ (resp. $W_1'$) and $W_2$ (resp. $W_2'$) would be actual walls everywhere in the $(t,u)$ plane once we prove that they are actual walls at the other endpoint. This is because there is no possible intersecting of these walls.

For the rest of this section, we prove the stability for $\mathscr{O}_{\Lambda}(-2)[1]$ and $\mathscr{I}_{P/\Lambda}(1)$ in $\mathscr{A}_1$. This will imply that $W_2'$ is an actual wall at the right endpoint. 

The dimension vector of $\mathscr{O}_{\Lambda}(-2)[1]$ is $[0,1,1,0]$, and $[0,0,1,0]$ is the only non-trivial subcomplex. It is straightforward to check that $\mathscr{O}_{\Lambda}(-2)[1]$ is stable in $\mathscr{A}_1$ for $t\in (0,1)$.

The dimension vector of $\mathscr{I}_{P/\Lambda}(1)$ is $[2,8,11,5]$. Without loss of generality, assume that the coordinates of $\mathbb{P}^3$ are $x,y,z,w$, $\Lambda$ is defined by $\{x=0\}$ and $P$ is defined by $\{x=y=z=0\}$. The presentation of $\mathscr{I}_{P/\Lambda}(1)$ is as follows:

$$\mathscr{O}^2_{\mathbb{P}^3}(-4)\xrightarrow{M}\mathscr{O}^8_{\mathbb{P}^3}(-3)\xrightarrow{N}\mathscr{O}^{11}_{\mathbb{P}^3}(-2)\xrightarrow{S}\mathscr{O}^5_{\mathbb{P}^3}(-1)\xrightarrow{T}\mathscr{I}_{P/\Lambda}(1)$$

The stability of $\mathscr{I}_{P/\Lambda}(1)$ follows from a direct computation that the slopes of all its subcomplexes are smaller than the slope of $\mathscr{I}_{P/\Lambda}(1)$ at the wall. See Appendix \ref{presentation and subcomplexes} for matrices $M,N,S$ and the dimension vector of all the subcomplexes.

\section{The wall-crossings for the class $3t+1$}

In this section, we study the wall-crossings for the class $v=(0,0,3,-5)$. The moduli space in the last chamber in $\mathscr{A}_1$ turns out to be the Gieseker moduli space. This gives some clue that the last wall in $\mathscr{A}_1$ is expected to be the last wall for all $t\in \mathbb{R}$, and the unbounded chamber containing $t>>0$ is the Gieseker chamber. The main technique we use is the elementary modification. A similar process can be found in \cite{MR3803142} and \cite{MR2998828}.

In section $6$, we found two actual walls in $\mathscr{A}_1$ for $v$. These are $W_1: 0\to \mathscr{O}_{\mathbb{P}^3}\to E \to Q[1] \to 0$ at $t=0.35$, and $W_2: 0\to \mathscr{O}_{\Lambda}\to E \to \mathscr{F}_1 \to 0$ at $t=0.72$. Denote the three chambers in $\mathscr{A}_1$ by $C_1:= \{t\in (0,0.35)\}$, $C_2:= \{t\in (0.35,0.72)\}$ and $C_3:= \{t\in (0.72,1]\}$.

\subsection {Moduli space $\mathscr{M}_{1}$ in $C_1$}

The moduli space in $t\in (0,0.35)$ is empty since every object $E$ is destabilized by $\mathscr{O}_{\mathbb{P}^3} \to E$. The existence of such a map is given by Serre Duality as shown in section $6$.

\subsection{Moduli space $\mathscr{M}_{W_1}$ at the first wall $W_1$}

$W_1$ is defined by $0\to \mathscr{O}_{\mathbb{P}^3}\to E \to Q[1] \to 0$. The moduli space at $W_1$ of $E$ is the same with the moduli of $Q[1]$ which is the Kronecker moduli space $K_{(2,3)}:=K^{[2,3]}_{\theta}$ ($\theta=(-3,2)$ defines the stability condition). It is a smooth variety of dimension $12$.

\subsection {Moduli space $\mathscr{M}_{2}$ in $C_2$.}

For $t\in (0.35,0.72)$. Recall that the quotient $Q[1]$ has two general representatives which are stable: (1) $Q=\mathscr{I}_{C}$ or (2) $Q=\mathscr{F}$, and the locus in $K_{(2,3)}$ parameterizing $\mathscr{F}$ is a smooth $5-$dimensional flag variety. Denote this locus by $H$. A direct computation shows that $Ext^1(\mathscr{I}_{C}[1],\mathscr{O}_{\mathbb{P}^3})=\mathbb{C}$ and $Ext^1(\mathscr{F}[1],\mathscr{O}_{\mathbb{P}^3})=\mathbb{C}^4$. This implies that $\mathscr{M}_{2}$ is isomorphic to $ \mathscr{M}_{W_1}$ outside $H$, and a $\mathbb{P}^3$ bundle over $H$. Denote this $\mathbb{P}^3$ bundle by $\mathscr{M}_{\mathscr{F}}$.

\subsection{Moduli space $\mathscr{M}_{W_2}$ at the second wall.}

$W_2$ is defined by 
$0\to \mathscr{O}_{\Lambda}\to E \to \mathscr{F}_1 \to 0$.

There are two strata on the moduli space $\mathscr{M}_{2}$ which are $K_{(2,3)}\backslash H$ and $M_{\mathscr{F}}$. $K_{(2,3)}\backslash H$ parameterizes the structure sheaf of space cubic curves $C$, and $M_{\mathscr{F}}$ parameterizes those objects $E$ which fit into the short exact sequence $0\to \mathscr{O}_{\mathbb{P}^3}\to E\to \mathscr{F}[1] \to 0$. 
Objects $E$ from $M_{\mathscr{F}}$ satisfy the following sequence as well 
$$0\to \mathscr{F}_1\to E \to \mathscr{O}_{\Lambda'} \to 0$$ 

In the sequence, $\mathscr{F}_1$ is the complex defined in the last section, and $\mathscr{F}_1$ corresponds to a point $(P,\Lambda)$ of the Flag variety $H$. $\Lambda'\subset \mathbb{P}^3$ is a plane but not necessarily the same with the plane $\Lambda$ encoded in $\mathscr{F}_1$. A direct computation shows that
\begin{equation*}
\left\{
\begin{array}{ccc}    
Ext^1(\mathscr{F}_1, \mathscr{O}_{\Lambda})=\mathbb{C}^9, & Ext^1(\mathscr{O}_{\Lambda}, \mathscr{F}_1)=\mathbb{C} & {}\\
Ext^1(\mathscr{F}_1, \mathscr{O}_{\Lambda'})=0, &
Ext^1(\mathscr{O}_{\Lambda'}, \mathscr{F}_1)=\mathbb{C} & \text{if} \quad \Lambda'\neq \Lambda\\
\end{array}
\right.
\end{equation*} 

This implies that $\mathscr{M}_{W_2}=\mathscr{M}_{2}$. Objects from the stratum $K_{(2,3)}\backslash H$ stay stable at the wall while objects from $\mathscr{M}_{\mathscr{F}}$ become semi-stable.

\subsection {Moduli space $\mathscr{M}_{3}$ in  $C_3$} 

From the extension classes 
\begin{equation*}
\left\{
\begin{array}{ccc}    
Ext^1(\mathscr{F}_1, \mathscr{O}_{\Lambda})=\mathbb{C}^9 &\\
Ext^1(\mathscr{F}_1, \mathscr{O}_{\Lambda'})=0 & \text{if} \quad \Lambda'\neq \Lambda\\
\end{array}
\right.
\end{equation*} 
we see that extensions from the sequence $0\to \mathscr{O}_{\Lambda'}\to E \to \mathscr{F}_1 \to 0 $ are not stable in $C_3$. The new stable objects are from the extension $0\to \mathscr{O}_{\Lambda}\to E \to \mathscr{F}_1 \to 0$, in which $\Lambda$ is the same with the one encoded in the complex $\mathscr{F}_1$.

So when crossing the second wall $W_2$ from $C_1$ to $C_2$, the stratum $K_{(2,3)}\backslash H$ stays, and the $\mathbb{P}^3$ bundle $\mathscr{M}_{\mathscr{F}}$ disappears with only the base $H$ remaining. $H$ then becomes a $\mathbb{P}^8$ bundle over $H$ from the above computation. Denote this bundle by $\mathbf{P}$. We will next study this $\mathbb{P}^8$ bundle $\mathbf{P}$, and then glue it to $K_{(2,3)}\backslash H$ using the elementary modification. The resultant moduli space $\mathscr{M}_{C_3}$ turns out to be the Gieseker moduli space, denoted by $\mathscr{M}^{3t+1}_{\mathbb{P}^3}$.

\subsubsection{A description of \ $\mathbf{P}$.}

We show in this subsection that $\mathbf{P}$ is the fibered space over $\mathbb{P}^{3\vee}$ whose fibers are $\mathscr{M}^{3t+1}_{\mathbb{P}^2}$.

We have shown that a complex $\mathscr{F}_1$ corresponds to a point of the flag variety: $\{P\in \Lambda \subset \mathbb{P}^3\}$. Indeed, this flag variety is the same with $H$ since $\mathscr{F}_1$ is the unique extension of $\mathscr{F}$. 
Without loss of generality, we fix a complex $\mathscr{F}_1$, in which $\Lambda$ is defined by $x_1=0$ and the point $P$ is defined by $x_1=x_2=x_3=0$. We will show that the vector space $Ext^1(\mathscr{F}_1,\mathscr{O}_{\Lambda})$ (up to a scalar multiplication) parameterizes plane cubic curves in $\Lambda$ which go through $P$. 

Let $\Lambda'$ be an arbitrary plane in $\mathbb{P}^3$, and we have the extension groups: 

\begin{equation*}
  Ext^1(\mathscr{F}_1,O_{\Lambda'})=\left\{
  \begin{array}{ll}
    0, & \text{if} \Lambda'\neq \Lambda \\
    \mathbb{C}^9, & \text{if} \Lambda' = \Lambda
  \end{array}\right.
\end{equation*}

In the short exact sequence: $0\to \mathscr{F}_1 \to \mathbb{C}_P \to \mathscr{O}_{\Lambda}(-3)[2]\to 0 $, the plane $\Lambda$ encoded in $\mathscr{F}_1$ is in fact the same with the plane in the quotient object $\mathscr{O}_{\Lambda}(-3)[2]$. 

Apply the functor $Hom(-,\mathscr{O}_{\Lambda})$ to the sequence: $0\to \mathscr{F}_1 \to \mathbb{C}_P \to \mathscr{O}_{\Lambda}(-3)[2]\to 0 $, we have a long exact sequence of cohomologies:

$$0\to Ext^1(\mathscr{F}_1,\mathscr{O}_{\Lambda})=\mathbb{C}^9 \to Ext^2(\mathscr{O}_{\Lambda}(-3)[2],\mathscr{O}_{\Lambda})=\mathbb{C}^{10} \overset{\phi}{\to} Ext^2(\mathbb{C}_P,\mathscr{O}_{\Lambda})=\mathbb{C}\to ... $$
in which $Ext^1(\mathscr{F}_1,\mathscr{O}_{\Lambda})$ is the kernal of $\phi$.

In the diagram, $Ext^2(\mathscr{O}_{\Lambda}(-3)[2],\mathscr{O}_{\Lambda})=Hom(\mathscr{O}_{\Lambda}(-3),\mathscr{O}_{\Lambda})$, and this can be computed by the resolution of $\mathscr{O}_{\Lambda}(-3)$: $[\mathscr{O}_{\mathbb{P}^3}(-4)\overset{x_1}{\to} \mathscr{O}_{\mathbb{P}^3}(-3)]\to \mathscr{O}_{\Lambda}(-3)$. Apply the functor $Hom(-,\mathscr{O}_{\Lambda})$ to this resolution, and we have 

$$0\to Hom(\mathscr{O}_{\Lambda}(-3),\mathscr{O}_{\Lambda})\to Hom(\mathscr{O}_{\mathbb{P}^3}(-3),\mathscr{O}_{\Lambda})=\mathbb{C}^{10} \xrightarrow[=0]{x_1} 
Hom(\mathscr{O}_{\mathbb{P}^3}(-4),\mathscr{O}_{\Lambda}) =\mathbb{C}^{15}$$
which implies $Hom(\mathscr{O}_{\Lambda}(-3),\mathscr{O}_{\Lambda})=\mathbb{C}^{10}$.

Apply the functor $Hom(-,\mathscr{O}_{\Lambda})$ to the Koszul resolution of $\mathbb{C}_P$, the cohomology at "$\mathscr{O}_{\mathbb{P}^3}(-3)$" gives  $Ext^2(\mathbb{C}_P,\mathscr{O}_{\Lambda})=\mathbb{C}$. By definition,  $Ext^2(\mathbb{C}_P,\mathscr{O}_{\Lambda})={Ker(\alpha)}/{Im(\beta)}$ in the complex:

$$Hom(\mathscr{O}_{\mathbb{P}^3}(-4),\mathscr{O}_{\Lambda})\xleftarrow{\alpha} Hom(\mathscr{O}^3_{\mathbb{P}^3}(-3),  \mathscr{O}_{\Lambda})\xleftarrow{\beta} Hom(\mathscr{O}^3_{\mathbb{P}^3}(-2),\mathscr{O}_{\Lambda})$$
 
Back to the exact sequence 
$$0\to Ext^1(\mathscr{F}_1,\mathscr{O}_{\Lambda})=\mathbb{C}^9 \to Ext^2(\mathscr{O}_{\Lambda}(-3)[2],\mathscr{O}_{\Lambda})=\mathbb{C}^{10} \overset{\phi}{\to} \mathbb{C}=Ext^2(\mathbb{C}_P,\mathscr{O}_{\Lambda})=\frac{Ker(\alpha)}{Im(\beta)} \to 0$$ 
we have that
$Ker(\phi)= Im(\beta)$.

$Im(\beta)$ is indeed the set $\{x_2Q_2+x_3Q_3\}$,
where $Q_2, Q_3 \in Hom(\mathscr{O}_{\mathbb{P}^3}(-2),\mathscr{O}_{\Lambda})=H^0(\Lambda, \mathscr{O}_{\Lambda}(2))$ are two quadric curves in $\Lambda$. This is because $x_1=0$ on $\Lambda$, and the non-zero objects in $Im(\beta)$ are computed by applying $Hom(-,\mathscr{O}_{\Lambda})$ to the diagram in Figure \ref{part of Koszul} (this is part of the Koszul resolution of $\mathbb{C}_P$).
The set $Im(\beta)=\left\{ x_2Q_2+x_3Q_3|Q_2,Q_3\in H^0(\Lambda,\mathscr{O}_{\Lambda}(2))\right\}$ consists of sections in $H^0(\Lambda,\mathscr{O}_{\Lambda}(3))$ which vanish at the intersection of the loci of $x_2$ and $x_3$. This is exactly the set of all cubic curve in $\Lambda$ that go through $P$.

\begin{figure}
    \centering
        \begin{tikzcd}
        {}&\mathscr{O}_{\mathbb{P}^3}(-3)\arrow{r}{x_2}\arrow{dr}{x_3}&\mathscr{O}_{\mathbb{P}^3}(-2)\\
        \mathscr{O}_{\mathbb{P}^3}(-4) \arrow{ur}{x_1}\arrow{r}{-x_2}\arrow{dr}{-x_3}&\mathscr{O}_{\mathbb{P}^3}(-3)&\mathscr{O}_{\mathbb{P}^3}(-2)\\
        {}&\mathscr{O}_{\mathbb{P}^3}(-3)&\mathscr{O}_{\mathbb{P}^3}(-2)
        \end{tikzcd}
        \caption{}
    \label{part of Koszul}
\end{figure}

So we have the following morphism:

$$\mathbf{P} \to H=\left\{ P\in \Lambda \subset \mathbb{P}^3 \right\}$$
in which the fiber at a point $(P,\Lambda)\in H$ parameterizes all the plane cubic curves in $\Lambda$ that go through $P$.

Moreover, consider the morphisms 
$$\mathbf{P}\to H=\left\{ P\in \Lambda \subset \mathbb{P}^3 \right\}\to \left\{ \Lambda \subset \mathbb{P}^3 \right\}=\mathbb{P}^{3\vee}$$

A point in $\mathbb{P}^{3\vee}$ corresponds to a plane $\Lambda \subset \mathbb{P}^3$, and the fiber over it in $\mathbf{P}$ parameterizes the pair $\left\{ C,P \right\}$, where $C\subset \Lambda$ is a plane cubic curve passing through the point $P\in \Lambda$. So this fiber is the universal cubic curve $\mathscr{C}\subset |H^0(\mathbb{P}^2, \mathscr{O}_{\mathbb{P}^2}(3))|\times \mathbb{P}^2$ which is the Gieseker moduli space $\mathscr{M}^{3t+1}_{\mathbb{P}^2}$ (\cite{le1993faisceaux}). 

This proves the claim that $\mathbf{P}$ is fibered over $\mathbb{P}^{3\vee}$ with fibers $\mathscr{M}^{3t+1}_{\mathbb{P}^2}$. It also matches the result in \cite{Freiermuth_2004} that $\mathbf{P}$ is a component of the Gieseker moduli space $\mathscr{M}^{3t+1}_{\mathbb{P}^3}$ parameterizing degree one line bundles on plane curves.

\subsubsection{The elementary modification}

We have shown that when crossing the second wall $W_2$, the component $\mathscr{M}_{\mathscr{F}}$ disappears, and its base $H$ is replaced by $\mathbf{P}$. It's known (\cite{Freiermuth_2004}) that $K_{(2,3)}\backslash H$ and $\mathbf{P}$ are the components of the Gieseker moduli space $\mathscr{M}^{3t+1}_{\mathbb{P}^3}$. So it is expected that $\mathbf{P}$ is glued to $K_{(2,3)}$ along the exceptional divisor of its blow-up $\mathbf{B}:=Bl_H{(K_{(2,3)})}$ ) (\cite{Freiermuth_2004}, \cite{MR3803142}).

We need some computational results from a few commutative diagrams in this subsection. We will show one of the diagrams in the next lemma, and see Appendix \ref{auxiliary diagrams} for the rest of the diagrams.

\begin{Lemma}{\label{commutative diagrams}}

\begin{enumerate}
\item For $0\to \mathscr{O}_{\mathbb{P}^3} \to E \to \mathscr{I}_C[1] \to 0$ ($C \in \mathbb{P}^3$ is a twisted cubic curve.) the following diagram is commutative.
\begin{footnotesize}
\[
\begin{tikzcd}
0 \arrow{r}\arrow{d}&\begin{matrix}Hom(\mathscr{O}_{\mathbb{P}^3},\mathscr{O}_{\mathbb{P}^3})\\=\mathbb{C}\end{matrix}\arrow{r}\arrow{d}{\cong}&\begin{matrix}Hom(\mathscr{O}_{\mathbb{P}^3},E)\\=\mathbb{C}\end{matrix}\arrow{r}\arrow{d}&\begin{matrix}Hom(\mathscr{O}_{\mathbb{P}^3},\mathscr{I}_C[1])\\=0\end{matrix}\arrow{r}\arrow{d}&\begin{matrix}Ext^1(\mathscr{O}_{\mathbb{P}^3},\mathscr{O}_{\mathbb{P}^3})\\=0\end{matrix}\arrow{d}\\
\begin{matrix}Hom(\mathscr{I}_C[1],\mathscr{I}_C[1])\\=\mathbb{C}\end{matrix}\arrow{r}{\cong}\arrow{d}&\begin{matrix}Ext^1(\mathscr{I}_C[1],\mathscr{O}_{\mathbb{P}^3})\\=\mathbb{C}\end{matrix}\arrow{r}\arrow{d}&\begin{matrix}Ext^1(\mathscr{I}_C[1],E)\\=\mathbb{C}^{12}\end{matrix}\arrow{r}\arrow{d}&\begin{matrix}Ext^1(\mathscr{I}_C[1],\mathscr{I}_C[1])\\=\mathbb{C}^{12}\end{matrix}\arrow{r}{=0}\arrow{d}&\begin{matrix}Ext^2(\mathscr{I}_C[1],\mathscr{O}_{\mathbb{P}^3})\\=\mathbb{C}^{11}\end{matrix}\arrow{d}\\
\begin{matrix}Hom(E,\mathscr{I}_C[1])\\=\mathbb{C}\end{matrix}\arrow{r}\arrow{d}&\begin{matrix}Ext^1(E,\mathscr{O}_{\mathbb{P}^3})\\=0\end{matrix}\arrow{r}\arrow{d}&\begin{matrix}Ext^1(E,E)\arrow{r}\arrow{d}\\=\mathbb{C}^{12}\end{matrix}&\begin{matrix}Ext^1(E,\mathscr{I}_C[1])\\=\mathbb{C}^{12}\end{matrix}\arrow{r}{=0}\arrow{d}&\begin{matrix}Ext^2(E,\mathscr{O}_{\mathbb{P}^3})\\=\mathbb{C}\end{matrix}\arrow{d}\\
\begin{matrix}Hom(\mathscr{O}_{\mathbb{P}^3},\mathscr{I}_C[1])\\=0\end{matrix}\arrow{r}\arrow{d}&\begin{matrix}Ext^1(\mathscr{O}_{\mathbb{P}^3},\mathscr{O}_{\mathbb{P}^3})\\=0\end{matrix}\arrow{r}\arrow{d}&\begin{matrix}Ext^1(\mathscr{O}_{\mathbb{P}^3},E)\\=0\end{matrix}\arrow{r}\arrow{d}&\begin{matrix}Ext^1(\mathscr{O}_{\mathbb{P}^3},\mathscr{I}_C[1])\\=0\end{matrix}\arrow{r}\arrow{d}&\begin{matrix}Ext^2(\mathscr{O}_{\mathbb{P}^3},\mathscr{O}_{\mathbb{P}^3})\\=0\end{matrix}\arrow{d}\\
\begin{matrix}Ext^1(\mathscr{I}_C[1],\mathscr{I}_C[1])\\=\mathbb{C}^{12}\end{matrix}\arrow{r}{=0}&\begin{matrix}Ext^2(\mathscr{I}_C[1],\mathscr{O}_{\mathbb{P}^3})\\=\mathbb{C}^{11}\end{matrix}\arrow{r}&\begin{matrix}Ext^2(\mathscr{I}_C[1],E)\\=\mathbb{C}^{11}\end{matrix}\arrow{r}&\begin{matrix}Ext^2(\mathscr{I}_C[1],\mathscr{I}_C[1])\\=0\end{matrix}\arrow{r}&\begin{matrix}Ext^3(\mathscr{I}_C[1],\mathscr{O}_{\mathbb{P}^3})\\=0\end{matrix}\\
\end{tikzcd}
\]
\end{footnotesize}

\item
For any plane $\Lambda' \subset \mathbb{P}^3$

$Hom(\mathscr{O}_{\Lambda'},\mathscr{F}_1) = Ext^2(\mathscr{O}_{\Lambda'},\mathscr{F}_1) = Ext^3(\mathscr{O}_{\Lambda'},\mathscr{F}_1) = 0$, $Ext^1(\mathscr{O}_{\Lambda'},\mathscr{F}_1)=\mathbb{C}^{1}$
\item 
$Hom(\mathscr{F}_1,\mathscr{O}_{\Lambda})=Ext^3(\mathscr{F}_1,\mathscr{O}_{\Lambda}))=0$, $Ext^1(\mathscr{F}_1,\mathscr{O}_{\Lambda}))=\mathbb{C}^{9}$, $Ext^2(\mathscr{F}_1,\mathscr{O}_{\Lambda}))=\mathbb{C}^{14}$

\end{enumerate}

\end{Lemma}

\textbf{Some notations}:

$\bullet$ $\mathbf{B}:=Bl_H(K_{(2,3)})\xrightarrow{b} K_{(2,3)}$ where the morphism is denoted by $b$. Let $D$ be the exceptional divisor, and $b_H$: $D\xrightarrow{b_H} H$ be the restriction of $b$ to the exceptional divisor.

$\bullet$ $\pi_H$, $\pi_{\mathbf{P}}$ and $\pi_D$ denote the projections: $H\times \mathbb{P}^3 \xrightarrow{\pi_H} H$, $\mathbf{P}\times \mathbb{P}^3 \xrightarrow{\pi_{\mathbf{P}}} \mathbf{P}$, $D\times \mathbb{P}^3\xrightarrow{\pi_D} D$. 

$\bullet$ $p$, $q$ are the projections: $H\xrightarrow{p} \mathbb{P}^{3\vee}$, (where $(P\in \Lambda)\mapsto \Lambda$), $\mathbf{P}\xrightarrow{q} H$.

$\bullet$ $i,j$ are the inclusions: $D \times \mathbb{P}^3\xrightarrow{i} \mathbf{B}\times \mathbb{P}^3$, $D\times \mathbb{P}^3\xrightarrow{j} \mathbf{P}\times \mathbb{P}^3$.

$\bullet$ Two universal families: 
\begin{enumerate}
    \item $U_{\mathscr{F}_1}$ on $H\times \mathbb{P}^3$ as the universal family of complexes $\mathscr{F}_1$. 
    \item $U_{\mathscr{O}_{\Lambda}}$ on $\mathbb{P}^{3\vee}\times \mathbb{P}^3$ as the universal family of planes in $\mathbb{P}^3$.
\end{enumerate}

\begin{Prop}{\label{extension1}}
There exists a universal family of extensions on $H$ of the form $$0\to U_{\mathscr{F}_1}\otimes \pi_H^* L\to U_{E} \to p^*( U_{\mathscr{O}_{\Lambda}})\to 0$$

where $L:=\mathscr{E}xt_{\pi_H}^1(p^*( U_{\mathscr{O}_{\Lambda}}), U_{\mathscr{F}_1})^*$ is a line bundle on $H$.
\end{Prop}

\begin{proof}
Let $L$ be the line bundle $L:=\mathscr{E}xt_{\pi_H}^1(p^*(U_{\mathscr{O}_{\Lambda}}),U_{\mathscr{F}_1})^*$ on $H$. From the part (2) in Lemma \ref{commutative diagrams} and our assumption, we have that 

$R\mathscr{H}om(p^*(U_{\mathscr{O}_{\Lambda}}), (U_{\mathscr{F}_1}\otimes L)[1])=R\mathscr{H}om((p^*(U_{\mathscr{O}_{\Lambda}})[-1], U_{\mathscr{F}_1}\otimes L)$ is a sheaf. 

There is then a canonical identity element:

$
\begin{array}{rl}
id & \in  H^0(H\times \mathbb{P}^3, R\mathscr{H}om(p^*U_{\mathscr{O}_{\Lambda}}[-1], U_{\mathscr{F}_1}\otimes \pi_{H*} L))\\
{}&=H^0(H\times \mathbb{P}^3, p^*U_{\mathscr{O}_{\Lambda}}^*\otimes L\otimes U_{\mathscr{F}_1}[1])\\
{} & = H^0(H,\pi_{H*}(p^*U_{\mathscr{O}_{\Lambda}}^*)\otimes \pi_{H*}(p^*U_{\mathscr{O}_{\Lambda}}^*\otimes U_{\mathscr{F}_1}[1])^*\otimes \pi_{H*}U_{\mathscr{F}_1}[1])\\
\end{array}
$

which gives the morphism $f_{id}:$

$$\to p^*(U_{\mathscr{O}_{\Lambda}})[-1]\xrightarrow{f_{id}} U_{\mathscr{F}_1}\otimes \pi^*_{H}L \to U_E\to $$

The cone $U_E$ from the triangle is the universal extension we want.
\end{proof}

There is a universal extension on $\mathbf{P}$ as well which follows from a similar construction.

\begin{Prop}{\label{extension2}}
There exists a universal family of extensions on $\mathbf{P}$ of the form 
$$0\to q^*(p^*U_{\mathscr{O}_{\Lambda}})\otimes \pi_P^*\mathscr{O}_{\mathbf{P}}(1)\to U_F \to q^*U_{\mathscr{F}_1}\to 0 $$
\end{Prop}

Next, we show that $D$ is embedded into $\mathbf{P}$.

The locus $H$ parameterizes objects $E$ which fit into the short exact sequence: $0\to \mathscr{F}_1 \to E \to \mathscr{O}_{\Lambda} \to 0$. 
The results in diagram (2) from Appendix \ref{auxiliary diagrams} gives the following diagram: 
\[
\begin{tikzcd}
&&&\mathbb{C}^3\arrow{d}&\\
0\arrow{r}&T_{E|H}=\mathbb{C}^5\arrow{r}\arrow{d}&T_{E|\mathscr{M}_2}=\mathbb{C}^{15}\arrow{r}\arrow{d}{\cong}&N_{H|\mathscr{M}_2}=\mathbb{C}^{10}\arrow{r}\arrow{d}&{0}\\
0\arrow{r}&Ker(\phi)=\mathbb{C}^8\arrow{r}\arrow{d}&Ext^1(E,E)=\mathbb{C}^{15}\arrow{r}{\phi}&Ext^1(\mathscr{F}_1,\mathscr{O}_{\Lambda})=\mathbb{C}^{9}\arrow{r}&{}\\
{}&\mathbb{C}^3&&&
\end{tikzcd}
\]
in which $T_{E|H}$ denotes the tangent bundle at the point $E$ in $H$, and $N_{H|\mathscr{M}_1}$ denotes the normal bundle of $H$ in $\mathscr{M}_1$ (at $E$).

The global version of the diagram is the following:

\[
\begin{tikzcd}
0\arrow{r}&\mathscr{T}_{H|K_{(2,3)}}\arrow{r}\arrow{d}&\mathscr{T}_{\mathscr{M}_2}\arrow{r}\arrow{d}{KS}&\mathscr{N}_{H|\mathscr{M}_2}\arrow{r}\arrow{d}&{0}\\
0\arrow{r}&\mathscr{K}er\arrow{r}&\mathscr{E}xt_{\pi_H}^1(U_E,U_E)\arrow{r}&\mathbf{P}=\mathscr{E}xt_{\pi_H}^1(U_{\mathscr{F}_1}\otimes \pi_H^*L,U_{\mathscr{O}_{\Lambda}})\arrow{r}&{}\\
\end{tikzcd}
\]

This implies that $N_{H|K_{(2,3)}}=\mathbb{C}^7\hookrightarrow Ext^1(\mathscr{F}_1,\mathscr{O}_{\Lambda})=\mathbb{C}^{9}$ for all points $E\in H$. Correspondingly, we have the embedding 
$D=\mathbb{P}(\mathscr{N}_{H|K_{(2,3)}}^*)\hookrightarrow \mathbf{P}:=\mathbb{P}(\mathscr{E}xt_{\pi_H}^1(U_{\mathscr{F}_1}\otimes \pi_H^*L,U_{\mathscr{O}_{\Lambda}})^*)$.

The following result shows that the dimension of the moduli space $\mathscr{M}_3$ along the exceptional divisor $D$ is one more than the dimension of $\mathbf{P}\backslash D$. This indicates the expected gluing at least set theoretically.

\begin{Prop}
If $F\in \mathbb{P}(\mathscr{N^*})$, then the morphism $Ext^1(F,F)\to Ext^1(\mathscr{O}_{\Lambda},\mathscr{F}_1)=\mathbb{C}$ (from diagram (4) in Appendix \ref{auxiliary diagrams}) is non zero. If $F\in \mathbf{P}\backslash \mathbb{P}(\mathscr{N^*})$, then the morphism  $Ext^1(F,F)\to Ext^1(\mathscr{O}_{\Lambda},\mathscr{F}_1)=\mathbb{C}$ is zero.
\end{Prop}

\begin{proof}
Same with Prop $4.10$ in \cite{MR3803142} using the results in diagram (4) in Appendix \ref{auxiliary diagrams}.

\end{proof}

Finally in this section, we show the gluing of $\mathbf{B}$ and $\mathbf{P}$ using the Elementary modification. 

\begin{enumerate}
    \item {Construct a universal family $\mathscr{K}$ on the blow up.}

There are three distinguished triangles (a),(b),(c) involved (the third one is from the composition of the first two). The octahedral axiom would give the fourth triangle.

$\mathscr{M}_2$ is a indeed a quiver moduli of the  dimension vector $[1694]$. Proposition $5.3$ in \cite{King_1994} implies that it is a fine moduli when $t\in \mathbb{Q}$. So $\mathscr{M}_2$ and $\mathscr{M}_3$ are both fine moduli spaces. Denote the universal family of representation on $\mathscr{M}_2$ by $U_2$. When restricting $U_2$ to $H$, there is a line bundle $L_1$ on $H$ such that $U_E\cong (U_2|H)\otimes \pi_H^*(L_1)$. To reduce the complexity of notations, we abuse the notation a bit by assuming that $L_1$ is trivial.

\begin{enumerate}
    \item This is from pulling back the extension in Prop \ref{extension1} from $H$ to $D$ and then pushforward to the blow-up $\mathbf{B}$. 
    
    $$\to i_*b_H^*(U_{\mathscr{F}_1}\otimes \pi_H^* L)\to i_*b_H^*(U_{E}) \xrightarrow{u} i_*b_H^*(p^*( U_{\mathscr{O}_{\Lambda}}))\to $$
    
    \item 
    $$\to b^*(U_2(-D\times \mathbb{P}^3))\to b^*(U_2) \xrightarrow{r} i_*(b^*U_2)_{D\times \mathbb{P}^3}\to $$

    \item Define $\mathscr{K}$ from the following distinguished triangle: ($\mathscr{K}$ will be the desired family on $\mathbf{B}$ for the gluing.)
    
    $$\to \mathscr{K} \to b^*U_2\xrightarrow{u\circ r} i_*b_H^*(p^*( U_{\mathscr{O}_{\Lambda}})) \to $$

    \item Apply the octahedral axiom, we have the following triangle:
    
    $$\to b^*(U_2(-D\times \mathbb{P}^3)) \to \mathscr{K} \to i_*b_H^*(U_{\mathscr{F}_1}\otimes \pi_H^* L)\to $$ 
    
    $\mathscr{K}$ is flat because it's a complex of vector bundles.

\end{enumerate}

\item Glue $\mathbf{B}$ to the component $\mathbf{P}$ using the family $\mathscr{K}$. We will apply the octahedral axiom again to triangles (a')$\sim$ (c') below. 

\begin{enumerate}[label=(\alph*')]

    \item  $$\to \mathscr{K}(-D\times \mathbb{P}^3)\to \mathscr{K}\xrightarrow{r} i_*Li^*(\mathscr{K})\to  $$
    
    \item Define a family $\mathscr{K}'$ from the following triangle:
    $$\to \mathscr{K}'\to Li^*(\mathscr{K})\to b_H^*(U_{\mathscr{F}_1}\otimes \pi_H^*L)\to  $$
    Then push it forward to $\mathbf{B}$ by $i_*$:
    $$\to i_*\mathscr{K}'\to i_*Li^*(\mathscr{K})\xrightarrow{v} i_*b_H^*(U_{\mathscr{F}_1}\otimes \pi_H^*L)\to  $$
    
    \item $$b^*U_2(-D\times \mathbb{P}^3) \to \mathscr{K} \xrightarrow{v\circ r}i_*b_H^*(U_{\mathscr{F}_1}\otimes \pi_H^*L)$$

\item Apply the octahedral axiom, and we have the triangle:

$$\to \mathscr{K}(-D\times \mathbb{P}^3)\to b^*U_2(-D\times \mathbb{P}^3) \to i_*\mathscr{K}' \to $$
    
\end{enumerate}

As desired, we have the following isomorphism, and this completes the proof that $\mathbf{B}$ is glued to $\mathbf{P}$ along the exceptional divisor algebraically.

\[
\begin{array}{ll}
    \mathscr{K}'&\cong  b_H^*p^*(U_{\mathscr{O}_{\Lambda}})\otimes \mathscr{O}_{D\times \mathbb{P}^3}(-D\times \mathbb{P}^3)\\ {}&\cong  b_H^*p^*(U_{\mathscr{O}_{\Lambda}})\otimes \pi_D^* \mathscr{O}_{\mathbb{P}(\mathscr{N}^*_{H|K_{(2,3)}})}(1)\\
    {}&\cong b_H^*p^*(U_{\mathscr{O}_{\Lambda}})\otimes \pi_{\mathbf{P}}^* \mathscr{O}_{\mathbf{P}}(1)\\ 
\end{array}
\]

\end{enumerate}

\section{An example of an actual wall built up from pieces}

We have shown the two walls for the class $v=(0,0,3,-5)$ in the $(t-u)-$plane in section $6$. These two walls, which are expected to be actual walls, do not intersect.
However, it is not always the case on a threefold (\cite{MR3597844} \cite{jardim2019walls}) as how they behave on a surface. We give a counter-example in this section. 

Let $C\subset \mathbb{P}^3$ be a rational quartic curve. The wall for $\mathscr{O}_C$ in the $(t-u)-$plane is expected to be the outermost parts of the following two numerical walls.

$$W_1: \quad 0\to \mathscr{O}_{\mathbb{P}^3}\to \mathscr{O}_{C} \to \mathscr{I}_{C}[1]\to 0$$

$$W_2: \quad 0\to \mathscr{O}_{Q}\to \mathscr{O}_{C} \to \mathscr{I}_{C/Q}[1] \to 0$$ 
where $Q\subset \mathbb{P}^3$ is a quadric surface $Q\subset \mathbb{P}^3$ containing $C$.

\begin{figure}[ht]
\begin{tikzpicture}[scale=1.2]
\begin{axis}[axis lines=middle ,xmin=-1.5,xmax=1,ymin=-1,ymax=1]
\draw[line width=1pt, color=red] (-1.321,0) to[out=90,in=180] (-0.494,0.858) to[out=0,in=140] (0.189,0.608) to[out=320, in=90] (0.483,0);
\draw[line width=1pt, color=red] (-1.321,0) to[out=270,in=180] (-0.494,-0.858) to[out=0,in=220] (0.189,-0.608) to[out=40, in=270] (0.483,0);
\draw[line width=1pt, color=violet] (0,1) to[out=300,in=120] (0.189,0.608) to[out=300 ,in= 115](0.35,0.28) to[out= 295,in=90] (0.414,0) to[out=270, in=65] (0.35,-0.28)  to[out=245, in=60] (0.189,-0.608) to [out=240, in=60] (0,-1);
\draw[line width=0.8pt, color=blue] (-0.25,0)circle(0.75);

\draw[line width=0.8pt, color=brown, rounded corners](0.186,1) to (0.189,0.608) to[out=275,in=90] (0.268,0) to[out=270,in=85] (0.189,-0.608) to (0.186,-1);

\filldraw [black] (0.189,0.608) circle (1.5pt); 
\node[color=black] at (0.27,0.7) {$S$};
\node[color=black] at (0.4,0.54) {$P$};
\node[color=black] at (0.58,0.1) {$T$};
\node[color=black] at (0.7,0.4) {$\mathscr{Q}_1$};
\node[color=black] at (-1.3,0.75) {$W_1$};
\node[color=black] at (-0.75,0.3) {$W_2$};
\node[color=black] at (0.1,0.9) {$W_3$};
\node[color=black] at (0.15,0.15) {$W_4$};
\draw[line width=1pt,color=green,dashed] (0,0) to[out=90,in=215] (0.5,0.7) to[out=325,in=90] (1,0);
\filldraw [black] (0.28,0.54) circle (1.5pt);
\filldraw [black] (0.5,0) circle (1.5pt);
\end{axis}
\end{tikzpicture}
\caption{}
\label{two pieces walls}
\end{figure}

In Figure \ref{two pieces walls}, $S$ is the intersection of $W_1$ and $W_2$, and the green region $\mathscr{Q}_1$ is the quiver region for $t\in [0,1]$. $T=0.5$ is the right endpoint of $W_2$, and $P$ is the intersection of $W_2$ with the boundary of $\mathscr{Q}_1$. There are two more walls, $W_3$ (purple) and $W_4$ (brown), that are defined by the short exact sequences $0\to \mathscr{O}_{\mathbb{P}^3} \to \mathscr{O}_{Q}\to \mathscr{O}_{\mathbb{P}^3}(-2)[1]\to 0$ and $0\to \mathscr{O}_{\mathbb{P}^3}(-2)[1] \to \mathscr{I}_{C}[1]\to \mathscr{I}_{C|Q}[1]\to 0$ respectivesly. A direct computation shows that $W_3$ and $W_4$ go through $S$ as well.

In fact, it's easy to see that $W_2$ is not an actual wall to the left of $S$ because $\mathscr{O}_Q$ is unstable, destabilized by $\mathscr{O}_{\mathbb{P}^3}$. Similarly, $W_1$ is not an actual wall to the right of $S$ because $\mathscr{I}_C[1]$ is unstable, destabilized by $\mathscr{O}_{\mathbb{P}^3}(-2)[1]$.

We prove the following property which implies that the actual wall of $\mathscr{O}_C$ is built up from more than one numerical wall. 

\begin{Prop}
The numerical wall $W_2$, defined by $0\to \mathscr{O}_{Q}\to \mathscr{O}_{C} \to \mathscr{I}_{C/Q}[1] \to 0$, is an actual wall in the quiver region $\mathscr{Q}_1$ and a pseudo wall to the left of $S$.
\end{Prop}

\begin{proof}
Recall that the stability condition in the $(t,u)-$plane is the pair $\sigma_{t,u}=(\mathscr{B}_t, z_{t,u}=-\chi_t+\frac{u^2}{2}\chi''_t + i \cdot \chi'_t)$. Let $\lambda_{t,u}$ be the slope function of $Z_{t,u}$. It's straightforward to check that $\lambda_{t,u}(\mathscr{O}_{\mathbb{P}^3})>\lambda_{t,u}(\mathscr{O}_{Q})$ for all $(t,u)\in W_2$ to the left of $S$ where $W_2$ is not an actual wall. 

Next, we prove that $\mathscr{O}_{Q}$ and $\mathscr{I}_{C/Q}[1]$ are stable in $\mathscr{Q}_1$.

\begin{enumerate}[label=(\alph*)]
\item $\mathscr{O}_{Q}$ and $\mathscr{I}_{C/Q}[1]$ are stable at $T$.

We use the Euler stability $\sigma_t=(\mathscr{A}_t,Z_t=\chi'_t+i\cdot \chi_t)$ with slope $\lambda^{Euler}_t=-\frac{\chi'_t}{\chi_t}$ at $T$. (It's the same with the stability condition $\sigma_{t,0}$ for $\mathscr{B}_t$.)

$\bullet$
For $\mathscr{O}_{Q}$, it has dimension vector $[1,4,7,4]$ in $\mathscr{A}_1$. Without loss of generality, assume $Q$ is defined by the equation $x_0x_3-x_1x_2=0$. Then the presentation of $\mathscr{O}_{Q}$ is given as follows,

$$\left[\mathscr{O}_{\mathbb{P}^3}(-4) \overset{M}{\to} \mathscr{O}_{\mathbb{P}^3}(-3)^4\overset{N}{\to} \mathscr{O}_{\mathbb{P}^3}(-2)^7 \overset{S}{\to} \mathscr{O}_{\mathbb{P}^3}^4(-1)\right]\cong \mathscr{O}_Q$$
where $N=\begin{pmatrix}x_2&-x_1&x_0&0&0&0&0\\ x_3&0&0&-x_1&x_0&0&0\\ 0&x_3&0&-x_2&0&x_0&0\\ 0&0&x_3&0&-x_2&x_1&0\\  \end{pmatrix}$, $S=\begin{pmatrix}-x_1& x_0&0&0\\ -x_2&0& x_0&0\\ 0&-x_2& x_1&0\\ -x_3&0&0& x_0\\ 0&-x_3&0& x_1\\ 0&0&-x_3& x_2\\ x_3&-x_2&0&0\\ \end{pmatrix}$,
and $M=\begin{pmatrix}-x_3& x_2&-x_1&x_0\end{pmatrix}$.

Stability of $\mathscr{O}_Q$ follows from checking the slopes of all its subcomplexes. We reduce the complexity in the following way. The inclusion $\mathscr{O}_{\mathbb{P}^3} \hookrightarrow \mathscr{O}_{Q} $ exists in both $\mathscr{B}_t$ ($t\in (-0.5774, 0.5774)$) and $\mathscr{A}_1$, in which the presentation of $\mathscr{O}_{\mathbb{P}^3}$ is its Koszul resolution (dim($\mathscr{O}_{\mathbb{P}^3}$)$=[1,4,6,4]$).

From the stability of $\mathscr{O}_{\mathbb{P}^3}$ in $\mathscr{A}_1$, we only need to check the subcomplexes of $\mathscr{O}_{Q}$ that are not subcomplexes of $\mathscr{O}_{\mathbb{P}^3}$. They are given as follows (in dimension vector): $[0143], [0154], [0164], [0174], [0264], [0274], [0374]$.
    
A direct computation shows that none of those can destabilize $\mathscr{O}_Q$ at $(0.5,0)$.

$\bullet$
For $\mathscr{I}_{C/Q}[1]$, its dimension vector in $\mathscr{A}_1$ is $[0,3,4,1]$. It's easy to check that a destabilizing subobject must have dimension vector $[0,a,b,0]$, where $a=0,...,3$ and $b=0,...,4$. In other words, if we prove that the dimension of a subcomplex must have a $"1"$ in the first position from the right, i.e. $[0,a,b,1]$, then $\mathscr{I}_{C/Q}[1]$ would be stable.

It is sufficient to show that there is no sub complex with dimension vector $[0,0,1,0]$. Suppose $[0,0,1,0]=\mathscr{O}_{\mathbb{P}^3}(-2)[1]$ is a sub complex, then there is a non-zero morphism $\mathscr{O}_{\mathbb{P}^3}(-2)[1] \to \mathscr{I}_{C/Q}[1]$. 
But we have a contradiction that 
$$Hom(\mathscr{O}_{\mathbb{P}^3}(-2)[1], \mathscr{I}_{C/Q}[1]) = Hom(\mathscr{O}_{\mathbb{P}^3}(-2), \mathscr{I}_{C/Q}) = H^0(\mathbb{P}^3, \mathscr{I}_{C/Q}(2))=0$$ This proves that $\mathscr{I}_{C/Q}[1]$ is stable in $\mathscr{A}_1$.

\item $\mathscr{O}_{Q}$ and $\mathscr{I}_{C/Q}[1]$ are stable along $\overset{\frown}{PT}$ in the quiver region. 

In the quiver region $\mathscr{Q}_1$, the stability condition $\sigma_{t,u}=(\mathscr{B}_t, Z_{t,u}=-\chi_t+\frac{u^2}{2}\chi''_t+ i \cdot \chi'_t)$ has essentially the same (up to a shift by $[1]$) slicings with the stability condition $\sigma^{Euler}_{t,u} := (\mathscr{A}_t, Z^{Euler}_{t,u}=\chi'_t+i \cdot(\chi_t-\frac{u^2}{2}\chi''_t))$. Let $\lambda^{Euler}_{t,u}:=-\frac{\chi'_t}{\chi_t-\frac{u^2}{2}\chi''_t}$ be the corresponding slope function. It's straightforward to check that the proof in step (a) holds if we replace $\lambda^{Euler}_t$ by $\lambda^{Euler}_{t,u}$ in $\mathscr{Q}_1$, and this proves the claim. 

\end{enumerate}

\end{proof}

\begin{Rem}
We expect that $0\to \mathscr{O}_{\mathbb{P}^3}\to \mathscr{O}_{Q} \to \mathscr{O}_{\mathbb{P}^3}(-2)[1]\to 0$ and $0\to\mathscr{O}_{\mathbb{P}^3}(-2)[1] \to \mathscr{I}_C[1] \to \mathscr{I}_{C|Q}[1]   \to 0$ are the unique wall for $\mathscr{O}_Q$ and $\mathscr{I}_{C|Q}$ in the $(t-u)$ plane, and the outermost parts of $W_1$ and $W_2$ build up the actual wall of $\mathscr{O}_{C}$.
\end{Rem}

\appendix

\section{}
\subsection{Presentation and subcomplexes of $\mathscr{I}_{P/\Lambda}(1)$}{\label{presentation and subcomplexes}}

The sequence below is the presentation of $\mathscr{I}_{P/\Lambda}(1)$ in $\mathscr{A}_1$. 

$$\left[\mathscr{O}^2_{\mathbb{P}^3}(-4)\xrightarrow{M}\mathscr{O}^8_{\mathbb{P}^3}(-3)\xrightarrow{N}\mathscr{O}^{11}_{\mathbb{P}^3}(-2)\xrightarrow{S}\mathscr{O}^5_{\mathbb{P}^3}(-1)\right]\xrightarrow{T}\mathscr{I}_{P/\Lambda}(1)$$
The matrices $M,N,S$ from the sequence and all the subcomplexes of $\mathscr{I}_{P/\Lambda}(1)$ are given as follows.

\textbf{Matrices}

$
N=\left(
\begin{array}{ccccccccccc}
z&-y&x&0&0&0&0&0&0&0&0\\
0&z&0&-y&x&0&0&0&0&0&0\\
w&0&0&0&0&-y&x&0&0&0&0\\
0&w&0&0&0&-z&0&0&x&0&0\\
0&0&0&0&0&z&0&-y&0&x&0\\
0&0&w&0&0&0&-z&0&y&0&0\\
0&0&0&w&0&0&0&-z&0&0&x\\
0&0&0&0&w&0&0&0&-z&-z&y\\
\end{array}\right)
$
$
S=
\begin{pmatrix}
x&0&0&0&0\\
0&x&0&0&0\\
-z&y&0&0&0\\
0&0&x&0&0\\
0&-z&y&0&0\\
0&0&0&x&0\\
-w&0&0&y&0\\
0&0&0&0&x\\
0&-w&0&z&0\\
0&0&0&-z&y\\
0&0&-w&0&z\\
\end{pmatrix}
$

$
T=
\begin{pmatrix}
y^2 \\ yz \\ z^2 \\ yw \\ zw\\
\end{pmatrix}
$
$
M=
\begin{pmatrix}
-w&0&z&-y&0&x&0&0\\
0&-w&0&z&z&0&-y&x\\
\end{pmatrix}
$

\textbf{Subcomplexes of $\mathscr{I}_{P/\Lambda}(1)$}

The dimension vectors of subcomplexes are given in the following table.

\begin{tabular}{|c|l|c|l|}
\hline
dimension vector & values of m,n & dimension vector& values of m,n\\
\hline
$[0,6,10,5]$& & $[0,7,11,5]$ &\\
\hline
$[0,8,11,5]$& & $[1,5,8,4]$ &\\
\hline
$[1,5,8,5]$&& $[1,6,10,5]$&\\
\hline
$[1,6,11,5]$&& $[1,7,11,5]$&\\
\hline
$[1,8,11,5]$&& $[0,4,8,4]$ &\\
\hline
$[0,4,8,5]$&& $[0,5,8,4]$&\\
\hline
$[0,0,0,n]$ & $n=1,...,5$ & $[0,1,n,5]$ & $n=9,10,11$\\
\hline
$[0,0,2,n]$ & $n=2,3,4,5$ & $[0,2,n,m]$ & $n=6,7,8$ $m=4,5$\\
\hline
$[0,0,3,n]$ & $n=2,3,4,5$ & $[0,2,n,5]$ & $n=9,10,11$\\
\hline
$[0,0,4,n]$ & $n=3,4,5$ & $[0,3,n,m]$ & $n=7,8$, $m=4,5$\\
\hline
$[0,0,5,n]$ & $n=3,4,5$ & $[0,3,n,5]$ & $n=9,10,11$\\
\hline
$[0,0,n,m]$ & $n=6,7,8$, $m=4,5$ & $[0,4,n,5]$ & $n=9,10,11$\\
\hline
$[0,0,n,5]$ & $n=9,10,11$ & $[0,5,n,5]$ & $n=8,9,10,11$\\
\hline
$[0,1,3,n]$ & $n=2,...,5$ & $[1,4,n,m]$ & $n=6,7,8$, $m=4,5$\\
\hline
$[0,1,4,n]$ & $n=3,4,5$ & $[1,5,n,5]$ & $n=9,10,11$\\
\hline
$[0,1,5,n]$ & $n=3,4,5$&$[0,2,5,n]$ & $n=3,4,5$\\
\hline
$[0,1,n,m]$ & $n=6,7,8$, $m=4,5$&&\\
\hline
\end{tabular}

\subsection{Auxiliary diagrams for Lemma \ref{commutative diagrams}}{\label{auxiliary diagrams}}

The following diagrams are commutative.

\begin{enumerate}
\item For $0\to \mathscr{O}_{\mathbb{P}^3} \to E \to \mathscr{F}[1] \to 0$

\begin{footnotesize}
\[
\begin{tikzcd}
0 \arrow{r}\arrow{d}&\begin{matrix}Hom(\mathscr{O}_{\mathbb{P}^3},\mathscr{O}_{\mathbb{P}^3})\\=\mathbb{C}\end{matrix}\arrow{r}\arrow{d}&\begin{matrix}Hom(\mathscr{O}_{\mathbb{P}^3},E)\\=\mathbb{C}\end{matrix}\arrow{r}\arrow{d}&\begin{matrix}Hom(\mathscr{O}_{\mathbb{P}^3},\mathscr{F}[1])\\=0\end{matrix}\arrow{r}\arrow{d}&\begin{matrix}Ext^1(\mathscr{O}_{\mathbb{P}^3},\mathscr{O}_{\mathbb{P}^3})\\=0\end{matrix}\arrow{d}\\
\begin{matrix}Hom(\mathscr{F}[1],\mathscr{F}[1])\\=\mathbb{C}\end{matrix}\arrow{r}{\cong}\arrow{d}&\begin{matrix}Ext^1(\mathscr{F}[1],\mathscr{O}_{\mathbb{P}^3})\\=\mathbb{C}^4\end{matrix}\arrow{r}\arrow{d}&\begin{matrix}Ext^1(\mathscr{F}[1],E)\\{}\end{matrix}\arrow{r}\arrow{d}&\begin{matrix}Ext^1(\mathscr{F}[1],\mathscr{F}[1])\\=\mathbb{C}^{12}\end{matrix}\arrow{r}{=0}\arrow{d}&\begin{matrix}Ext^2(\mathscr{F}[1],\mathscr{O}_{\mathbb{P}^3})\\=\mathbb{C}^{14}\end{matrix}\arrow{d}\\
\begin{matrix}Hom(E,\mathscr{F}[1])\\=\mathbb{C}\end{matrix}\arrow{r}\arrow{d}&\begin{matrix}Ext^1(E,\mathscr{O}_{\mathbb{P}^3})\\\mathbb{C}^3\end{matrix}\arrow{r}\arrow{d}&\begin{matrix}Ext^1(E,E)\arrow{r}\arrow{d}\\{}\end{matrix}&\begin{matrix}Ext^1(E,\mathscr{F}[1])\\=\mathbb{C}^{12}\end{matrix}\arrow{r}{=0}\arrow{d}&\begin{matrix}Ext^2(E,\mathscr{O}_{\mathbb{P}^3})\\=\mathbb{C}^{14}\end{matrix}\arrow{d}\\
\begin{matrix}Hom(\mathscr{O}_{\mathbb{P}^3},\mathscr{F}[1])\\=0\end{matrix}\arrow{r}\arrow{d}&\begin{matrix}Ext^1(\mathscr{O}_{\mathbb{P}^3},\mathscr{O}_{\mathbb{P}^3})\\=0\end{matrix}\arrow{r}\arrow{d}&\begin{matrix}Ext^1(\mathscr{O}_{\mathbb{P}^3},E)\\=0\end{matrix}\arrow{r}\arrow{d}&\begin{matrix}Ext^1(\mathscr{O}_{\mathbb{P}^3},\mathscr{F}[1])\\=0\end{matrix}\arrow{r}\arrow{d}&\begin{matrix}Ext^2(\mathscr{O}_{\mathbb{P}^3},\mathscr{O}_{\mathbb{P}^3})\\=0\end{matrix}\arrow{d}\\
\begin{matrix}Ext^1(\mathscr{F}[1],\mathscr{F}[1])\\=\mathbb{C}^{12}\end{matrix}\arrow{r}&\begin{matrix}Ext^2(\mathscr{F}[1],\mathscr{O}_{\mathbb{P}^3})\\=\mathbb{C}^{14}\end{matrix}\arrow{r}&\begin{matrix}Ext^2(\mathscr{F}[1],E)\\{}\end{matrix}\arrow{r}&\begin{matrix}Ext^2(\mathscr{F}[1],\mathscr{F}[1])\\=0\end{matrix}\arrow{r}&\begin{matrix}Ext^3(\mathscr{F}[1],\mathscr{O}_{\mathbb{P}^3})\\=0\end{matrix}\\
\end{tikzcd}
\]
\end{footnotesize}

\item For $0\to \mathscr{F}_1 \to E \to \mathscr{O}_{\Lambda} \to 0$

\begin{footnotesize}
\[
\begin{tikzcd}
0 \arrow{r}\arrow{d}&\begin{matrix}Hom(\mathscr{F}_1,\mathscr{F}_1)\\=\mathbb{C}\end{matrix}\arrow{r}\arrow{d}&\begin{matrix}Hom(\mathscr{F}_1,E)\\=\mathbb{C}\end{matrix}\arrow{r}\arrow{d}&\begin{matrix}Hom(\mathscr{F}_1,\mathscr{O}_{\Lambda})\\=0\end{matrix}\arrow{r}\arrow{d}&\begin{matrix}Ext^1(\mathscr{F}_1,\mathscr{F}_1)\\=\mathbb{C}^5\end{matrix}\arrow{d}\\
\begin{matrix}Hom(\mathscr{O}_{\Lambda},\mathscr{O}_{\Lambda})\\=\mathbb{C}\end{matrix}\arrow{r}\arrow{d}&\begin{matrix}Ext^1(\mathscr{O}_{\Lambda},\mathscr{F}_1)\\=\mathbb{C}\end{matrix}\arrow{r}\arrow{d}&\begin{matrix}Ext^1(\mathscr{O}_{\Lambda},E)\\=\mathbb{C}^{3}\end{matrix}\arrow{r}\arrow{d}&\begin{matrix}Ext^1(\mathscr{O}_{\Lambda},\mathscr{O}_{\Lambda})\\=\mathbb{C}^{3}\end{matrix}\arrow{r}{=0}\arrow{d}&\begin{matrix}Ext^2(\mathscr{O}_{\Lambda},\mathscr{F}_1)\\=0\end{matrix}\arrow{d}\\
\begin{matrix}Hom(E,\mathscr{O}_{\Lambda})\\=\mathbb{C}\end{matrix}\arrow{r}\arrow{d}&\begin{matrix}Ext^1(E,\mathscr{F}_1)\\=\mathbb{C}^5\end{matrix}\arrow{r}\arrow{d}&\begin{matrix}Ext^1(E,E)\arrow{r}\arrow{d}\\=\mathbb{C}^{15}\end{matrix}&\begin{matrix}Ext^1(E,\mathscr{O}_{\Lambda})\\=\mathbb{C}^{12}\end{matrix}\arrow{r}\arrow{d}&\begin{matrix}Ext^2(E,\mathscr{F}_1)\\=\mathbb{C}^2\end{matrix}\arrow{d}\\
\begin{matrix}Hom(\mathscr{F}_1,\mathscr{O}_{\Lambda})\\=0\end{matrix}\arrow{r}\arrow{d}&\begin{matrix}Ext^1(\mathscr{F}_1,\mathscr{F}_1)\\=\mathbb{C}^5\end{matrix}\arrow{r}\arrow{d}&\begin{matrix}Ext^1(\mathscr{F}_1,E)\\=\mathbb{C}^{12}\end{matrix}\arrow{r}\arrow{d}&\begin{matrix}Ext^1(\mathscr{F}_1,\mathscr{O}_{\Lambda})\\=\mathbb{C}^{9}\end{matrix}\arrow{r}\arrow{d}&\begin{matrix}Ext^2(\mathscr{F}_1,\mathscr{F}_1)\\=\mathbb{C}^2\end{matrix}\arrow{d}\\
\begin{matrix}Ext^1(\mathscr{O}_{\Lambda},\mathscr{O}_{\Lambda})\\=\mathbb{C}^{3}\end{matrix}\arrow{r}&\begin{matrix}Ext^2(\mathscr{O}_{\Lambda},\mathscr{F}_1)\\=0\end{matrix}\arrow{r}&\begin{matrix}Ext^2(\mathscr{O}_{\Lambda},E)\\=0\end{matrix}\arrow{r}&\begin{matrix}Ext^2(\mathscr{O}_{\Lambda},\mathscr{O}_{\Lambda})\\=0\end{matrix}\arrow{r}&\begin{matrix}Ext^3(\mathscr{O}_{\Lambda},\mathscr{F}_1)\\=0\end{matrix}\\
\end{tikzcd}
\]
\end{footnotesize}

\item For $0\to \mathscr{F}_1 \to E \to \mathscr{O}_{\Lambda'} \to 0$

\begin{footnotesize}
\[
\begin{tikzcd}
0 \arrow{r}\arrow{d}&\begin{matrix}Hom(\mathscr{F}_1,\mathscr{F}_1)\\=\mathbb{C}\end{matrix}\arrow{r}\arrow{d}&\begin{matrix}Hom(\mathscr{F}_1,E)\\=\mathbb{C}\end{matrix}\arrow{r}\arrow{d}&\begin{matrix}Hom(\mathscr{F}_1,\mathscr{O}_{\Lambda'})\\=0\end{matrix}\arrow{r}\arrow{d}&\begin{matrix}Ext^1(\mathscr{F}_1,\mathscr{F}_1)\\=\mathbb{C}^5\end{matrix}\arrow{d}\\
\begin{matrix}Hom(\mathscr{O}_{\Lambda'},\mathscr{O}_{\Lambda'})\\=\mathbb{C}\end{matrix}\arrow{r}\arrow{d}&\begin{matrix}Ext^1(\mathscr{O}_{\Lambda'},\mathscr{F}_1)\\=\mathbb{C}\end{matrix}\arrow{r}\arrow{d}&\begin{matrix}Ext^1(\mathscr{O}_{\Lambda'},E)\\=\mathbb{C}^{3}\end{matrix}\arrow{r}\arrow{d}&\begin{matrix}Ext^1(\mathscr{O}_{\Lambda'},\mathscr{O}_{\Lambda'})\\=\mathbb{C}^{3}\end{matrix}\arrow{r}{=0}\arrow{d}&\begin{matrix}Ext^2(\mathscr{O}_{\Lambda'},\mathscr{F}_1)\\=0\end{matrix}\arrow{d}\\
\begin{matrix}Hom(E,\mathscr{O}_{\Lambda'})\\=\mathbb{C}\end{matrix}\arrow{r}\arrow{d}&\begin{matrix}Ext^1(E,\mathscr{F}_1)\\=\mathbb{C}^5\end{matrix}\arrow{r}\arrow{d}&\begin{matrix}Ext^1(E,E)\arrow{r}\arrow{d}\\=\mathbb{C}^{8}\end{matrix}&\begin{matrix}Ext^1(E,\mathscr{O}_{\Lambda'})\\=\mathbb{C}^{3}\end{matrix}\arrow{r}{=0}\arrow{d}&\begin{matrix}Ext^2(E,\mathscr{F}_1)\\=0\end{matrix}\arrow{d}\\
\begin{matrix}Hom(\mathscr{F}_1,\mathscr{O}_{\Lambda'})\\=0\end{matrix}\arrow{r}\arrow{d}&\begin{matrix}Ext^1(\mathscr{F}_1,\mathscr{F}_1)\\=\mathbb{C}^5\end{matrix}\arrow{r}\arrow{d}&\begin{matrix}Ext^1(\mathscr{F}_1,E)\\=\mathbb{C}^{5}\end{matrix}\arrow{r}\arrow{d}&\begin{matrix}Ext^1(\mathscr{F}_1,\mathscr{O}_{\Lambda'})\\=0\end{matrix}\arrow{r}\arrow{d}&\begin{matrix}Ext^2(\mathscr{F}_1,\mathscr{F}_1)\\=0\end{matrix}\arrow{d}\\
\begin{matrix}Ext^1(\mathscr{O}_{\Lambda'},\mathscr{O}_{\Lambda'})\\=\mathbb{C}^{3}\end{matrix}\arrow{r}&\begin{matrix}Ext^2(\mathscr{O}_{\Lambda'},\mathscr{F}_1)\\=0\end{matrix}\arrow{r}&\begin{matrix}Ext^2(\mathscr{O}_{\Lambda'},E)\\=0\end{matrix}\arrow{r}&\begin{matrix}Ext^2(\mathscr{O}_{\Lambda'},\mathscr{O}_{\Lambda'})\\=0\end{matrix}\arrow{r}&\begin{matrix}Ext^3(\mathscr{O}_{\Lambda'},\mathscr{F}_1)\\=0\end{matrix}\\
\end{tikzcd}
\]
\end{footnotesize}

\item For $0\to \mathscr{O}_{\Lambda} \to E \to \mathscr{F}_1\to 0$. In the diagram, "singular" and "smooth" means the support of $E$ is singular or smooth.

\begin{footnotesize}
\[
\begin{tikzcd}
0 \arrow{r}\arrow{d}&\begin{matrix}Hom(\mathscr{O}_{\Lambda},\mathscr{O}_{\Lambda})\\=\mathbb{C}\end{matrix}\arrow{r}\arrow{d}&\begin{matrix}Hom(\mathscr{O}_{\Lambda},E)\\=\mathbb{C}\end{matrix}\arrow{r}\arrow{d}&\begin{matrix}Hom(\mathscr{O}_{\Lambda},\mathscr{F}_1)\\=0\end{matrix}\arrow{r}\arrow{d}&\begin{matrix}Ext^1(\mathscr{O}_{\Lambda},\mathscr{O}_{\Lambda})\\=\mathbb{C}^3\end{matrix}\arrow{d}\\
\begin{matrix}Hom(\mathscr{F}_1,\mathscr{F}_1)\\=\mathbb{C}\end{matrix}\arrow{r}\arrow{d}&\begin{matrix}Ext^1(\mathscr{F}_1,\mathscr{O}_{\Lambda})\\=\mathbb{C}^9\end{matrix}\arrow{r}\arrow{d}&\begin{matrix}Ext^1(\mathscr{F}_1,E)\\=\mathbb{C}^{3}\end{matrix}\arrow{r}\arrow{d}&\begin{matrix}Ext^1(\mathscr{F}_1,\mathscr{F}_1)\\=\mathbb{C}^{5}\end{matrix}\arrow{r}{=0}\arrow{d}&\begin{matrix}Ext^2(\mathscr{F}_1,\mathscr{O}_{\Lambda})\\=\mathbb{C}^{14}\end{matrix}\arrow{d}\\
\begin{matrix}Hom(E,\mathscr{F}_1)\\=\mathbb{C}\end{matrix}\arrow{r}\arrow{d}&\begin{matrix}Ext^1(E,\mathscr{O}_{\Lambda})\\=\mathbb{C}^8\end{matrix}\arrow{r}\arrow{d}&\begin{matrix}Ext^1(E,E)\\=\mathbb{C}^{13}smooth\\=\mathbb{C}^{14}singular\\\end{matrix}\arrow{r}\arrow{d}&\begin{matrix}Ext^1(E,\mathscr{F}_1)\\=\mathbb{C}^{5} smooth\\=\mathbb{C}^6 singular\end{matrix}\arrow{r}{}\arrow{d}&\begin{matrix}Ext^2(\mathscr{F}_1,\mathscr{O}_{\Lambda})\\=\mathbb{C}^{11}\end{matrix}\arrow{d}\\
\begin{matrix}Hom(\mathscr{O}_{\Lambda},\mathscr{F}_1)\\=0\end{matrix}\arrow{r}\arrow{d}&\begin{matrix}Ext^1(\mathscr{O}_{\Lambda},\mathscr{O}_{\Lambda})\\=\mathbb{C}^3\end{matrix}\arrow{r}\arrow{d}&\begin{matrix}Ext^1(\mathscr{O}_{\Lambda},E)\\=\mathbb{C}^{4}\end{matrix}\arrow{r}\arrow{d}&\begin{matrix}Ext^1(\mathscr{O}_{\Lambda},\mathscr{F}_1)\\=\mathbb{C}\end{matrix}\arrow{r}\arrow{d}&\begin{matrix}Ext^2(\mathscr{O}_{\Lambda},\mathscr{O}_{\Lambda})\\=0\end{matrix}\arrow{d}\\
\begin{matrix}Ext^1(\mathscr{F}_1,\mathscr{F}_1)\\=\mathbb{C}^{5}\end{matrix}\arrow{r}&\begin{matrix}Ext^2(\mathscr{F}_1,\mathscr{O}_{\Lambda})\\=\mathbb{C}^{14}\end{matrix}\arrow{r}&\begin{matrix}Ext^2(\mathscr{F}_1,E)\\=0\end{matrix}\arrow{r}&\begin{matrix}Ext^2(\mathscr{F}_1,\mathscr{F}_1)\\=\mathbb{C}^2\end{matrix}\arrow{r}&\begin{matrix}Ext^3(\mathscr{F}_1,\mathscr{O}_{\Lambda})\\=0\end{matrix}\\
\end{tikzcd}
\]
\end{footnotesize}

\end{enumerate}

\bibliographystyle{alpha}
\bibliography{ref.bib}

\end{document}